\documentclass[11pt,a4paper]{article}

\usepackage[english]{babel} 
\usepackage[T1]{fontenc} 
\usepackage{amsmath}
\usepackage{amssymb} 
\usepackage{ucs} 
\usepackage[utf8x]{inputenc}
\usepackage[mathcal]{eucal} 
\usepackage[]{graphicx} 
\usepackage{latexsym}
\usepackage{amsthm} 
\usepackage{stmaryrd}
\usepackage{verbatim}
\usepackage{pgf,tikz}
\usetikzlibrary{arrows}

\usepackage{float}
\usepackage{hyperref}
\usepackage{ulem}

\newtheorem{defi}{Definition}[section] 
\newtheorem{theo}[defi]{Theorem}
\newtheorem{theorem}[defi]{Theorem}
\newtheorem{coro}[defi]{Corollary} 
\newtheorem{lemme}[defi]{Lemma}
\newtheorem{lemma}[defi]{Lemma}
\newtheorem{prop}[defi]{Proposition}
\newtheorem{conj}[defi]{Question}

\newcommand{\fonction}[5]{\begin{array}{l|ccc}
{#1}  & {#2} & \longrightarrow & {#3} \\
        & {#4} & \longmapsto & {#5} 
    \end{array}
    }

\theoremstyle{remark}

\newcommand{\Ss}{\mathbb{S}}

\newcommand{\proj}{\mathbb{P}}

\newcommand{\R}{\mathbb{R}} 

\newcommand{\h}{\mathbb{H}}

\newcommand{\N}{\mathbb{N}} 

\newcommand{\e}{\varepsilon}
\newcommand{\p}{\varphi}

\newcommand{\Stab}{\mathrm{Stab}}
\newcommand{\PSL}{\mathrm{PSL}}
\newcommand{\SL}{\mathrm{SL}}

\newcommand{\AdS}{\mathrm{AdS}}
\newcommand{\ADS}{\mathbb{A}\mathrm{d}\mathbb{S}}

\newcommand{\PO}{\mathrm{PO}}
\newcommand{\Ach}{\mathrm{Arccosh}}

\newcommand{\Hdim}{\mathrm{Hdim}}
\newcommand{\tx}{\tilde{x}}
\newcommand{\ty}{\tilde{y}}
\newcommand{\tz}{\tilde{z}}
\newcommand{\txi}{\tilde{\xi}}
\newcommand{\teta}{\tilde{\eta}}
\newcommand{\ttau}{\tilde{\tau}}

\newcommand{\intff}[2]{\mathopen{[}#1\,,#2\mathclose{]}}

\renewcommand{\tilde}{\widetilde}

\newcommand{\pscal}[2]{\langle #1\, |\, #2 \rangle }
\newcommand{\pscalb}[2]{\langle #1\, |\, #2 \rangle_{p,q+1} }
\newcommand{\pscalc}[2]{\langle #1\, |\, #2 \rangle_{n,2} }
\newcommand{\tv}{\rightarrow}
\newcommand{\G}{\Gamma}
\newcommand{\g}{\gamma}

\newcommand{\St}{{\widetilde{\Sigma}}}

\DeclareMathOperator{\Card}{Card}

\DeclareMathOperator{\Id}{Id}

\DeclareMathOperator{\supp}{supp \,}

\DeclareMathOperator{\vol}{Vol}
\DeclareMathOperator{\Vol}{Vol}

\newcommand{\Hyp}{\mathbb{H}}

\title{ Critical exponent and Hausdorff dimension in pseudo-Riemannian hyperbolic geometry} 
\author{Olivier Glorieux\footnote{O. Glorieux is supported by CNPq grant 16088/2015-1} \, And Daniel Monclair \footnote{D. Monclair is supported by National Research Fund, Luxembourg}}

\begin{document}
\maketitle

\begin{abstract}
The aim of this article is to understand the geometry of limit sets in pseudo-Riemannian hyperbolic geometry. We focus on a class of subgroups of $\PO(p,q+1)$ introduced by Danciger, Guéritaud and Kassel, called $\h^{p,q}$-convex cocompact. We define a pseudo-Riemannian analogue of critical exponent and Hausdorff  dimension of the limit set. We show that they are equal and bounded from above by the usual Hausdorff dimension of the limit set. We also prove a rigidity result in $\h^{2,1}=\ADS^3$ which can be understood as a Lorentzian version of a famous Theorem of R. Bowen in  $3$-dimensional hyperbolic geometry. 
\end{abstract}
\setcounter{tocdepth}{1}
\tableofcontents

\section{Introduction}
\indent Limit sets of discrete groups of isometries of real hyperbolic space (and more generally rank one symmetric spaces) are a central theme in hyperbolic geometry. They are especially nice for convex cocompact groups, as they have nice dynamical properties as well as a fractal nature that is well understood. Indeed, it is known that the Hausdorff dimension of the limit set is equal to the critical exponent, which is a dynamical invariant of the action on the hyperbolic space.\\
\indent In higher rank, limit sets are much more complicated. The main reason for this is that the visual boundary of a higher rank symmetric space is not a homogeneous space, but a more involved object stratified by collections of orbits of the isometry group.\\
\indent Looking at convex cocompact subgroups of a higher rank simple Lie group $G$ (i.e. that act cocompactly on a convex subset of the Riemannian symmetric space of $G$) will not produce an interesting family of limit sets, as it was proved independently by J-F. Quint \cite{Quint2005} as well as B. Kleiner and B. Leeb \cite{KleinerLeeb} that these groups are uniform lattices.\\
\indent   Anosov groups  were introduced by Labourie \cite{labourie2006anosov} as subgroups of semi-simple Lie groups having dynamical properties that mimic  convex cocompact groups in rank one. One of their most important features is that they have a nice closed invariant set in a compact homogeneous space, that we can consider as a limit set.\\
 \indent However, the geometry of these limit sets is yet to be understood, and it can be quite different from rank one examples. For Hitchin representations in $\SL(n,\R)$ (which were the motivation for the definition of the Anosov property in \cite{labourie2006anosov}), one finds a limit set in $\R\proj^{n-1}$ which is a $\mathcal C^1$ circle. In particular, the Hausdorff dimension is not necessarily the most relevant invariant in order to understand the geometry of these limit sets.\\
 \indent The goal of this paper is to understand the geometry of these limit sets for a specific class of subgroups of $\PO(p,q+1)$ introduced by J. Danciger, F. Guéritaud and F. Kassel in \cite{DGK} that are related to pseudo-Riemannian hyperbolic geometry. More precisely, we will define  a geometric quantity associated to the limit set which we call the pseudo-Riemannian Hausdorff dimension (its generalizes the Hausdorff dimension), and show that it is equal to a notion of critical exponent in pseudo-Riemannian hyperbolic geometry.  

\subsection{$\h^{p,q}$-convex cocompact groups}
\indent Consider the pseudo-Euclidean space $\R^{p,q+1}$, which is $\R^{p+q+1}$ equipped with the standard quadratic form of signature $(p,q+1)$. Define $\h^{p,q}$ as the open subset of $\R\proj^{p+q}$ consisting of negative lines in $\R^{p,q+1}$. Its boundary $\partial\h^{p,q}\subset \R\proj^{p+q}$ is the set of isotropic lines in $\R^{p,q+1}$.\\
\indent A pseudo-Riemannian metric of signature $(p,q)$ can be defined by identifying the  tangent space $T_{[x]}\h^{p,q}$ with the orthogonal $x^\perp$ in $\R^{p,q+1}$. This provides a pseudo-Riemannian metric of constant negative sectional curvature. The isometry group is $\PO(p,q+1)$.\\
\indent Geodesics of $\h^{p,q}$ are intersections of $\h^{p,q}$ with projectivizations of $2$-planes in $\R^{p,q+1}$. They are very different depending on the signature of the $2$-plane. We will mostly work with spacelike geodesics, which correspond to planes of signature $(1,1)$.\\
\indent A subset $\Omega\subset\h^{p,q}$ is called properly convex if its closure in $\R\proj^{p,q}$ is convex (i.e. it is contained and convex in some affine chart). We follow the work of \cite{DGK} and say that a discrete group $\G\subset\PO(p,q+1)$ is $\h^{p,q}$-convex cocompact if it acts properly discontinuously and cocompactly on  a closed properly convex set $\Omega\subset \h^{p,q}$  with non empty interior such that the intersection of the closure $\overline\Omega\subset\R\proj^{p+q}$  with $\partial\h^{p,q}$ does not contain any non trivial line segment. \\

\indent There are many examples.
\begin{itemize}
\item When $q=0$, $\h^{p,0}$ is the Klein model of $\h^p$. The boundary $\partial \h^{p,0}$ contains no non trivial line segments, so $\h^{p,0}$-convex cocompactness is equivalent to the usual notion of convex cocompactness.
\item When $q=1$, $\h^{p,1}$ is a Lorentzian manifold called anti-de Sitter space, usually denoted by $\ADS^{p+1}$. Some of the most studied  $\ADS$-manifolds are called globally hyperbolic $\ADS$-spacetimes \cite{mess2007lorentz,barbot2007constant,barbot2008causal,merigot2012anosov} . One of the key points in the pioneering work of Mess \cite{mess2007lorentz} was the fact that their holonomies are $\h^{p,1}$-convex cocompact. 
\item When $m$ is even (resp. odd), Hitchin representations in $\PO(m+1,m)$  are $\h^{m,m}$-convex cocompact (resp.$\h^{m+1,m-1}$-convex cocompact) \cite{DGK}.
\item The maximal representations of surface groups in $\PO(2,q+1)$ studied by B. Collier, N. Tholozan and J. Toulisse   in \cite{CTT} are $\h^{2,q}$-convex cocompact (this is a consequence of the work in \cite{BILW} and \cite{DGK2}).  
\end{itemize}

If $\G\subset\PO(p,q+1)$ is $\h^{p,q}$-convex cocompact, we define its limit set $\Lambda_\G$ as the intersection with $\partial\h^{p,q}$ of the closure of the orbit of any point in a convex set $\Omega$ introduced in the definition of $\h^{p,q}$-convex cocompactness. It does not depend on a choice of  such a  convex set. \\
\indent There is no canonical choice of this convex set $\Omega$, however they all contain the convex hull $C(\Lambda_\G)$ of $\Lambda_\G$ (but it can have empty interior).\\

It is shown in \cite{DGK2} that $\h^{p,q}$-convex cocompact groups are Anosov. The most important consequence is that if $\G\subset\PO(p,q+1)$ is $\h^{p,q}$-convex cocompact, then $\G$ is Gromov-hyperbolic and the action on $\Lambda_\G$ is conjugate to the action on the Gromov boundary $\partial_\infty\G$.

\subsection{Critical exponent in $\h^{p,q}$}
\indent We recall the classic definition of the critical exponent in metric spaces. Let $G$ be a countable group acting on a metric space $(X,d)$, and $o\in X$. The \textbf{critical exponent } $\delta(G, X)$ is 
$$\delta(G, X) := \limsup_{R\tv \infty} \frac{1}{R} \log \Card \{ g\in G \, |\, d(g o , o )\leq R\}.$$
\indent A simple computation based on the triangle inequality shows that this number does not depend on  $o\in X$. It measures the exponential growth rate of the orbits of $G$ in $X$.\\
 \indent For example, by a simple argument of volume, we can see that the critical exponent of a uniform lattice of $\PO(p,1)$ acting on $\Hyp^{p}$ is equal to $p-1$.  More generally this applies to fundamental group of compact Riemannian manifolds of negative curvature, where the critical exponent is equal to the exponential growth rate of the volume of  balls. For a more thorough treatment we refer to the text of M. Peigné \cite{peigne2013autour} and F. Paulin \cite{paulin1997critical} \\
%In the same idea, when we consider quasi-Fuchsian deformations of $\G\subset \SO_\circ(n,1)$ into $\SO_\circ(n+1,1)$, the critical exponent of $\G$ acting on $\Hyp^{n+1}$ will be equal to the exponential growth rate of the volume of balls intersected with the convex core. \\
\indent A famous theorem of R. Bowen \cite{bowen1979hausdorff} in dimension $3$ and Yue \cite{Yue} in higher dimension shows that the critical exponent of a quasi-Fuchsian group in $\PO(p+1,1)$ is greater than $n-1$ with equality if and only if the group is Fuchsian, that is conjugate to a subgroup of $\mathrm{O}(p,1)$. \\

 The main problem when it comes  to defining this invariant in pseudo-Riemannian hyperbolic geometry is that  $\h^{p,q}$ is not a metric space: if $q>0$, there are no $\PO(p,q+1) $ invariant distances on $\h^{p,q}$. \\
\indent The starting point of our work is the search for a  good replacement for the distance on the convex hull $C(\Lambda_\G)$ of a $\h^{p,q}$-convex cocompact group $\Gamma\subset\PO(p,q+1)$, which will lead to an $\h^{p,q}$ critical exponent. This will be done in Section \ref{sec - triangle inequality}. \\
\indent Its definition is simple: if two points are on the same spacelike geodesic then their $\h^{p,q}$-distance  is defined to be the length on this geodesic (since it is spacelike, the induced metric on this geodesic is Riemannian), in other configurations their $\h^{p,q}$-distance is defined to be $0$.  \\
\indent We call $d_{\h^{p,q}}(\cdot, \cdot)$ this function on $C(\Lambda_\G) \times C(\Lambda_\G)$. \\
\indent This function is not a distance, and the first part of our work consists in finding a weak form of the triangle inequality when looking at the convex hull $C(\Lambda_\G)$. 

\begin{theorem} \label{th - intro inegalite triangulaire} If $\G\subset \PO(p,q+1)$ is $\h^{p,q}$-convex cocompact, there is a constant $k_\G>0$ such that $d_{\h^{p,q}}(x,y)\le d_{\h^{p,q}}(x,z)+d_{\h^{p,q}}(z,y)+k_\G$ for all $x,y,z\in C(\Lambda_\G)$.
\end{theorem}

 This allow us to define the \textbf{critical exponent for $\h^{p,q}$-convex cocompact groups} by 
$$\delta_{\h^{p,q}}(\G):= \limsup_{R\tv \infty} \frac{1}{R} \log \Card \{ \g\in \G \, |\, d_{\h^{p,q}}(\g o , o )\leq R\}.$$
Thanks to \ref{th - intro inegalite triangulaire}, it does not depend on the choice of a point $o\in C(\Lambda_\G)$.

Let us mention that F. Kassel and T. Kobayashi \cite{Kassel2016} studied critical exponents and the associated Poincaré series for Clifford-Klein forms, including closed anti-de Sitter manifolds.

\subsection{Pseudo-Riemannian Hausdorff dimension}

The other invariant  that we want to generalize is the Hausdorff dimension of the limit set. In the hyperbolic case, it is known since the work of S. J. Patterson and D. Sullivan \cite{patterson1976limit,sullivan1979density} that the Hausdorff dimension of the limit set of a quasi-Fuchsian group  is equal to the critical exponent. It provides a link between the action of the group inside the hyperbolic space and the fractal geometry of the limit set. \\
\indent The Hausdorff dimension is not the right invariant  in our case: if $q>0$, we will see that the limit set $\Lambda_\G$ of any $\h^{p,q}$-quasi Fuchsian group   has  Hausdorff dimension $p-1$. Using the pseudo-Riemannian geometry of the boundary we define a new notion called the \textbf{pseudo-Riemannian Hausdorff dimension}, that we will denote  by $\Hdim_{p,q}$ ($\Hdim_{p,0}$ being the usual notion of Hausdorff dimension). Roughly speaking, this dimension measures the number of pseudo-Riemannian balls (which are the interiors of quadrics) necessary to cover a set, in the same way that the classical Hausdorff dimension measures the number of metric balls necessary to cover the set.\\
\indent We obtain our first main result 
\begin{theorem}\label{th - Intro exposnt critique = Hdim}
If $\G\subset \PO(p,q+1)$ is $\h^{p,q}$-convex cocompact, then 
$$\delta_{\h^{p,q}}(\G)  = \Hdim_{p,q} (\Lambda_\G)\leq \Hdim(\Lambda_\G)\leq p-1.$$
\end{theorem}

\subsection{Patterson-Sullivan densities}
The main tool in the proof of Theorem \ref{th - Intro exposnt critique = Hdim} is the construction of a conformal density, which is a family of measures supported on the limit set indexed by points of the convex hull.\\
\indent The definition of conformal densities relies on the notion of Buseman functions, defined by $\beta_\xi(x,y)=\lim_{p\to \xi}d_{\h^{p,q}}(x,p)-d_{\h^{p,q}}(y,p)$ where $x,y\in\h^{p,q}$ and $\xi \in \partial \h^{p,q}$. Buseman functions will be studied in Subsection \ref{subsec - limit set}.

\begin{defi}
A conformal density of dimension $s\in\R$ is a family of measures $(\nu_x)_{x\in C(\Lambda_\G)}$  satisfying the following conditions:
\begin{enumerate}
\item $\forall \g\in \G$, $\g^*\nu_x = \nu_{\g x}$ (where $  \g^*\nu (E) = \nu(\g^{-1} E))$
\item $\forall x,y\in C(\Lambda_\G)$, $\frac{d \nu_x}{d\nu_{y}} (\xi) = e^{-s \beta_\xi (x,y)} $
\item $\supp(\nu_x)=\Lambda_\G$
\end{enumerate}
\end{defi}

An adaptation of the classical construction due to S.J. Patterson and D. Sullivan \cite{patterson1976limit,sullivan1979density}  in the hyperbolic case provides a conformal density of dimension $\delta_{\h^{p,q}}(\G)$. For a nice introduction of this theory we refer to the lecture notes of J.-F. Quint \cite{quint2006overview}. \\
\indent One of the important steps in order to identify $\delta_{\h^{p,q}}(\G)$ with the pseudo-Riemannian Hausdorff dimension of $\Lambda_\G$  consists in showing that the measure of a pseudo-Riemannian ball of radius $r$ in the boundary  $\partial\h^{p,q}$ behaves like $r^{\delta_{\h^{p,q}}}$ as $r\tv 0$. This is a consequence of a result known as the Shadow Lemma  (Theorem \ref{shadow_lemma}), proved in the hyperbolic case by Sullivan \cite{sullivan1979density}.
\begin{theorem} \label{th - Intro measure_balls}
Let $\mu$ be the Patterson-Sullivan density, and let $x\in C(\Lambda_\G)$. There is $c>0$ such that for all $\xi\in \Lambda_\G$, $r\in (0,1)$, we have: \[\frac{\mu_x(B_x(\xi,r))}{r^{\delta_\Gamma}} \in \left[\frac{1}{c},c\right],\]
where $B_x(\xi,r)$ is the pseudo-Riemannian ball on the boundary. 
\end{theorem}

Thanks to the Shadow Lemma we also prove that  Patterson-Sullivan measures are ergodic. From the Patterson-Sullivan density we also construct  a measure on the non-wandering set of the unit tangent for the geodesic flow a finite, invariant, ergodic measure: the Bowen-Margulis measure. 

\subsection{Rigidity in $\h^{2,1}$}
The second main result of this paper is the rigidity result that we obtain in $\h^{2,1}=\ADS^3$. \\
A group $\G\subset \PO(p,q+1)$ is called $\h^{p,q}$-quasi Fuchsian if it is $\h^{p,q}$-convex cocompact and $\Lambda_\G$ is homeomorphic to $\Ss^{p-1}$. It is called Fuchsian if it is conjugate to a uniform lattice in $\mathrm{O}(p,1)\subset\PO(p,q+1)$.

\begin{theorem}\label{th - Intro rigidity }
Let $\G\subset \PO(2,2)$ be a $\h^{2,1}$-quasi Fuchsian group. Then
$$\delta_{\h^{2,1}} (\G) \leq 1,$$
with equality if and only if $\G$ is Fuchsian.
\end{theorem}

It is the Lorentzian analogue of R. Bowen's Theorem \cite{bowen1979hausdorff}. The proof of the inequality mimics a classic method using a comparison of the volume of large balls, that one can found in \cite{knieper1995volume,glorieux2015entropy}. The main argument is a  comparison of two distances on the boundary of the convex core. These two distances are the intrinsic and extrinsic distances coming from the $\AdS$ distance. We will see that the usual inequality between intrinsic and extrinsic distances in Riemannian geometry is reversed in the Lorentzian context.%We prove a reverse inequality compared to the Riemannian case: the extrinsic distance is greater than the intrinsic distance (up to a fixed additive constant). This inequality on distances has to be proven only on 2-dimensional anti de Sitter space, where it follows by a simple computation. 

The characterization of the equality case consists in comparing  the two Patterson-Sullivan measures coming from the intrinsic and extrinsic distances. We show that if there is equality between the critical exponent of the intrinsic and extrinsic distances then the measures should be equivalent. By a classical remark using Bowen-Margulis measures, we show that the marked length spectra of the boundary of the convex core and of the ambient, quasi-Fuchsian $\ADS^3$-manifold are equal. This ends the proof thanks to a easy argument of  two-dimensional hyperbolic geometry.

\emph{Remark:} The Mess parametrization \cite{mess2007lorentz} associates a pair of Fuchsian representations in $\PSL(2,\R)$ to every quasi-Fuchsian group in $\PO(2,2)$.  Using this parametrization, the work in \cite{glorieux2015asymptotic}, and the right notion of entropy, one can show that $\delta_{\h^{2,1}}(\G)$ is equal to a number associated with the action of a surface group  on $\Hyp^2\times \Hyp^2$  in  \cite{bishop1991three} by Bishop and Steger, where they prove the same rigid inequality (another proof can be found in \cite{potrie2014}). The proof we propose is totally independent, and we actually prove a stronger result (Theorem \ref{th - inequality h<I delta}  ).

\subsection{Questions}

Theorem \ref{th - Intro exposnt critique = Hdim} gives an inequality between the pseudo-Riemannian Hausdorff dimension and the usual Hausdorff dimension. A natural question is to understand when the equality is achieved. It is the case when $\G$ is conjugate to a convex cocompact subgroup of $\mathrm O(p,1)$, and more generally when $\G$ preserves a totally geodesic copy of $\h^p$.

\begin{conj} Let $\G\subset \PO(p,q+1)$ be $\h^{p,q}$-convex cocompact. If $\delta_{\h^{p,q}}(\G)= \Hdim(\Lambda)$, is  $\G$ conjugate to a subgroup of $\mathrm{P}(\mathrm O(p,1)\times \mathrm O(q))$ ?
\end{conj}

Note that even in the quasi-Fuchsian case, i.e. when $\Lambda_\G$ is homeomorphic to $\Ss^{p-1}$, is far from obvious. In this case, an intermediate situation is the case where $\Lambda_\G$ is smooth.

\begin{conj} Let $\G\subset \PO(p,q+1)$ be $\h^{p,q}$-quasi Fuchsian. If $\Lambda_\G$ is a $C^1$-submanifold of $\partial\h^{p,q}$, is  $\G$ Fuchsian ?
\end{conj}

  We expect that these two question have positive answers. This will be the subject of  following work.\\

Another interesting direction is the world of Anosov representations.  The entropy of an Anosov representation can be seen as a generalization of the critical exponent. In higher rank Lie groups, there are several notions of entropy for these representations (one of them being the critical exponent for the action on the Riemannian symmetric space), some of them leading to inequalities with rigidity in the equality case \cite{potrie2014}. It would be interesting to find other settings in which  the entropy coincides with a generalization of the Hausdorff dimension of the limit curve, thus giving a geometric interpretation of a dynamical invariant.\\

%Finally, since the rigidity result in Theorem \ref{th - Intro rigidity } is only stated for $\SO_\circ(2,2)$, one can ask whether it also holds in higher dimensions. This  will be the subject of a following paper.

\subsection{Plan of the paper}
We start by covering the necessary background in pseudo-Riemannian hyperbolic geometry. We emphasize on the differences and similarities with the usual Riemannian case. 

It is followed by  a collection of geometric facts about asymptotic geometry for the distance $d_{\h^{p,q}}$. We review the usual concepts of shadows,  radial convergence, Buseman functions, and Gromov distances in our setting. 

Section \ref{sec - conformal density} is devoted to the  study of conformal densities. For the existence and ergodicity of conformal densities, which are very close to classical results in hyperbolic geometry, we refer to the appendix.  %Depending on the original use of the triangle inequality, some of the proofs are almost a "cut and paste" (up to additive constant coming from our triangle inequality) of the proofs for the similar results in the hyperbolic case. Other are more technical and uses properties of $\AdS$ geometry. 

In Section \ref{sec - Lorentzian hausdorff dim}, we define the pseudo Riemannian Hausdorff dimension and prove Theorem \ref{th - Intro exposnt critique = Hdim}. 

Finally, the last section is devoted to the proof of Theorem \ref{th - Intro rigidity }.

\subsection*{Acknowledgements} We would like to thank Jean-Marc Schlenker and Nicolas Tholozan for many encouraging and enlightening conversations on this topic. We  also thank the anonymous referees for their many valuable comments.

\section{Pseudo-Riemannian hyperbolic geometry} \label{sec - AdS geometry}
In this section we introduce the pseudo-Riemannian hyperbolic space $\h^{p,q}$ and {\color{black}go over  its basic} properties. A nice introduction to anti-de Sitter geometry can be found in \cite{barbot2007constant} for $\ADS^3$, and in \cite{merigot2012anosov} for arbitrary dimension.
\subsection{Models for $\h^{p,q}$ and its boundary}
%\sout{As stated in the introduction, the anti de Sitter space $\AdS^{n+1}$ is $q_{p,q+1}^{-1}(\{-1\})$, where $q_{p,q+1}=-du^2-dv^2+dx_1^2+\cdots+dx_n^2$ is the standard $(p,q+1)$ signature quadratic form on $\R^{n+2}$, endowed with the restrictions of $q$ to tangent spaces. \\
%We denote by $\pscal{\cdot}{\cdot}$ the corresponding bilinear form. } \\
{\color{black}\indent Define $\R^{p,q+1}$ as the space $\R^{p+q+1}$ endowed with the symmetric bilinear form \[\langle u\vert v\rangle_{p,q+1}=u_1v_1+\cdots+u_pv_p-u_{p+1}v_{p+1}-\cdots-u_{p+q+1}v_{p+q+1}.\] 
\indent We will consider the Klein model of the pseudo-Riemannian hyperbolic space \[ \h^{p,q}=\{ [u]\in \R\proj^{p+q} \vert \langle u\vert u\rangle_{p+q+1}<0\}.\] 
\indent There is a natural pseudo-Riemannian metric of signature $(p,q)$ defined on $\h^{p,q}$ defined by identifying $T_{[u]}\h^{p,q}$ with $u^{\perp}\subset \R^{p,q+1}$, and considering the restriction of $\langle\cdot\vert\cdot\rangle_{p,q+1}$ to $u^{\perp}$ (where orthogonals are considered with respect to $\langle\cdot\vert\cdot\rangle_{p,q+1}$).\\
\indent There is also a linear model for the pseudo-Riemannian hyperbolic space
\[\mathcal H^{p,q} = \{ u\in \R^{p,q+1} \vert \langle u\vert u\rangle_{p,q+1}=-1\}.\]
The linear model $\mathcal H^{p,q}$ is a double cover of the projective model $\h^{p,q}$. Note that $\ADS^{n+1}=\h^{n,1}$ and $\h^{n,0}=\h^n$.  }\\

\indent {\color{black}The boundary $\partial \h^{p,q}$ is simply defined as the boundary in $\R\proj^{p+q}$, which is \[ \partial \h^{p,q}=\{[u]\in \R\proj^{p+q} \vert \langle u\vert u\rangle_{p,q+1}=0\}.\] }
%\indent This is a conformal Lorentzian manifold, someticalled the Einstein Universe $\Ein^n$. 

\paragraph{The $\h^{2,1}$ case}
When $(p,q)=(2,1)$ there is another convenient model, as {\color{black}$\h^{2,1}=\ADS^3$ is isometric to $\PSL(2,\R)$ endowed with a bi-invariant metric (a multiple of its Killing form).} Indeed we can see $\R^{2,2}$ as $M(2,\R)$  endowed by the quadratic form $q=-\det$. Then we can see $\SL(2,\R)$ as the level $\{q=-1\}$ endowed with the restriction of $q$ to tangent spaces. {\color{black} The following map induces an isometry between $\ADS^3$ and $\PSL(2,\R)$:} 
$$\fonction{}{\R^{2,2}}{M(2,\R)}{(u_1,u_2,u_3,u_4)}{\left(\begin{array}{cc}
      u_1-u_3 &  -u_2+u_4 \\
      u_2+u_4 & u_1+u_3
   \end{array}
   \right).}$$

{\color{black}\paragraph{Notations and inner product for points of $\overline{\h^{p,q}}$}   
   Given a point $x\in \h^{p,q}\subset \R\proj^{p+q}$, which is a line in $\R^{p,q+1}$, we will add a tilde to denote a lift $\tx\in x\subset\R^{p+q+1}$ satisfying $\pscalb{\tx}{\tx}=-1$, provided the expression in which it is used does not depend on such a lift (as there are two choices).\\
   \indent With this convention, for $x,y\in \h^{p,q}$, we can define the number $\vert\pscalb{\tx}{\ty}\vert$ (but not $\pscalb{\tx}{\ty}$).\\
  \indent Similarly, given $\xi\in \partial \h^{p,q}$ we will add a tilde to denote a choice of a lift $\txi\in\xi\subset\R^{p+q+1}$, provided once again that the expression in which it is used does not depend on such a lift.\\
  \indent For example, given $(x,\xi)\in\h^{p,q}\times\partial\h^{p,q}$, the expression $\pscalb{\tx}{\txi}\neq 0$ is well defined, even though the number $\pscalb{\tx}{\txi}$ is not.\\
      
\subsection{Isometries of $\h^{p,q}$}
The  group of orientation preserving isometries of  $\h^{p,q}$ is the group $\mathrm{PO}(p,q+1)$ of projective transformations of $\R\proj^{p+q}$ whose lifts to $\R^{p+q+1}$ preserve the {\color{black} bilinear form $\langle\cdot\vert\cdot\rangle_{p,q+1}$.}  It acts transitively on $\h^{p,q}$.\\
\indent The stabilizer of a point $x\in\h^{p,q}$ in $\PO(p,q+1)$ is isomorphic to $\mathrm O(p,q)$. For {\color{black}$x_0=[0:\cdots:0:1]$}, the associated inclusion $\mathrm O(p,q)\subset\PO(p,q+1)$ corresponds to the standard inclusion by block-diagonal matrices, so $\h^{p,q}$ can be seen as the homogeneous space $\PO(p,q+1)/\mathrm O(p,q)$.\\

\subsection{Geodesics}

{\color{black}
\begin{defi}
Geodesics of $\h^{p,q}$ are intersections of $\h^{p,q}$ with projectivizations $\proj(V)\subset \R\proj^{p+q}$ of $2$-dimensional planes $V\subset \R^{p+q+1}$ such that $\proj(V)\cap\h^{p,q}\neq\emptyset$. 
\end{defi}

The condition $\proj(V)\cap\ADS^{n+1}\neq\emptyset$ is equivalent to the fact that the restriction of $\pscalb{\cdot}{      \cdot}$ to $V$ has signature $(0,2)$, $(1,1)$ or $(0,1)$.\\}
This definition is equivalent to the classical notion of geodesics in pseudo-Riemannian geometry, however this is the characterization that will be the most useful for us.\\

\indent {\color{black}Given two distinct points $x,y\in \h^{p,q}$, there is a unique geodesic of $\h^{p,q}$ joining $x$ and $y$, which we will note $(xy)$ (it can be defined as $\proj(x\oplus y)\cap\h^{p,q}$).\\}
\indent As in any pseudo-Riemannian manifold, geodesics of $\h^{p,q}$ (for $q\geq 1$) are classified in three different types: spacelike geodesics (for which tangent vectors are positive for $\langle\cdot\vert\cdot\rangle_{p,q+1}$), timelike geodesics (negative tangent vectors) and lightlike geodesics (null tangent vectors).\\
\indent {\color{black} The type of the geodesic $(xy)$ depends only on the inner product. It is (recall that $\tx$ denotes a lift of $x$ to $\R^{p+q+1}$ satisfying $\pscalb{\tx}{\tx}=-1$, and $\ty$ is defined in a similar way)}:
\begin{itemize}
\item Spacelike if and only if $\vert\pscalb{\tx}{\ty}\vert>1$,
\item Lightlike if and only if $\vert\pscalb{\tx}{\ty}\vert=1$,
\item Timelike if and only if $\vert\pscalb{\tx}{\ty}\vert<1$.
\end{itemize}
{\color{black}These conditions can also be detected by looking at the signature of the restriction of $\pscalb{\cdot}{\cdot}$ to $\proj(x\oplus y)$, spacelike (resp. lightlike, timelike) being equivalent to having signature  $(1,1)$ (resp. $(0,1)$, $(0,2)$).}
 }
%\begin{center}\caption{Geodesics of $\AdS^3$ in an affine chart}
%\includegraphics[scale=0.5]{geo}
%\end{center}

\paragraph{The $\h^{2,1}$ case}
In the $\PSL(2,\R)$ model, the geodesics passing through $\Id$ are precisely the  $1$-parameter subgroups. {\color{black} The classification of $1$-parameter subgroups of $\PSL(2,\R)$ into hyperbolic, parabolic and elliptic groups corresponds to the classification of Lorentzian geodesics into spacelike, lightlike and timelike geodesics.  
\begin{itemize}
\item Spacelike geodesics are conjugate to $  \left[
   \begin{array}{cc}
      e^{t}  &   0 \\
      0  & e^{-t}
   \end{array}
   \right]$
   
   \item Lightlike geodesics are conjugate to $  \left[
   \begin{array}{cc}
      1  &   x \\
      0  & 1
   \end{array}
   \right]$
   
   \item Timelike geodesics are conjugate to $  \left[
   \begin{array}{cc}
      \cos(\theta)  &   \sin(\theta) \\
      -\sin(\theta)  & \cos(\theta)
   \end{array}
   \right]$
\end{itemize}

All geodesics are left translates of geodesics passing through $\Id$. }

\paragraph{Geodesics and $\partial \h^{p,q}$.} {\color{black}Not all geodesics of $\h^{p,q}$ have endpoints on $\partial \h^{p,q}$,  which is a major difference with Riemannian hyperbolic geometry.}\\
However, the situation is nicer if we restrict ourselves to spacelike geodesics. Indeed, timelike geodesics are closed, so they never meet the boundary, and lightlike geodesics meet the boundary at exactly one point.\\

\begin{lemme} \label{lem - equations geodesiques} Given $x\in \h^{p,q}$ and $\xi\in\partial \h^{p,q}$, there is a unique geodesic $(x\xi)$ of $\h^{p,q}$  passing though $x$ with endpoint $\xi$. It is spacelike if and only if $\pscalb{\tx}{\txi}\neq 0$. In this case, it can be parametrized  as $f(s)$ where: 
\begin{align*}\tilde{f(s)}  &= \cosh(s) \tx -\sinh(s) \left( \frac{\txi}{\pscalb{\tx}{\txi} } +\tx\right) \\
&= e^{-s} \tx - \frac{\sinh s}{\pscalb{\tx}{\txi}}\txi.
\end{align*}
\end{lemme}
\begin{proof} Since $\proj(x\oplus\xi)$ is the only projective line containing $x$ and $\xi$, we have the existence and uniqueness of the geodesic. Since $\pscalb{\tx}{\tx}=-1$ and $\pscalb{\txi}{\txi}=0$, the signature of the restriction of $\pscalb{\cdot}{\cdot}$ to $x\oplus\xi$ is $(1,1)$ if and only if  $\pscalb{\tx}{\txi}\neq 0$.\\
Since the formula for $f(s)$ is a unique speed parametrization of $\proj(x\oplus\xi)\cap\ADS^{n+1}$, it is a parametrization of the geodesic $(x\xi)$.
\end{proof}

\indent We will denote by $[x\xi)$ the half geodesic going from $x$ to $\xi$.\\

{\color{black} \indent Pairs of points in $\partial\h^{p,q}$ do not always define a geodesic of $\h^{p,q}$. Given $\xi,\eta\in\partial\h^{p,q}$, there is a spacelike geodesic of $\h^{p,q}$ with endpoints $\xi$ and $\eta$ if and only if $\pscalb{\txi}{\teta}\neq 0$. \\ }

\subsection{The geometry of $\partial \h^{p,q}$}  \label{subsec - geometry of the boundary}

\begin{defi}
Let $x\in \ADS^{n+1}$. Its dual hyperplane is the set 
$$x^* := \{ y \in  \label{subsec - geometry of the boundary} \, | \, \pscalb{\tx}{\ty} =0 \}.$$
\end{defi}

%\paragraph{Example n=2}
%$Id^* = \{ M \in SL_2(\R) \, |\, \Tr(M) =0\}.$

The dual hyperplane $x^*$ is a totally geodesic embedded copy of $\h^{p,q-1}$ in $\h^{p,q}$. Conversely, any totally geodesic embedded copy of $\h^{p,q-1}$ in $\h^{p,q}$ is equal to $x^*$ for a unique point $x\in\h^{p,q}$.\\
Note that if $q=0$, then $x^*$ is empty.\\

\begin{defi} \label{definition - domaine de Sitter}
Let $x\in\h^{p,q}$. The affine domain associated to $x$ is 
\[U(x)=\{y\in\h^{p,q}\,\vert\,\pscalb{\tx}{\ty}\neq0\} = \h^{p,q}\setminus x^*.\]
The pseudo-spherical domain associated to $x$ is 
\[\partial U(x)=\{\xi\in\partial\h^{p,q}\,\vert\, \pscalb{\tx}{\txi}\neq 0\}=\partial\h^{p,q}\setminus \partial x^*.\]
\end{defi}

In order to understand why $U(x)$ is called an affine domain, consider $x_0=[0:\cdots:0:1]\in\h^{p,q}$.\\
\indent The affine domain $U(x_0)$ consists of points $[u_1:\cdots:u_{p+q+1}]\in\h^{p,q}$ such that $u_{p+q+1}\neq 0$. The affine chart $[u_1:\cdots:u_{p+q+1}]\mapsto (\frac{u_1}{u_{p+q+1}},\dots,\frac{u_{p+q}}{u_{p+q+1}})$ maps $U(x_0)$ to an open set $V$ of $\R^{p+q}$, and sends geodesics to affine lines in $\R^{p+q}$.\\
\indent More precisely, if we denote by $q_{p,q}$ the standard quadratic form of signature $(p,q)$ on $\R^{p+q}$, then $V=q_{p,q}^{-1}((-\infty,1))$ is the interior of the quadric  $Q=q_{p,q}^{-1}(\{1\})$, which is the image of $\partial U(x_0)$ through the same map.\\
If $q=0$, then $Q$ is a sphere, and we recover the Klein model of hyperbolic space. If $q=1$, then $Q$ is a one sheeted hyperboloid.\\
%\indent To understand why we call $\partial U(x)$ the de Sitter domain, this is because $H$ endowed with the restriction of $q_{n,1}$ to its tangent spaces is a Lorentzian manifold called the de Sitter space (it is the standard Lorentzian manifold of constant curvature $1$). Transposed to $\partial U(x)\subset \partial\AdS^{n+1}$, this gives a Lorentzian metric on $\partial U(x)$ which is in the conformal class of $\partial\AdS^{n+1}$, has constant curvature $+1$, and is invariant under the stabilizer of $x$ in $\PO(p,q+1)$ (which is a copy of $\PO(n,1)$). Such a metric is unique.\\

%
%In this model, it is easy to understand why a pair of points of $\partial\ADS^{n+1}$ is not always joined by a geodesic in $\ADS^{n+1}$, this is because  the interior of the hyperboloid $H$ is not convex (see Figure \ref{fig - geodesics}).\\

\indent Note that a similar description of $U(x)$ and $\partial U(x)$ is valid for any point $x\in\h^{p,q}$ (because the isometry group $\PO(p,q+1)$ acts transitively on $\h^{p,q}$).

\begin{figure} 
  \centering
 \includegraphics[scale=0.5]{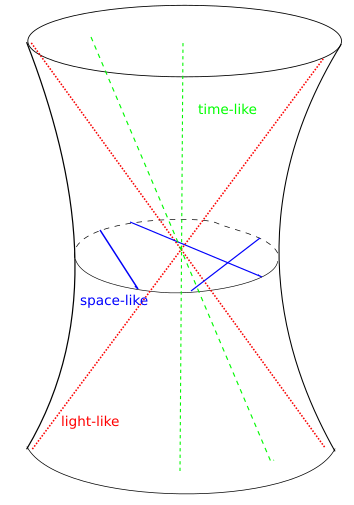}
   \caption{ {\color{black}Geodesics of $\h^{2,1}$ in an affine domain.}}
   \label{fig - geodesics}
\end{figure}

%
%As explained in \cite{barbot2007constant},  we can define a locally projective structure to $\AdS^n$.  An open hemisphere of $S^3$ is a connected component of $S^3$  minus a great two sphere.  For example we can consider for any $p \in \AdS^n$ the open hemisphere $U_p := \{ q\in S^3 \, | \, \Phi_q(p,q)<0\}$. The corresponding subspace of $\AdS$ is $A_p :=U_p \cap \AdS =  \{ q\in \AdS^3 \, | \, \Phi_q(p,q)<0\}$.
%Let $p_0:=[1 : 0:0:0] \in S^3$ 
%and consider the diffeomorphism 
%$$\fonction{\Psi_{0}}{U_{p_0}}{\R^3}{[x_1,x_2,x_3,x_4]}{\left(\frac{x_3}{x_1},\frac{x_4}{x_1},\frac{x_2}{x_1}\right).}$$
%
%The diffeomorphism $\Psi_{0}$ maps $A_{p_0}$ to the interior of the hyperboloid $(x^2+y^2-z^2<1)$, the boundary $\partial A_{p_0}$ to the hyperboloid $(x^2+y^2-z^2=1)$. Finally it maps the great circles of $S^3$ to affine lines. 
%
% Any subset of $\AdS^n$ which is contained in an half space can also be map as an affine manifolds  hyperplane we could send $\AdS^n$ 
%

%\subsection{The geometry of $\partial \ADS^{n+1}$} \label{subsec - geometry of the boundary}

Given $x\in \h^{p,q}$, {\color{black} we will equip} the pseudo-spherical domain $\partial U(x) =\{\xi\in\partial\h^{p,q}\,\vert\, \pscal{x}{\xi}\neq 0\}$ with a pseudo-Riemannian metric $g_x$ of signature $(p-1,q)$ such that $(\partial U(x),g_x)$ is  isometric to the pseudo-Riemannian sphere  $\Ss^{p-1,q}$.\\

If $\xi\in \partial \h^{p,q}$, the tangent space $T_\xi\partial\h^{p,q}$ is naturally identified with the quotient space $\xi^\perp/\xi$. If $\xi\in \partial U(x)$, then $\xi^\perp=(\xi^\perp\cap x^\perp)\oplus \xi$, so $\xi^\perp/\xi$ can be identified with $\xi^\perp\cap x^\perp$. The restriction of $\pscalb{\cdot}{\cdot}$ to $\xi^\perp\cap x^\perp$ has signature $(p-1,q)$, so it defines a pseudo-Riemannian metric $g_x$ of signature $(p-1,q)$ on $\partial U(x)$.\\

\indent The pseudo-Riemannian sphere is the standard pseudo-Riemannian manifold of constant positive curvature $1$. To define it, consider the quadratic form $\pscal{\cdot}{\cdot}_{p,q}$ of signature $(p,q)$ on $\R^{p+q}$, the pseudo-Riemannian sphere $\Ss^{p-1,q}$ is the level $\{ v\in \R^{p+q} \vert \pscal{v}{v}_{p,q}=1\}$ equipped with the restriction of $\pscal{\cdot}{\cdot}_{p,q}$ to its tangent spaces. Its isometry group is $\mathrm O(p,q)$.\\

\indent {\color{black}Considering the point $x_0=[0:\cdots:0:1]\in\h^{p,q}$, we find a diffeomorphism $\partial U(x_0)\to \Ss^{p,q}$ by sending  $[\xi_1:\cdots:\xi_{p+q+1}]$ to $(\frac{\xi_1}{\xi_{p+q+1}},\dots,\frac{\xi_{p+q}}{\xi_{p+q+1}})$. This is an isometry from $(\partial U(x_0),g_{x_0})$ to  $\Ss^{p-1,q}$.\\}

\indent Given any other point $x\in\h^{p,q}$, consider $\gamma\in\PO(p,q+1)$ such that $\gamma.x=x_0$. We then have $\gamma.\partial U(x)=\partial U(x_0)$, and  $g_x=\gamma^*g_{x_0}$, so $(\partial U(x),g_x)$ is also isometric to $\Ss^{p-1,q}$.\\
\indent Note that the isometry group of $(\partial U(x),g_x)$ is exactly  $\Stab(x)$.\\

\indent It is not possible to find a pseudo-Riemannian metric defined on the whole boundary $\partial \h^{p,q}$ which is invariant under $\PO(p,q+1)$ (because there are no preserved volume forms). However, there is an invariant pseudo-Riemannian conformal structure. Indeed, the metrics defined on the pseudo-spherical domains define the same conformal class on their intersection. This conformal pseudo-Riemannian manifold is called the Einstein Universe $\mathbb{E}\mathrm{in}^{p-1,q}$.\\

\subsection{$\h^{p,q}$-convex cocompact groups}

We follow the definitions in \cite{DGK}.
\begin{defi} A subset $\Omega\subset \h^{p,q}$ is  properly convex if its closure in $\R\proj^{p+q}$ is convex (i.e. it is contained and convex in some affine chart).\\
A group $\G\subset\PO(p,q+1)$ is $\h^{p,q}$-convex cocompact if it is discrete and the action of $\Gamma$ on $ \h^{p,q}$ preserves a set $\Omega$ with the following properties:
\begin{enumerate}
\item $\Omega$ is closed in  $\h^{p,q}$, is properly convex and has non empty interior.
\item The intersection of the closure $\overline\Omega\subset\R\proj^{p+q}$  with $\partial\h^{p,q}$ does not contain any non trivial line segment.
\item The action of $\G$ on $\Omega$ is properly discontinuous and cocompact.
\end{enumerate} 
\end{defi}

%\emph{Remarks:} 
%\begin{itemize} \item {\color{green}Compare with \cite{DGK}.}
%\item This definition is slightly different from the definition of \cite{merigot2012anosov}. Indeed, they ask that a quasi-Fuchsian group must be isomorphic to a uniform lattice in $\PO(n,1)$, which is not necessary for our work.
%%\item One can check that discreteness is automatic, except when $C(\Lambda)$ is a totally geodesic copy of $\h^n$, this is a consequence of the Anosov property.
%\item Torsion can be allowed, but it will be more practical to consider torsion free groups so that all elements of $\Gamma$ have the same type of dynamics when acting on $\Lambda$. To deal with torsion, first use Selberg's Lemma which states that $\Gamma$ has a finite index torsion free subgroup, and check that our results stay valid up to finite index.
%\end{itemize} 
%
%\indent The acausal set $\Lambda$ involved in the definition of a quasi-Fuchsian group $\Gamma$, called the limit set, is unique because it is the closure of the set of attractive fixed points on $\partial \ADS^{n+1}$ of elements of $\Gamma$.

We will focus most of our attention on the limit set of such a group.

\begin{defi} If $\G\subset \PO(p,q+1)$ is $\h^{p,q}$-convex cocompact, and $\Omega\subset\h^{p,q}$ satisfies the conditions mentioned above, then the limit set $\Lambda_\G$ is the set of accumulation points of the $\G$-orbit of any point in $\Omega$.
\end{defi}
It follows from the proper discontinuity of the action on $\Omega$ that it does not depend on the choice of a point in $\Omega$. Since two such sets necessarily intersect, it does not depend on the choice of $\Omega$ either.

\begin{prop}[\cite{DGK}]   If $\Gamma\subset \PO(p,q+1)$ is $\h^{p,q}$-convex cocompact, then $\Gamma$ is Gromov-hyperbolic, and the action of $\Gamma$ on $\Lambda_\G$ is topologically conjugate to the action on its Gromov boundary $\partial_\infty\Gamma$.
\end{prop}

{\color{black}For  $\h^{p,1}$-convex cocompact subgroups of $\PO(p,2)$ for which the limit set is homeomorphic to $\Ss^{p-1}$, this was proved in \cite{merigot2012anosov}.\\}
In particular, the action of $\Gamma$ on the set of triples of distinct points in $\Lambda_\G$ is properly discontinuous and cocompact, which will be of some use to us. This implies that infinite order elements of $\G$ have a north-south dynamic as for any  hyperbolic group acting on its boundary.
\begin{prop}\label{prop -north south dynamic}
If $\G\subset \PO(p,q+1)$ is a convex cocompact group, then  every infinite order element $\g\in \G\setminus\{\Id\}$ acts on $\Lambda$ with exactly two fixed points: $\g^\pm$. For every $\xi\in \Lambda \setminus\{\g^\pm\},$ we have $\lim_{n\tv \pm\infty} \g^n \xi = \g^\pm$. 
\end{prop}

\begin{defi}
For every infinite order element $\g\in \G\setminus$ we call the spacelike geodesic $(\g^-\g^+)$ the axis of $\g$. 
\end{defi}

\subsection{Negative sets, convex hulls and black domains}

Instead of considering the action on a properly convex set with non empty interior (as the ones involved in the definition of $\h^{p,q}$-convex cocompactness, we will work with the convex hull of the limit set (which can have empty interior).

\begin{defi} f $\G\subset \PO(p,q+1)$ is $\h^{p,q}$-convex cocompact, we define $C(\Lambda_\G)$ as the intersection of $\h^{p,q}$ with the convex hull of $\Lambda_\G$ defined in some affine chart containing a convex set $\Omega$ as defined above.
\end{defi}

One of the important properties of $\Lambda_\G$ is that it is negative.

{\color{black}\begin{defi} A subset $\Lambda\subset \partial\h^{p,q}$ is negative if it lifts  to a cone in $\R^{p,q+1}\setminus\{0\}$ on which all inner products for $\pscalb{\cdot}{\cdot}$ of non collinear vectors are negative.\\
If $\Lambda$ has at least three elements, this is equivalent to any triple $(\xi,\eta,\tau)\in \Lambda^3$ of pairwise distinct points satisfying $\pscalb{\txi}{\teta}\pscalb{\teta}{\ttau}\pscalb{\ttau}{\txi}<0$.
\end{defi}
Note that the sign of $\pscalb{\txi}{\teta}\pscalb{\teta}{\ttau}\pscalb{\ttau}{\txi}$ does not depend on a choice of lifts $\txi,\teta,\ttau \in  \R^{p+q+1}$. \\

This condition means that the intersection of the copy of $\R\proj^2$ spanned by $\xi,\eta,\tau$ with $\h^{p,q}$ is a totally geodesic copy of $\h^2$. As a consequence (if $\Lambda$ has at least three points), any two points of $\Lambda$ can be joined by a spacelike geodesic of $\h^{p,q}$.\\}
%Note that a negative set $\Lambda$ is always included in the boundary of an affine domain: choose two distinct points $\xi,\eta\in\Lambda$, and let $x$ be any point on the geodesic $(\xi\eta)$. A simple computation shows that $\Lambda\subset \partial U(x)$. This allows us to define the convex hull of $\Lambda$.
%
%\begin{defi} Let $\Lambda\subset \partial\h^{p,q}$ be a  negative set {\color{black}containing at least three points}. We denote by $\overline{C(\Lambda)}\subset \overline{\h^{p,q}}$ the convex hull of $\Lambda$, computed in any affine domain containing $\Lambda$ (it does not depend on the choice of an affine domain).\\
%We also let $C(\Lambda)=\overline{C(\Lambda)}\cap \h^{p,q}=\overline{C(\Lambda)}\setminus\Lambda$.
%\end{defi}

Another important subset of $\h^{p,q}$ associated to $\Lambda$ is its black domain (or invisible domain).

\begin{defi}Let $\Lambda\subset \partial\h^{p,q}$ be a negative set, and a lift $\tilde\Lambda\subset \R^{p,q+1}\setminus\{0\}$  on which all inner products of non collinear vectors are negative. Its black domain is $E(\Lambda)=\proj(\{ u\in \R^{p,q+1} \,\vert\, \pscalb{u}{v}<0 0~ \forall v\in\tilde\Lambda\}).$
\end{defi}

One can check that $E(\Lambda)$ is convex, and that it contains $C(\Lambda)$. Note that in \cite{DGK2}, the set $E(\Lambda)$ is denoted by $\Omega_{\max}$.

\begin{lemma}\label{lemma_disjoint} If $x\in E(\Lambda)$, then the dual hyperplane $x^*$ is disjoint from $C(\Lambda)$. \end{lemma}

\begin{proof}
It comes from the definition of the black domain $E(\Lambda)$ that $x^*$ is disjoint from $\Lambda$. This can be translated as $\Lambda\subset \partial U(x)$.\\
 Since $U(x)\cup\partial U(x)$ is convex and it contains $\Lambda$, it must contain $\overline{C(\Lambda)}$.
\end{proof}

%
%{\color{green} \emph{Remark:} At a first glance, our definition of acausal might seem different from the one in \cite{merigot2012anosov}. This is only because  we consider the projective model $\ADS$ instead of the linear model $\AdS$, which forces us to look at triples of points instead of pairs. One can easily check that  $\Lambda\subset \partial\ADS^{n+1}$ is acausal (as defined here) if and only if it is the image under the projection  $\partial \AdS^{n+1}\to\partial \ADS^{n+1}$  of a set $\widehat \Lambda\subset \partial \AdS^{n+1}$ which is acausal (as defined in \cite{merigot2012anosov}) on which the restriction of the projection $\partial \AdS^{n+1}\to\partial \ADS^{n+1}$ is injective. \\ \indent Note that acausal is called negative in \cite{DGK}.
%}
%

\subsection{Examples of $\h^{p,q}$-convex cocompact groups}

When $q=0$, the notion of $\h^{p,0}$-convex cocompactness is equivalent to the usual notion in real hyperbolic geometry. This allows to construct examples in any signature: consider a convex cocompact group $\G\subset \mathrm{O}(p,1)$, and its image through the standard embedding $\mathrm{O}(p,1)\hookrightarrow \PO(p,q+1)$. It is $\h^{p,q}$-convex cocompact. One can also consider any representation $\alpha:\G\to \mathrm{O}(q)$, and the image of its graph $\{(\g,\alpha(\g))\vert\g\in\G\}\subset \mathrm{O}(p,1)\times\mathrm{O}(q)$ through the standard embedding $\mathrm{O}(p,1)\times\mathrm{O}(q)\hookrightarrow \PO(p,q+1)$. It is also $\h^{p,q}$-convex cocompact. An interesting case is the one for which $\G$ is a uniform lattice in $\mathrm{O}(p,1)$.
\begin{defi} A group $\G\subset\PO(p,q+1)$ is called $\h^{p,q}$-Fuchsian if it acts properly discontinuously and cocompactly on a totally geodesic copy of $\h^p$.
\end{defi}
If $\G$ is $\h^{p,q}$-Fuchsian, its limit set is a smooth $(p-1)$-sphere. A direct generalization is to ask for $\Lambda_\G$ to be a topological $(p-1)$-sphere.

\begin{defi} A group $\G\subset\PO(p,q+1)$ is called $\h^{p,q}$-quasi-Fuchsian if it is $\h^{p,q}$-convex cocompact and its limit set $\Lambda_\G$ is homeomorphic to $\Ss^{p-1}$.
\end{defi}

{\color{black} In any dimension and signature, there are examples of uniform lattices $\Gamma\subset \mathrm O(p,1)$  that can be deformed non trivially in $\PO(p,q+1)$, the standard construction being the so called bending deformations \cite{johnson1987deformation}. Because of the equivalence of $\h^{p,q}$-convex cocompactness with the Anosov property showed in \cite{DGK2}, such groups are $\h^{p,q}$-quasi-Fuchsian.\\

%
%The standard examples of quasi-Fuchsian groups are Fuchsian groups. Given a point $x\in \ADS^{n+1}$, we have an injection $i_x:\mathrm O(n,1)\hookrightarrow \PO(p,q+1)$ given by identifying $\mathrm O(n,1)$ and the stabilizer of $x$ in $\PO(p,q+1)$. Given a torsion free uniform lattice $\Gamma_0\subset \mathrm O(n,1)$, we consider $\Gamma=i_x(\Gamma_0)$. Such a group is called Fuchsian, and it is the simplest example of a quasi-Fuchsian group. Indeed, the group $\Gamma$ preserves $\partial x^*\subset \partial \ADS^{n+1}$ which is an acausal sphere. The convex hull $C(\partial x^*)=x^*$ is a totally geodesic copy of $\h^n$ in $\ADS^{n+1}$, and the action of $\G$ on $C(\partial x^*)$ is conjugate to the action of $\G_0$ on $\h^n$, which is cocompact.\\
%
%
%
%
%\indent Note that a quasi-Fuchsian group $\Gamma\subset \PO(p,q+1)$ is Fuchsian if and only if it preserves a totally geodesic copy of $\h^n$. Indeed, if $P\subset \ADS^{n+1}$ is such a hyperplane, there is a unique point $x\in\ADS^{n+1}$ such that $P=x^*$, and $\Gamma\subset \mathrm{Stab}(x)$, i.e. $\Gamma=i_x(\Gamma_0)$ for a certain group $\Gamma_0\subset \mathrm O(n,1)$. Furthermore, one necessarily has $\Lambda=\partial P$ (because $\Lambda$ is the closure of attractive fixed points of elements of $\Gamma$, and those fixed points lie on $\partial x^*$), hence $C(\Lambda)=P$. The action of $\Gamma$ on $C(\Lambda)$ being cocompact, we see that $\Gamma_0$ is a uniform lattice in $\mathrm O(n,1)$.\\

This notion of $\h^{p,q}$-quasi-Fuchsian groups has  been studied when $q=1$, as $\h^{p,1}$ is an important space in Lorentzian geometry, called the anti-de Sitter space $\ADS^{p+1}$. Torsion free $\ADS^{p+1}$-quasi-Fuchsian groups in $\PO(p,2)$ are related to globally hyperbolic $\ADS^{p+1}$-spacetimes \cite{mess2007lorentz,barbot2008causal}

\indent One can also consider deformations of Fuchsian groups, i.e. groups $\Gamma=\rho_1(\Gamma_0)$ where $\Gamma_0\subset\mathrm O(p,1)$ is a uniform lattice and $(\rho_t)_{t\in [0,1]}$ is a continuous path of representations with $\rho_0=$ being the restriction to $\G_0$ of the standard inclusion $\mathrm{O}(p,1)\hookrightarrow \PO(p,2)$.

\begin{theorem}\cite{barbot2015deformations}
Deformations of $\ADS^{p+1}$-Fuchsian groups are $\ADS^{p+1}$-quasi-Fuchsian.
\end{theorem}

\indent Note that deformations of Fuchsian groups are not the only examples of quasi-Fuchsian groups in $\PO(p,2)$, as examples of $\ADS^{p+1}$-quasi-Fuchsian groups $\Gamma\subset \PO(p,2)$ which are not isomorphic to a lattice in $\mathrm{O}(p,1)$ have recently been constructed  by G-S. Lee and L. Marquis \cite{Lee2017ads} for $4\leq p\leq 8$.}

Several interesting families of surface groups representations provide examples of $\h^{p,q}$-convex cocompact groups when the group $\PO(p,q+1)$ has specific Lie-theoretic properties. For the real split forms (i.e. $p=q$ or $p=q+1$), one can study Hitchin representations.

\begin{prop}[Proposition 11.14 in \cite{DGK2}] Let $\G$ be the fundamental group of a closed orientable hyperbolic surface, and let $m\geq 1$.\\
For any Hitchin representation $\rho:\G\to\PO(m+1,m)$, the group $\rho(\G)$ is $\h^{m+1,m-1}$-convex cocompact if $m$ is odd, and $\h^{m,m}$-convex cocompact if $m$ is even.\\
For any Hitchin representation $\rho:\G\to\PO(m+1,m)$, the group $\rho(\G)$ is $\h^{m+1,m-1}$-convex cocompact if $m$ is odd, and $\h^{m,m}$-convex cocompact if $m$ is even.
\end{prop}

When $p=2$, the Lie group $\PO(2,q+1)$ is of Hermitian type, and one can study maximal representations of surface groups. It is shown in \cite{BILW} that maximality is equivalent to the Anosov property with an additional property on the limit set which can be shown to be equivalent to negativity. Combining with the results of \cite{DGK2}, one gets the following:

\begin{theo}[\cite{BILW,DGK2}] Let $\G$ be the fundamental group of a closed orientable hyperbolic surface, and let $q\geq 0$. A faithful representation $\rho:\G\to\PO(2,q+1)$ is maximal if and only if its image is $\h^{2,q}$-quasi-Fuchsian.
\end{theo}

These representations are studied in \cite{CTT}. \\

Note that there is an example common to all these situations, which is $\h^{2,1}=\AdS^3$. In this case, we can use the exceptional isomorphism $\PO(2,2)\approx \PSL(2,\R)\times\PSL(2,\R)$ to describe $\h^{2,1}$-quasi-Fuchsian groups.

\begin{theo}[\cite{mess2007lorentz,barbot2008causal}] A torsion free group $\Gamma\subset \PO(2,2)\approx \PSL(2,\R)\times \PSL(2,\R)$ is $\h^{2,1}$-quasi-Fuchsian if and only if there is a closed orientable surface $S$, and  two hyperbolic metrics $h_1,h_2$ on $S$  such that \[\Gamma= \{(\rho_1(\gamma),\rho_2(\gamma))\,\vert\, \gamma\in\pi_1(S) \}\] where  $\rho_1,\rho_2:\pi_1(S)\to \PSL(2,\R)$ are the holonomy representations of $h_1,h_2$.

\end{theo}

This is a Lorentzian analogue of the Bers simultaneous uniformization theorem \cite{bers1972uniformization}. In this description of quasi-Fuchsian subgroups of $\PO(2,2)$ (called the Mess parametrization), Fuchsian groups correspond to pairs $(\rho_1,\rho_2)$ where $\rho_1$ and $\rho_2$ are conjugate in $\PSL(2,\R)$ (i.e. they represent the same point in the Teichmüller space of $S$).

\section{Geometric toolbox}\label{sec -geometric toolbox}
This section contains all the geometric lemmas that will be used in the rest of the paper. We start by the definition of the pseudo-Riemannian distance in the convex hull, $d_{\h^{p,q}}$. It is by definition semi-definite and we prove that it satisfies a triangle inequality up to fixed additive constant.  We then study the geometry of the limit set of $\G$ from the point of view of the distance $d_{\h^{p,q}}$, in particular we show that every limit points are radial limit points. Finally we define the Gromov product and distance on the boundary as in the case of Gromov hyperbolic spaces and study their properties.\\

{\color{black}From now on, a quasi-Fuchsian group $\G\subset \PO(p,q+1)$ is fixed, and $\Lambda\subset \partial\h^{p,q}$ is its limit set.}

\subsection{Triangle inequality}\label{sec - triangle inequality}

%{\color{black}\sout{The group of isometries of $\ADS^{n+1}$, $\PO(p,q+1)$, acts transitively on the set of 2-plane of $\R^{p,q+1}$ of given signature. In particular all plane of signature $(1,1)$ are in the same $\PO(p,q+1)$-orbit, i.e.  $\PO(p,q+1)\cdot \Vect(e_1,e_{n+2})$. The intersection of this plane with $\ADS^{n+1}$, consist of a geodesic parametrized by:
%$[\g(t)]=[\sinh(t):  0: \dots: 0: 0: \cosh(t))$. It is a unit speed geodesic:}
%$$\|\g'(t)\|^2_{p,q+1} = \cosh^2(t) - \sinh^2(t) = 1.$$}

%
%\begin{figure}[H]
%
%\begin{center}
%
%
%\begin{tikzpicture}[scale=1.3] 
%
%
%\draw (-2.5,0) -- (2.5,0);
%\draw (0,-1.5) -- (0,2.5);
%\draw[->]        (0,0)   -- (1,0);
%\draw (1,0) node[below] {$e_1$};
%\draw[->]        (0,0)   -- (0,1);
%\draw (0,1) node[right] {$e_{n+2}$};
%
%
%
%\draw [color=red] (0,-0) -- (2,2);
%\draw [color=red] (0,0) -- (-2,2);
%\draw [color=red] (-1.5,1.5) node[left] {Light cone};
%
%
%
%tikz \draw[scale=0.5,domain=-2:2,smooth, color=blue, variable=\t]
%  plot ({sinh(\t )},{cosh(\t)});
%  
%%tikz \draw[scale=0.5,domain=-2:2,smooth, color=blue, variable=\t]
%  %plot ({-sinh(\t )},{-cosh(\t)}); 
%
%\draw [color=blue] (1.4,1.5) node[left] {$\g(t)$};
%%\draw [color=blue] (-1.4,-1.5) node[right] {$-\g(t)$};
%
%
%\end{tikzpicture}
% \caption{\label{space like geodesic} Intersection of a (1,1)-plane  and $\ADS^{n+1}$}
%\end{center}
%\end{figure}

\begin{defi}\label{def - definition de la distance dasn C(Lambda)}
Let $x,y\in \h^{p,q}$. If $(xy)$ is spacelike, we denote by $d_{\h^{p,q}}(x,y)$ the length  of the spacelike  geodesic in $\h^{p,q}$ between $x$ and $y$. If $(xy)$ is lightlike or timelike then we set  $d_{\h^{p,q}}(x,y)=0$.
\end{defi}

\begin{prop}\label{prop - formula for the distance}
Let $x,y\in \h^{p,q}$. Then $x$ and $y$ are joined by a spacelike geodesic if and only if $|\pscal{\tx}{\ty}|_{p,q+1}>1 $. Moreover, in that case 
$$d_{\h^{p,q}}(x,y) =\Ach(|\pscal{\tx}{\ty}|_{p,q+1}  ).$$
\end{prop}
\begin{proof}
{\color{black}{
Choose $\tx,\ty\in \h^{p,q}$  any representatives of  $x$ and $y$. Since $\pscalb{\tx}{\tx}=\pscalb{\ty}{\ty}=-1$, the matrix of the restriction of  $\pscalb{\cdot}{\cdot}$   to $x\oplus y$ in the basis $(\tx,\ty)$ is $\begin{pmatrix}
-1 & \pscal{\tx}{\ty}_{p,q+1}\\
\pscal{\tx}{\ty}_{p,q+1} & -1
\end{pmatrix}$. The geodesic $(xy)$ is space like if and only if the signature of this matrix is $(1,1)$, which is equivalent to $|\pscal{\tx}{\ty}|_{p,q+1}>1$. 
}}
In that case, up to the action of $\PO(p,q+1)$, we can assume that $\tx=(0, \dots,  0, 1)$ and $\ty=(\sinh(t), 0, \dots, 0, \cosh(t))$ for some $t\in \R$.  We find $d_{\h^{p,q}}(x,y)=t$ and $\vert\pscalb{\tx}{\ty}\vert=\cosh(t)$. 
\end{proof}
Remark that for any $x,y\in \h^{p,q}$ we have, $$d_{\h^{p,q}}(x,y)=\Ach(\max(1, |\pscal{\tx}{\ty}|_{p,q+1}  ).$$

{\color{black}{Let $x\in \h^{p,q}$ and  $r\in \R$. If we were to define closed balls for $d_{\h^{p,q}}$ as   $\{y \in \h^{p,q}\, , \, d_{\h^{p,q}}(x,y)\leq r \} = \{y \in \h^{p,q}\, , \, |\pscalb{\tx}{\ty}| \leq \cosh(r) \} $, then closed balls would not be compact. For this reason, we only consider balls in the convex hull $C(\Lambda)$: 
\begin{lemme}\label{lem - balls are compact}
For all $r>0$ and all $x\in C(\Lambda)$, the set  
 \[B_{C(\Lambda)}(x,r)=\{y \in C(\Lambda)\, , \,  |\pscalb{\tx}{\ty}| \leq \cosh(r) \}\] is compact. We call it the $\h^{p,q}$ or pseudo-Riemannian ball of center $x$ and radius $r$.
\end{lemme}
\begin{proof}
Take $y_k$ a sequence in $\{y \in C(\Lambda) \, , \,  |\pscalb{\tx}{\ty}| \leq \cosh(r) \} $ and suppose by contradiction that it is not bounded. Then we can extract a converging subsequence to  $\eta\in  \partial\h^{p,q}$ that we still denote by $y_k$.   
Since $y_k\tv \xi \in \partial \h^{p,q}$ we have  $\lim_{k\tv \infty} \|\ty_k\|_{e} =+\infty$, where $\|\cdot\|_e$ denotes the Euclidean norm. Indeed, we have $\pscalb{\ty_k}{\ty_k}=-1$,  $\pscalb{\txi}{\txi}=0$ and therefore  
\begin{eqnarray*}
\lim_{k\tv \infty}\pscalb{ \frac{\ty_k}{\|\ty_k\|_{e}}}{\frac{\ty_k}{\|\ty_k\|_{e}}} &=& \pscalb{\txi}{\txi}\\
\lim_{k\tv \infty}\frac{-1}{\|\ty_k\|_{e}}&=& 0
\end{eqnarray*}
Now $ \|\ty_k\|_{e}|\langle\tx|\frac{\ty_k}{\|\ty_k\|_{e}}\rangle_{p,q+1}|=|\langle\tx|\ty_k\rangle_{p,q+1} |\leq \cosh(r). $  However, since $x\in C(\Lambda)\subset E(\Lambda)$, and $\xi\in \Lambda$ we have, $\langle \tx|\txi\rangle_{p,q+1}\neq 0$, or for the lifts  $|\langle \tx|\txi\rangle_{p,q+1}|> 0$. Finally we have: $$\lim_{k\tv \infty}  \|\ty_k\|_{e}|\langle\tx|\frac{\ty_k}{\|\ty_k\|_{e}}\rangle_{p,q+1}|= \lim_{k\tv \infty}  \|\ty_k\|_{e} | \langle \tx|\txi\rangle_{p,q+1}|=+ \infty.$$ 
This is a contradiction.

\end{proof}

More generally, with the same proof,  we can show that the $r$-neighborhood of any compact  set $K$  is compact. It will be denoted by $B_ {C(\Lambda)}(K,r)$.

Balls of radius $0$ are cones and  those of positive radius are interiors of one sheeted hyperboloids (more precisely intersection of cones or hyperboloids with $C(\Lambda)$.)

}}

\begin{lemma} \label{lemma_bounded} The function $F:  C(\Lambda)^3\to \R$ defined by $F(x,y,z) =  \frac{\pscalb{\tx}{\ty}}{\pscalb{\tx}{\tz}\pscalb{\tz}{\ty}}$ extends to a continuous bounded function  on $\overline{ C(\Lambda)}^2  \times C(\Lambda)$.
\end{lemma}
 \begin{proof}
First, notice that for all $x\in C(\Lambda)\subset E(\Lambda)$ and  $z\in C(\Lambda)$ we have by Lemma \ref{lemma_disjoint}  $\pscalb{\tx}{\tz} \neq 0$. Hence $F$ is well defined. \\
For all $\xi \in \Lambda$ and all  $z\in C(\Lambda)\subset E(\Lambda)$, we have by definition of $E(\Lambda)$ that $\pscalb{\txi}{\tz} \neq 0$. Therefore  $F$ extends to a function on $ 
  \overline{C(\Lambda)}^2  \times C(\Lambda)$.  As it is defined by ratios of scalar products, it is continuous. {\color{black}{We  can verify it by basic computations, using lifts of the different elements: we take $x_i$ a sequence converging to $\xi\in \partial \h^{p,q}$. The lifts satisfy $\lim \frac{\tx_i}{\|\tx_i\|_{e}}= \txi$ for some lift $\txi$ of $\xi$.  
  \begin{align*}F(x_i,y,z) &= \frac{\pscalb{\tx_i}{\ty}}{\pscalb{\tx_i}{\tz}\pscalb{\tz}{\ty}}\\ 
  &= \frac{\pscalb{\frac{\tx_i}{\|\tx_i\|_{e}}}{\ty}}{\pscalb{\frac{\tx_i}{\|\tx_i\|_{e}}}{\tz}\pscalb{\tz}{\ty}} 
  &\longrightarrow  \frac{\pscalb{\txi}{\ty}}{\pscalb{\txi}{\tz}\pscalb{\tz}{\ty}}=F(\xi,y,z)
  \end{align*} %and we see that  $\frac{\langle \tx_i\vert \ty \rangle_{p,q+1}}{\langle \tx_i\vert \tz \rangle_{p,q+1}} = \frac{\langle \frac{\tx_i}{\|\tx_i\|_{e}}\vert \ty \rangle_{p,q+1}}{\langle \frac{\tx_i}{\|\tx_i\|_{e}} \vert \tz \rangle_{p,q+1}} $ which converges to   $\frac{\langle \txi \vert \ty \rangle_{p,q+1}}{\langle \txi\vert \tz \rangle_{p,q+1}}$. Therefore, $$\lim F(x_i,y,z)=\lim F(\tx_i,\ty,\tz)=F(\txi,\ty,\tz)=F(\xi,y,z).$$
}}
 Let $K\subset C(\Lambda)$ be a compact set such that $\Gamma.K=C(\Lambda)$. Since $F$ is $\Gamma$-invariant, its values are all taken on the compact set $\overline{C(\Lambda)}^2\times K$ and $F$ is continuous, which shows that $F$ is bounded.
 \end{proof}

\begin{theorem}[Triangle inequality]\label{th - triangle inequality}
There is a constant $k_\G>0$ such that $d_{\h^{p,q}}(x,y)\le d_{\h^{p,q}}(x,z)+d_{\h^{p,q}}(z,y)+k_\G$ for all $x,y,z\in C(\Lambda)$.
\end{theorem}

\begin{proof}
If $x$ and $y$ are causally related (that is $|\pscalb{\tx}{\ty}|\leq 1$), the inequality is automatic because the left hand side is $0$, so we can assume that $x$ and $y$ are separated by a spacelike geodesic. {\color{black}{This means that $d_{\h^{p,q}}(x,y)=\Ach \vert\langle x\vert y \rangle_{p,q+1}\vert $. We are going to study the three disjoint cases: $d_{\h^{p,q}}(x,z)=d_{\h^{p,q}}(y,z)=0$, $d_{\h^{p,q}}(x,z)>0$ and $d_{\h^{p,q}}(y,z)=0$, and the last one, with both  $d_{\h^{p,q}}(x,z)>0$ and $d_{\h^{p,q}}(y,z)>0$. 

\begin{itemize}
\item First assume that $d_{\h^{p,q}}(x,z)=d_{\h^{p,q}}(y,z)=0$. We wish to show that $d_{\h^{p,q}}(x,y)$ is bounded. Up to the action of $\Gamma$, we can suppose that $z$ belongs to a compact fundamental domain $K$. The conditions imply that $x$ and $y$ are in the $0$-neighborhood of $z$, so in the $0$-neighborhood of $K$, which is compact from  Lemma \ref{lem - balls are compact}. It follows that $d_{\h^{p,q}}(x,y)$ is bounded.

\item  Now assume that  $d_{\h^{p,q}}(x,z)>0$ and $d_{\h^{p,q}}(y,z)=0$. Again, we can assume that $z$ is in a compact fundamental domain $K$, and $y$ lies in the compact $B_{C(\Lambda)}(K,0)$, the $0$-neighborhood of K inside $C(\Lambda)$.  Consider $g(x,y,z)=d_{\h^{p,q}}(x,y)-d_{\h^{p,q}}(x,z) \leq \ln\left(\frac{|\langle \tx | \ty \rangle_{p,q+1}|}{|\langle \tx | \tz \rangle_{p,q+1}|} \right)+\ln(2)$, this is a continuous function on $C(\Lambda)\times B_{C(\Lambda)}(K,0)\times K$. Similarly as Lemma \ref{lemma_bounded}, this function extends to a continuous function on $\overline{C(\Lambda)}\times B_{C(\Lambda)}(K,0)\times K$ which is compact. It is therefore bounded.

\item Finally assume that $x,y,z$ are pairwise joined by spacelike geodesics, that is $ d_{\h^{p,q}}(x,z)=\Ach \vert\langle x\vert z \rangle_{p,q+1}\vert $ and $d_{\h^{p,q}}(z,y) =  \Ach \vert\langle z\vert y \rangle_{p,q+1}\vert$. Since we have $\ln t \le \Ach t \le \ln t + \ln 2$ for all $t>1$, in order to show that $(x,y,z)\mapsto d_{\h^{p,q}}(x,y)- d_{\h^{p,q}}(x,z)- d_{\h^{p,q}}(z,y)$ is bounded from above, it is enough to show that it is the case for $(x,y,z)\mapsto \ln  \frac{\vert\langle x\vert y \rangle_{p,q+1}\vert}{\vert\langle x\vert z \rangle_{p,q+1}\langle z\vert y \rangle_{p,q+1}\vert}$. This is true because of Lemma \ref{lemma_bounded}.\\ 

\end{itemize}
}}

\end{proof}

{\color{black}{ \emph{Notation:} We will use throughout the rest of the paper the notation $k_\G$ to refer at this constant in the triangle inequality. 
This theorem allows us to define a good notion of pseudo-Riemannian critical exponent for quasi-Fuchsian groups:
\begin{defi}\label{def- critical exponent}
We call critical exponent the real number defined by:
$$\delta_{\h^{p,q}}(\G):=  \limsup_{R\tv \infty} \frac{1}{R} \log \Card \{ \g\in \G \, |\, d_{\h^{p,q}}(\g o , o )\leq R\},$$
for any point $o\in C(\Lambda)$. It does not depend on $o$ thanks to Theorem \ref{th - triangle inequality}. 
\end{defi}

Another straightforward consequence is that if $\G_0\subset \G$ is a finite index subgroup, then $\delta_{\h^{p,q}}(\G_0)=\delta_{\h^{p,q}}(\G)$. 
%If we want to distinguish this critical exponent from another one coming from the action of the group $\G$ on another space, we will denote it by $\delta_{\ADS}(\G)$ 
} }

\subsection{Limit set} \label{subsec - limit set}

\begin{lemma}\label{lem - Two space like rays with the same endpoint are at bounded distance.}
Two spacelike rays with the same endpoint are at bounded distance.
\end{lemma}

\begin{proof}
Let $x, y\in \h^{p,q}$ and $\eta\in \partial \h^{p,q}$. We consider the parametrization of the geodesic rays $[x\eta)$ and $[y\eta)$:
%The equation of the geodesic joining $x$ to $\eta$ 
\[ \tx(t)=e^{-t} \tx - \frac{\sinh t}{\langle \tx\vert\teta\rangle_{p,q+1}} \teta  \quad \text{and} \quad \ty(t)=e^{-t} \ty - \frac{\sinh t}{\langle \ty\vert\teta\rangle_{p,q+1}} \teta. \]
%The equation of the geodesic joining $y$ to $\eta$ is:
%\[ \ty(t)=e^{-t} \ty - \frac{\sinh t}{\langle \ty\vert\eta\rangle} \eta\]
We estimate  $d(x(t),y(t))=\Ach \vert\langle \tx(t)\vert \ty(t)\rangle_{p,q+1}\vert$, looking at:
\begin{align*}
 \langle \tx(t) \vert \ty(t)\rangle_{p,q+1} &= \langle e^{-t} \tx - \frac{\sinh t}{\langle \tx\vert\teta\rangle_{p,q+1}} \teta \vert e^{-t} \ty - \frac{\sinh t}{\langle \ty\vert\teta\rangle_{p,q+1}} \teta \rangle_{p,q+1} \\
 &= e^{-2t} \langle \tx\vert \ty\rangle_{p,q+1} - e^{-t} \sinh t \left( \frac{\langle \tx\vert\teta\rangle_{p,q+1}}{\langle \ty\vert \teta\rangle_{p,q+1}}+\frac{\langle \tx\vert\teta\rangle_{p,q+1}}{\langle \ty\vert \teta\rangle_{p,q+1}} \right).
 \end{align*}
This function is bounded since $\lim_{\infty} e^{-t} \sinh t = 1/2$. (Here, $x$,$y$ and $\eta$ are fixed) 
\end{proof}

We follow the usual definition of Buseman functions, in the hyperbolic case they are deeply studied in \cite{ballmann1985manifolds}
\begin{defi}
The Buseman function centered at $\xi\in \Lambda$, is the function $\beta_\xi$ on $C(\Lambda)^2$ defined by:
$$ \forall x, y \in C(\Lambda)~~\beta_{\xi} (x,y)= \ln \left( \left|\frac{\pscal{\txi}{ \tx}_{p,q+1}}{\pscal{\txi }{\ty}_{p,q+1}}\right|\right).$$
\end{defi}

\begin{lemme}
$$\lim_{z\tv \xi} d_{\h^{p,q}}(z,x)-d_{\h^{p,q}}(z,y) =\beta_\xi (x,y).$$
\end{lemme}

\begin{proof}
Let $\xi\in \Lambda$, $x\in C(\Lambda)$ and $z_i\in C(\Lambda) $ such that $\lim_{i\tv \infty} z_i =\xi$.
As we have seen in Lemma \ref{lem - balls are compact}, $\lim_{i\tv \infty}\langle \tx | \tz_i\rangle_{p,q+1}=+\infty$. \\
In particular,  for $i$ large enough, the geodesics joining $z_i$ with $x$ and $y$ are spacelike,  and 
$$d_{\h^{p,q}}(z_i,x)-d_{\h^{p,q}}(z_i,y) = \Ach(|\pscal{\tz_i}{\tx}_{p,q+1}| )- \Ach(|\pscal{\tz_i}{\tx}_{p,q+1}|).$$
By simple calculus we have $\Ach(t)= \ln(2t) +o(1/t)$ as $t\tv \infty$. 
Hence $d_{\h^{p,q}}(z_i,x)-d_{\h^{p,q}}(z_i,y) = \ln(|\pscal{\tz_i}{\tx}_{p,q+1}|) - \ln (|\pscal{\tz_i}{\tx}_{p,q+1}|) + o(\frac{1}{\|\tz_i\|_{e}})$ and finally we have: 
\begin{eqnarray*}
\lim_{i\tv \infty} d_{\h^{p,q}}(z_i,x)-d_{\h^{p,q}}(z_i,y)  &=& \lim_{k\tv \infty} \ln\left(\frac{|\pscal{\tz_i}{\tx}_{p,q+1}|  }{|\pscal{\tz_i}{\ty}_{p,q+1}|  } \right)\\
									&=& \ln\left(\frac{|\pscal{\txi}{\tx_i}_{p,q+1}|  }{|\pscal{\txi}{\ty_i}_{p,q+1}| } \right)\\
									&=& \beta_\xi(x,y).
\end{eqnarray*}
\end{proof}

%%%%%%%%%%%%%%%%%%%%%%%%%%%%%%%%%%%%%%
%
%\begin{lemma}{\color{black}Finalement on n'en a pas besoin, je le laisse là au cas où} \label{not_too_far} Let $K\subset C(\Lambda)$ be a compact set. There is a constant $d>0$ such that, for all $x,y\in K$, $\eta\in \Lambda$ and $z\in [x\eta)$, one has $d(z,[y\eta))\le d$.
%\end{lemma}
%

\subsection{Shadows}

\begin{defi} Let $x,y\in C(\Lambda)$, and $r>0$. The shadow $\mathcal S_r(x,y)$ is $\{\xi\in \Lambda\,  \vert\,  [x,\xi)\cap B_{C(\Lambda)}(y,r) \ne \emptyset\}$, where $B_{C(\Lambda)}(y,r)$ is the Lorentzian ball.

\end{defi}

\emph{Remark:} This is slightly different from the usual definition of shadows as we require that points in shadows lie on the limit set.

The following lemma is classical in the setting of CAT($-1$) spaces, compare for example with \cite[Lemma 4.1]{quint2006overview}.
\begin{lemma} \label{shadow_1} Let $x,y\in C(\Lambda)$ and $r>0$. For all $\xi\in \mathcal S_r(x,y)$, one has:
\[ d_{\h^{p,q}}(x,y)-2r-2k_\G \le \beta_\xi(x,y)\le d_{\h^{p,q}}(x,y)+k_\G . \]
Recall that $k_\G$ is the constant in the triangle inequality \ref{th - triangle inequality}.
\end{lemma}

\begin{proof} Let $x(t)$ be the geodesic such that $x(0)=x$ and $x(+\infty)=\xi$. By definition of the shadow, there is $t_0\ge 0$ such that $d_{\h^{p,q}}(y,x(t_0))<r$.\\
Since $x(+\infty)=\xi$, one has $\beta_\xi(x,y)=\lim_{t\to +\infty} d_{\h^{p,q}}(x,x(t))-d_{\h^{p,q}}(y,x(t))$.\\

\[d_{\h^{p,q}}(y,x(t))\le d_{\h^{p,q}}(y,x(t_0))+d_{\h^{p,q}}(x(t_0),x(t))+k_\G \]

If $t\ge t_0$, one has $d_{\h^{p,q}}(x(t_0),x(t))=d_{\h^{p,q}}(x,x(t))-d_{\h^{p,q}}(x,x(t_0))$, hence

\[ d_{\h^{p,q}}(y,x(t)) \le r + d_{\h^{p,q}}(x,x(t)) - d_{\h^{p,q}}(x,x(t_0)) +k_\G\]

\[ d_{\h^{p,q}}(x,x(t))-d_{\h^{p,q}}(y,x(t)) \ge d_{\h^{p,q}}(x,x(t_0)) - r - k_\G \]

However, $d_{\h^{p,q}}(x,x(t_0))\ge d_{\h^{p,q}}(x,y) - d_{\h^{p,q}}(y,x(t_0))-k_\G$, hence:

\[ d_{\h^{p,q}}(x,x(t))-d_{\h^{p,q}}(y,x(t)) \ge d_{\h^{p,q}}(x,y) -2(r+k_\G)\]

Letting $t\to +\infty$ gives the left hand side of the desired inequality.\\
For the right hand side, simply notice that $d_{\h^{p,q}}(x,x(t))-d_{\h^{p,q}}(y,x(t))\le d_{\h^{p,q}}(x,y) +k_\G$.

\end{proof}

\subsection{Radial convergence}

\begin{defi}
A point $\xi\in\Lambda$ is radial if there exists a sequence  $g_i\in \G$ and a point $x$ such that  $g_i x$ converges to $\xi$ and is at bounded distance of one (any) spacelike ray with endpoint $\xi$. In that case we say that the sequence $g_i x$ converges radially to $\xi$. 
\end{defi}
\begin{lemma}
If $\xi $ is radial there exists a sequence $\g_i \in \G$ and a constant $L>0$ such that :
$$\left| \beta_{\xi}(x,\g_i x ) -d_{\h^{p,q}}(x,\g_i x)\right| \leq L.$$
\end{lemma} 

\begin{proof}
If $\xi$ is radial, by definition there exists a sequence $\g_i\in \G$, a point $x\in \ADS^{n+1}$  and a constant $L$ such that the sequence $d_{\h^{p,q}}(\g_i x , [x\xi) ) \leq L$. This means that $\xi\in \mathcal{S}_L(x,\g_i x) $. Conclude by Lemma \ref{shadow_1}.
\end{proof}

Recall from   Proposition \ref{prop -north south dynamic} that any element $\g\in \G\setminus\{Id\}$ has exactly two fixed points in $\Lambda$ denoted by $\g^\pm$.
\begin{lemma}
For any $\g\in\G\setminus\{Id\}$, the point $\g^+$ is radial. 
\end{lemma}

\begin{proof}
 For any point $x$ on the axis of $\g$, the sequence $\g^i x $ converges radially to $\g^+$. 
\end{proof}
A stronger statement is true:
\begin{lemma}
Every  point of $\Lambda$ is radial. 
\end{lemma}

\begin{proof}
This is due to the cocompactness of the action on the convex core. Let $x\in C(\Lambda)$ and $\xi \in \Lambda$.  Let $K$ be a compact fundamental domain for the action of $\G$ on $C(\Lambda)$. The geodesic ray $[x,\xi)$ is covered by a sequence $(g_i K)$ of translates of $K$ \cite[Lemma 7.5]{merigot2012anosov}. Then the sequence $g_i x$ converges radially to $\xi. $
\end{proof}

\begin{lemma} \label{lemma_description_limit_set}
We choose an enumeration $\Gamma=\{\gamma_p : p\in \N\}$ of $\Gamma$.
\[\Lambda = \bigcup_{r>0}\bigcap_{N\in \N}\bigcup_{p\geq N} \mathcal S_r(x,\gamma_p.x)\]
\end{lemma}
\begin{proof} The fact that $\Lambda \supset \bigcup_{r>0}\bigcap_{N\in \N}\bigcup_{p\geq N} \mathcal S_r(x,\gamma_p.x)$ comes from the definition of shadows (they are subsets of $\Lambda$).\\
 Let $\xi\in \Lambda$. We wish to find $r>0$ such that $\xi \in \bigcap_{N\in \N}\bigcup_{p\geq N} \mathcal S_r(x,\gamma_p.x)$, i.e. such that there is an infinite number of elements $\gamma_i\in \Gamma$ such that $\xi\in \mathcal S_r(x,\gamma_i.x)$. For this, we choose a sequence $g_i\in \Gamma$ such that $g_i$ converges radially to $\xi$. Let $r>0$ be such that all $g_i.x$ are at distance at most $r$ from the half geodesic $[x\xi)$. We then have $\xi \in \mathcal S_r(x,g_i.x)$ for all $i\in \N$, hence the result.
\end{proof}

\subsection{Gromov distance} \label{subsec - gromov distance}

%Let $C=Co(\Lambda)\setminus \Lambda$.\\
We denote by $\Lambda^{(2)}$ the pairs of distinct points of $\Lambda$.\\
 The Gromov product of three points $x,y,z\in C(\Lambda)$ is: \[(x\vert y)_z=\frac{1}{2}(d_{\h^{p,q}}(x,z)+d_{\h^{p,q}}(y,z)-d_{\h^{p,q}}(x,y)).\]  It extends to a continuous function on $\Lambda^{(2)}\times C(\Lambda)$, and we have $\forall\xi,\eta\in\Lambda^{(2)}\, , \forall x\in C(\Lambda)$: \[(\xi |\eta)_x=\frac{1}{2}\ln\left|\frac{2\langle\txi\vert \tx\rangle_{p,q+1}\langle \tx\vert\teta\rangle_{p,q+1}}{\langle \txi\vert\teta\rangle_{p,q+1}}\right| .\]
Note that $(\xi |\eta)_x=\frac{1}{2}(\beta_\xi(x,y)+\beta _\eta(x,y))$ for any $y\in (\xi\eta)$.\\
 For $x\in C(\Lambda)$ and $\xi,\eta \in \Lambda^{(2)}$, we set $d_x(\xi,\eta)=e^{-(\xi\vert\eta)_x}$. The explicit formula is:
 \[ d_x(\xi,\eta)=\sqrt{ \left| \frac{\langle \txi\vert\teta\rangle_{p,q+1}}{2\langle\txi\vert \tx\rangle_{p,q+1}\langle \tx\vert\teta\rangle_{p,q+1}} \right|}.\]
 {\color{black} Since $\Lambda$ is acausal, the term in the absolute value is always negative, so
 \[ d_x(\xi,\eta)=\sqrt{ \frac{-\langle \txi\vert\teta\rangle_{p,q+1}}{2\langle\txi\vert \tx\rangle_{p,q+1}\langle \tx\vert\teta\rangle_{p,q+1}} }.\]}
 
 Note that we always have $d_x(\xi,\eta)\leq 1$, and $d_x(\xi,\eta)=1$ if and only if $x\in (\xi\eta)$. Indeed, if $x\in C(\Lambda)$ and $\xi,\eta\in \Lambda$, then the affine subspace spanned by $x,\xi,\eta$ is a totally geodesic copy of $\h^2$, so this follows from the fact that in $\h^2$, the distance $d_x$ is half of the chordal distance, when $x$ is seen as the centre of the unit disk.
 
We remark that for all $y,z,x,x'$, we have $|(z|y)_x-(z|y)_{x'}| \leq 2d_{\h^{p,q}}(x,x')+2k_\G$, so  the functions $d_x$ and $d_{x'}$ are bi-Lipschitz equivalent.

The function $d_x$ is symmetric and $d_x(\xi,\eta)=0\iff \xi=\eta$, however it is not necessarily a distance. Just as for the Lorentzian distance on $C(\Lambda)$, we have a weak form of the triangle inequality which will be of some use.
 
 \begin{lemma} \label{triangle_boundary} There is a constant $\lambda_\G\ge 1$ such that:
\[ \forall x\in C(\Lambda)~ \forall \xi,\eta,\tau \in \Lambda ~ d_x(\xi,\eta)\le \lambda_\G (d_x(\xi,\tau) + d_x(\tau, \eta)).\]
\end{lemma} 

\begin{proof} It is enough to show the inequality when $\xi,\eta,\tau$ are pairwise distinct ($\lambda_\G=1$ gives the inequality when it is not the case).\\

Consider the function $F:(x,\xi,\eta,\tau)\mapsto \frac{d_x(\xi,\tau)+d_x(\tau,\eta)}{d_x(\xi,\eta)}$ defined on $C(\Lambda)\times \Lambda^{(3)}$ where $\Lambda^{(3)}$ is the set of distinct triples of points of $\Lambda$. This function is $\Gamma$-invariant, and the action of $\Gamma$ on $\Lambda^{(3)}$ is co-compact, so it is enough to see that $x\mapsto F(x,\xi,\eta,\tau)$ is bounded from below for fixed $(\xi,\eta,\tau)\in \Lambda^{(3)}$.\\
Assume that it is not the case, then one can find a sequence $x_i\in C(\Lambda)$ such that $F(x_i,\xi,\eta,\tau)\to 0$. The expression of $F$ is:

\[ F(x,\xi,\eta,\tau) = \sqrt{\frac{\langle\xi\vert\tau\rangle_{p,q+1}\langle\eta\vert x\rangle_{p,q+1}}{\langle\xi\vert\eta\rangle_{p,q+1}\langle\tau\vert x\rangle_{p,q+1}}}+ \sqrt{\frac{\langle\eta\vert\tau\rangle_{p,q+1}\langle\xi\vert x\rangle_{p,q+1}}{\langle\xi\vert\eta\rangle_{p,q+1}\langle\tau\vert x\rangle_{p,q+1}}}.\]

The fact that $F(x_i,\xi,\eta,\tau)\to 0$ implies that $x_i\to \eta$ (first term of $F$) and $x_i\to \xi$ (second term of $F$). This is impossible because $\xi\ne \eta$.

\end{proof}

{\color{black}{
\begin{defi}
For $x\in C(\Lambda)$, $\xi\in \Lambda$ and $r>0$, we consider 
\[B_x(\xi,r)=\{\eta\in \Lambda | d_x(\xi, \eta)\leq r\}\] and call this set a ball on the boundary of $\h^{p,q}$.  
\end{defi}

}}
{\color{black}{\emph{Geometric interpretation of $d_x$:} For $x\in C(\Lambda)$, we have defined the pseudo-spherical domain to be $\partial U(x) = \{\xi\in \partial \h^{p,q} \, |\, \langle x| \xi\rangle_{p,q+1} \neq \emptyset\}$ and seen that it is isometric to the pseudo-Riemannian sphere $\Ss^{p-1,q}$.

Let $\tx=(0,...,0,0,1)$ and $\hat{\xi}, \hat{\eta}$ two points of $\Ss^{p-1,q}\subset \R^{p,q}$. We can consider $\txi = (\hat{\xi},1)$ and $\teta = (\hat{\eta},1)$ the associated points of $\R^{p,q+1}$ lying in $\partial \mathcal H^{p,q}$. Finally let $\xi,\eta$ be the corresponding images in $\partial \h^{p,q}$.
\begin{align*}d_x( \xi, \eta) &= \sqrt{ \frac{-\langle \txi\vert\teta\rangle_{p,q+1}}{2\langle\txi\vert \tx\rangle_{p,q+1}\langle \tx\vert\teta\rangle_{p,q+1}} }\\ 
&=\sqrt{\frac{1}{2}(1- \langle\hat{\xi}| \hat{\eta}\rangle_{p,q})}\\
&= \frac{1}{2} \sqrt{\pscal{\hat\xi-\hat\eta}{\hat\xi-\hat\eta}_{p,q}}.
\end{align*}
This shows that the balls on the boundary are intersections of $\Lambda$ with interiors of quadrics of signature $(p-1,q)$.\\
%Since $\Lambda$ is acausal, for all $\eta,xi\in \Lambda$ we have: 
%$$\frac{\langle \txi\vert\teta\rangle_{p,q+1}}{2\langle\txi\vert \tx\rangle_{p,q+1}\langle \tx\vert\teta\rangle_{p,q+1}}<0.$$
%And therefore, 
%$$d_x( \xi, \eta) = \sqrt{\frac{1}{2} (1-\langle\hat{\xi}| \hat{\eta}\rangle_{n,1}) }.$$

%And $$\|\hat{\eta}-\hat{\xi}\|^2_{n,1} = \|\hat{\eta}\|^2_{n,1}+\|\hat{\eta}\|^2_{n,1} -2\langle \hat{\eta}| \hat{\xi}\rangle_{n,1}  =2(1- \langle \hat{\eta}| \hat{\xi}\rangle_{n,1}).$$
%We see that if $\hat{\xi}$ is in the light cone of $\hat{\eta}$ if and only if $\langle \hat{\eta}| \hat{\xi}\rangle_{n,1}>1$.  So, for $\xi,\eta\in \Lambda^{(2)}$ we have: 
%$$2d_x(\xi,\eta) =   \|\hat{\eta}-\hat{\xi}\|_{n,1} $$

 }}

Given $x\in C(\Lambda)$ and $\xi,\eta\in \Lambda$, the quantity $d_x(\xi,\eta)$ can be computed from the lengths of the side of any triangle whose vertices are $x$, a point of $[x\xi)$ and a point of $[x\eta)$. We want to stress that the following lemma can be seen as a purely hyperbolic geometry result, since all the points are on a unique $\Hyp^2\subset \h^{p,q}$.
\begin{lemma} \label{computation_distance_boundary} Let $x,y,z\in C(\Lambda)$ and $\xi,\eta\in \Lambda$.  If  $y\in [x\eta)$ and $z\in [x\xi)$, then:
\[ d_x(\xi,\eta)^2=\frac{\cosh d_{\h^{p,q}}(y,z)- \cosh(d_{\h^{p,q}}(x,y)-d_{\h^{p,q}}(x,z))}{2\sinh d_{\h^{p,q}}(x,y) \sinh d_{\h^{p,q}}(x,z)}.\]
\end{lemma}

\begin{proof}
We denote by $\xi(u)$ (resp. $\eta(u)$) the geodesic joining $x$ and $\xi$ (resp. $\eta$). The equations are:
\begin{equation*}
\txi(u) = e^{-u} \tx - \frac{\sinh u}{\langle \tx\vert \txi\rangle_{p,q+1}}\txi\quad \text{and}\quad \teta(u) = e^{-u} \tx - \frac{\sinh u}{\langle \tx\vert \teta\rangle_{p,q+1}}\teta.
\end{equation*}

Since we have $\ty=\teta(u)$ (where $u=d_{\h^{p,q}}(x,y)$) and $\tz=\txi(v)$ (where $v=d_{\h^{p,q}}(x,z)$), we find:
\begin{align*}
\cosh d_{\h^{p,q}}(y,z) &= |\langle \ty\vert \tz\rangle_{p,q+1} |\\
&=\left| \langle e^{-u} x - \frac{\sinh u}{\langle \tx\vert \teta\rangle_{p,q+1}}\teta \vert  e^{-v} \tx - \frac{\sinh v}{\langle \tx\vert \txi\rangle}\txi \rangle_{p,q+1}\right| \\
&= \left|e^{-u-v} - e^{-v}\sinh u - e^{-u} \sinh v + \sinh u \sinh v \frac{\langle \txi\vert\teta\rangle_{p,q+1}}{\langle \txi\vert \tx\rangle_{p,q+1} \langle \tx\vert \teta\rangle_{p,q+1}}\right|\\
&= |-\cosh(u-v) - 2d_x(\xi,\eta)^2 \sinh u \sinh v|\\
&= \cosh(d_{\h^{p,q}}(x,y)-d_{\h^{p,q}}(x,z)) + 2d_x(\xi,\eta)^2 \sinh d_{\h^{p,q}}(x,y) \sinh d_{\h^{p,q}}(x,z).
\end{align*}
\end{proof}

The two following technical lemmas compare balls on the boundary with shadows. 
\begin{coro} \label{ball_subset_shadow} Let $\xi\in \Lambda$ and $r\in (0,1)$. If $y\in [x\xi)$ is such that $d_{\h^{p,q}}(x,y)=-\ln r$, then $B_x(\xi,r)\subset \mathcal S_{\ln 6}(x,y)$.
\end{coro}
\begin{proof}
Let $\eta\in B_x(\xi,r)$, and let $z\in [x\eta)$ be such that $d(x,z)=d(x,y)=-\ln r$. We find $\cosh d_{\h^{p,q}}(y,z)=1+2(\sinh d_{\h^{p,q}}(x,y)d_x(\xi,\eta))^2\leq 3$, hence $d_{\h^{p,q}}(y,z)\leq \Ach\, 3 \leq \ln 6$, and $\eta\in \mathcal S_{\ln 6}(x,y)$.
\end{proof}

\begin{coro} \label{shadow_subset_ball} Let $\xi\in \Lambda$, $r\in (0,1)$ and $t>0$. If $y\in [x\xi)$ is such that $d_{\h^{p,q}}(x,y)=t+k+\frac{\ln 2}{2} -\ln\frac{r}{4} $, then $\mathcal S_t(x,y)\subset B_x(\xi,r)$.
\end{coro}
\begin{proof}
Let $y\in [x\xi)$ and $\eta\in \mathcal S_t(x,y)$. Given $z\in [x\eta)$  such that $d(z,y)<t$, we find:
\[ d_x(\xi,\eta)^2=\frac{\cosh d(y,z) - \cosh( d_{\h^{p,q}}(x,y)-d_{\h^{p,q}}(x,z))}{2\sinh d_{\h^{p,q}}(x,y) \sinh d_{\h^{p,q}}(x,z)} < \frac{e^t}{\sinh d_{\h^{p,q}}(x,y) \sinh d_{\h^{p,q}}(x,z)}. \]
Since $d_{\h^{p,q}}(x,z)\geq d_{\h^{p,q}}(x,y)-t-k$, we find that $d_{\h^{p,q}}(x,z)\geq \frac{\ln 2}{2}$, and $d_{\h^{p,q}}(x,y)\geq \frac{\ln 2}{2}$, hence $\sinh d_{\h^{p,q}}(x,y)\geq \frac{e^{d_{\h^{p,q}}(x,y)}}{4}$ and $\sinh d_{\h^{p,q}}(x,z)\geq \frac{e^{d_{\h^{p,q}}(x,z)}}{4}$ (here we use the fact that $u\geq \frac{\ln 2}{2}$ implies $\sinh u\geq \frac{e^u}{4}$). Finally,
\[ d_x(\xi,\eta)^2 < 16 e^{t-d_{\h^{p,q}}(x,y)-d_{\h^{p,q}}(x,z)}\leq 16 e^{2t+k-2d_{\h^{p,q}}(x,y)}\]
In order to have $\mathcal S_t(x,y)\subset B_x(\xi,r)$, it is enough to have $4e^{t+\frac{k}{2}} e^{-d_{\h^{p,q}}(x,y)}\leq r$, which is guaranteed by the condition on $d_{\h^{p,q}}(x,y)$.
\end{proof}

%Note that radial convergence can be expressed in terms of the Gromov distance $d_x$.
%
%\begin{lemma} Let $\xi\in\Lambda$, and let $\gamma_p$ be a sequence in $\Gamma$ such that $\gamma_p.x$ converges radially to $\xi$. Denote by $\eta_p\in \Lambda$ a point that is causally related in $\partial \AdS^{n+1}\setminus \partial x^*$ to the endpoint of the half-geodesic based at $x$ passing through $\gamma_p.x$. Then the sequence $e^{d_{\h^{p,q}}(x,\gamma_p.x)}d_x(\xi,\eta_p)$ is bounded.
%\end{lemma} 
% 
%\begin{proof}
%Let $y_p\in [x\xi)$ be the point such that $d(x,y_p)=d(x,\gamma_p.x)$. Since $\gamma_p.x$ converges radially to $\xi$, we have that $d_{\h^{p,q}}(\gamma_p.x,y_p)$ is bounded. Let $z_p$ be the point on $[x\eta_p)$ such that $d_{\h^{p,q}}(x,z_p)=d_{\h^{p,q}}(x,\gamma_p.x)$. \\
%A simple calculation (see the proof of Theorem \ref{shadow_lemma}) shows that $\gamma_p.x$ is causally related to $z_p$. By Lemma \ref{computation_distance_boundary}, we see that $e^{d_{\h^{p,q}}(x,\gamma_p.x)}d_x(\xi,\eta_p)\leq K \sqrt{\cosh d_{\h^{p,q}}(y_p,z_p)}$ for some constant $K$. Since $z_p$ is causally related to $\gamma_p.x$, we also have that $d_{\h^{p,q}}(y_p,z_p)$ is bounded, hence the result.
%\end{proof}

\subsection{Cross-ratios}
In the last section, we will use the following lemma proven by J-P. Otal \cite{cinca1992geometrie} for Hadamard spaces. Recall that the cross-ratio of four boundary points is defined by 
\begin{equation}\label{eq - cross ratio}
[a,b,c,d]_{\h^{p,q}}:=\frac{d_x(a,c)d_x(b,d)}{d_x(a,d) d_x(b,c)}.
\end{equation}
It is independent of $x$. Indeed, computing  $\left(\frac{d_x(a,c)d_x(b,d)}{d_x(a,d) d_x(b,c)}\right)^2$ we find
\begin{align*}
&\frac{\pscal{\tilde a}{\tilde c}_{p,q+1} }{\pscal{\tilde a}{\tx}_{p,q+1} \pscal{\tilde c}{\tx}_{p,q+1} } \frac{\pscal{\tilde b}{\tilde d}_{p,q+1} }{\pscal{\tilde b}{\tx}_{p,q+1} \pscal{\tilde d}{\tx}_{p,q+1} } \frac{\pscal{\tilde a}{\tx}_{p,q+1} \pscal{\tilde d}{\tx}_{p,q+1} }{\pscal{\tilde a}{\tilde d}_{p,q+1} } \frac{\pscal{\tilde b}{\tx}_{p,q+1} \pscal{\tilde c}{\tx}_{p,q+1} }{\pscal{\tilde b}{\tilde c}_{p,q+1} }\\
&=  \frac{\pscal{\tilde a}{\tilde c}_{p,q+1}\pscal{\tilde b}{\tilde d}_{p,q+1}  }{\pscal{\tilde b}{\tilde c}_{p,q+1}\pscal{\tilde a}{\tilde d}_{p,q+1}}.
\end{align*}

Recall that every element $\g$ in a $\h^{p,q}$-convex cocompact group has  north-south dynamics, Proposition \ref{prop -north south dynamic}. The fixed points are denoted by $\g^\pm$ and the axis by $(\g^- \g^+)$ which is  a spacelike geodesic of $\h^{p,q}$, invariant by $\g$. The latter acts by translation on this geodesic, and we call $\ell_{\h^{p,q}}(\g)$ its translation length: $\ell_{\h^{p,q}}(\g) := d_{\h^{p,q}}(\g x, x) $ for any $x\in (\g^-\g^+)$.

\begin{lemme}
Let $\g\in \G$. If $\g^-,\g^+\in\Lambda$ are its repelling and attracting fixed points, then for any $\xi \in \Lambda\setminus\{g^{\pm}\}$:
$$[\g^-, \g^+,\g\xi,\xi]_{\h^{p,q}} = e^{\ell_{\h^{p,q}}(\g)}.$$
\end{lemme}
\begin{proof}
The computation above gives 
\begin{eqnarray*}
[\g^-, \g^+,\g(\xi),\xi]^2_{\h^{p,q}}&=&\frac{\pscal{\tilde \g^-}{\g\tilde \xi}_{p,q+1}\pscal{\tilde \g^+}{\txi}_{p,q+1}  }{\pscal{\tilde \g^+}{\g\txi}_{p,q+1}\pscal{\tilde \g^-}{\txi}_{p,q+1}}.\\									
\end{eqnarray*}
Let $P$ be the plane in $\R^{p+q+1}$ such that $\proj(P)\cap \h^{p,q} = (\g^-\g^+).$
It is of signature $(+,-)$, so its orthogonal $P^\perp $ (for $\pscalb{\cdot}{\cdot}$) satisfies 
$P\oplus P^\perp = \R^{p+q+1}$. It is clear from the definition of the Buseman functions $\beta_{\g^\pm}$ that for $x\in  (\g^-\g^+),$ we have 
$$\left| \frac{\pscal{\tilde \g^-}{\g\tilde x}_{p,q+1}}{\pscal{\tilde \g^-}{\tilde x}_{p,q+1}  }\right|=e^{\beta_{\g^-} (\g x ,x)}= e^{\ell_{\h^{p,q}}(\g)}.$$
$$\left|\frac{\pscal{\tilde \g^+}{\tilde x}_{p,q+1}}{\pscal{\tilde \g^+}{\g\tilde x}_{p,q+1}  }\right|=e^{\beta_{\g^+} (x ,\g x) }= e^{\ell_{\h^{p,q}}(\g)}.$$
Let $h\in P^\perp$ then 
$$\left|\frac{\pscal{\tilde \g^-}{\g(\tilde x+\tilde h)}_{p,q+1}}{\pscal{\tilde \g^-}{\tilde x+\tilde h}_{p,q+1}  }\right|= \left|\frac{\pscal{\tilde \g^-}{\g\tilde x}_{p,q+1}}{\pscal{\tilde \g^-}{\tilde x}_{p,q+1}  }\right|,$$
since $\g$ preserve $P^\perp$. Hence for all $x\in \R^{p+q+1}$ we have 
$$\beta_{\g^-} (\g x ,x)= \ell_{\h^{p,q}}(\g).$$
And by the same argument 
$$\beta_{\g^+} (x ,\g x) = \ell_{\h^{p,q}}(\g).$$
It follows by projectivizing that 
$$\left|\frac{\pscal{\tilde \g^-}{\g\tilde x}_{p,q+1}}{\pscal{\tilde \g^-}{\tilde x}_{p,q+1}  } \right|= \left|\frac{\pscal{\tilde \g^-}{\g\tilde \xi}_{p,q+1}}{\pscal{\tilde \g^-}{\tilde \xi}_{p,q+1}  }\right|=e^{\ell_{\h^{p,q}}(\g)} .$$
Therefore
\begin{eqnarray*}
[\g^-, \g^+,\g\xi,\xi]^2_{\h^{p,q}}&=&e^	{2\ell_{\h^{p,q}}(\g)}.				
\end{eqnarray*}

\end{proof}

\section{Conformal densities} \label{sec - conformal density}
This section is devoted to the definition and properties of conformal densities in the pseudo-Riemannian setting. Two important proofs,  the existence of conformal densities and their ergodicity, are postponed to the appendix as the proofs of the corresponding results in metric geometry can be carried out mutatis mutandis (the only difference is the presence of the additive constant $k_\G$ in the triangle inequality for $d_{\h^{p,q}}$). This should result in a clearer outline of the theory. However technical difficulties sometimes appear due to the pseudo-Riemannian context, this is notably the case for the Shadow Lemma (Theorem \ref{shadow_lemma}). 

Once again, a $\h^{p,q}$-convex cocompact group $\G\subset\PO(p,q+1)$ is fixed, and $\Lambda\subset \partial \h^{p,q}$ is its limit set. We will mainly follow the notes of J.-F. Quint \cite{quint2006overview}. Another  reference for this notion in hyperbolic geometry is the book of P. Nicholls \cite{nicholls1989ergodic}. The strategy is the following: 
\begin{itemize}
\item Show the existence of a conformal density of dimension $\delta_{\h^{p,q}}(\G)$ by the Patterson-Sullivan method.
\item Show that conformal densities have no atoms and prove the Shadow Lemma.
\item Show the ergodicity of conformal densities and conclude on the uniqueness of the density. 
\end{itemize}

\begin{defi}
A conformal density of dimension $s$ is a family of measures $(\nu_x)_{x\in C(\Lambda)}$ {\color{black} on $\Lambda$} satisfying the following conditions:
\begin{enumerate}
\item $\forall \g\in \G$, $\g^*\nu_x = \nu_{\g x}$ (where $  \g^*\nu (E) = \nu(\g^{-1} E))$
\item $\forall x,y\in C(\Lambda)$, $\frac{d \nu_x}{d\nu_{y}} (\xi) = e^{-s\beta_\xi (x,y)} $
\item $\supp(\nu_x)=\Lambda$\label{eq -support}
\end{enumerate}
\end{defi}

As we said, by the classical Patterson-Sullivan construction we can build conformal density, therefore we postpone the proof of the following theorem to the appendix: 
\begin{theorem}
There exists a conformal density of dimension $\delta_{\h^{p,q}}(\G)$.
\end{theorem}
We will denote by $(\mu_x)_{x\in C(\Lambda)}$ this conformal density, called the Patterson-Sullivan density.

\subsection{Properties of conformal densities}
\subsubsection{Atomic part}

As a consequence of radial convergence, conformal densities have no atoms. The proof of the following proposition  is very similar to the Riemannian case. 

\begin{prop}
Let $(\nu_x)_{x\in C(\Lambda)}$ be a conformal density of dimension $s\in\R $. For all $x\in C(\Lambda)$, $\nu_x$ has no atom. 
\end{prop}
\begin{proof}
Let us fix $x\in C(\Lambda)$. Assume by contradiction that there exists $\xi\in \Lambda$ with $\nu_x(\xi)>0$. \\
Note that $s=0$ is impossible, because  $\nu_x$ would be an invariant under $\G$ on $\Lambda$, and the action of $\G$ on $\Lambda$ is topologically conjugate to the actions on its Gromov boundary.\\
First, let us assume that there is $\g\in \G$ such that $\g \xi =\xi$. 
We have 
\begin{eqnarray*}
\nu_x(\xi) &= &\nu_x(\g^i \xi) \\
					&=&  \nu_{\g^{-i} x} (\xi)\\
					&=& e^{-s \beta_\xi (\g^{-i} x, x)} \nu_x(\xi)
\end{eqnarray*}
As $\xi$ is assumed to be fixed by $\g$ we have $\beta_\xi (\g^{-i} x, x)= \beta_\xi (x,\g^i x)$, and then $ \lim_{i\tv +\infty} e^{-s \beta_\xi (\g^{-i} x, x)} =+\infty$ if $s> 0$, and  $ \lim_{i\tv +\infty} e^{-s \beta_\xi (\g^{-i} x, x)} =0$ if $s< 0$. Both cases lead to a contradiction.\\
Therefore  $Stab_{\G}(\xi) = \Id$
\begin{eqnarray*}
\nu_x(\Lambda) &\geq &  \sum_{\g\in \G} \nu_x(\g^{-1} \xi)\\
							&\geq & \sum_{\g\in \G} e^{-s \beta_\xi(\g x,x)}\nu_x(\xi).
\end{eqnarray*}
Every limit point is radial, hence we can find a sequence $\g_i \in \G$ such that $\beta_\xi(x,\g_i x) \tv +\infty $, hence $\beta_\xi(\g_i x,x) \tv -\infty $, and therefore $\nu_x(\Lambda)=+\infty$ if $s>0$. This is a contradiction with the fact that $\nu_x$ is a finite measure. \\
 If $s<0$, consider any sequence $(\g_i)$ such that $\g_ix$ converges in $\overline{\h^{p,q}}$ to some $\eta\neq \xi$. We then have $\beta_\xi(\gamma_ix,x)\to+\infty$, which also gives a contradiction with the fact that $\nu_x$ is a finite measure. 
\end{proof}

{A consequence of the non existence of atoms is the fact that small pseudo-Riemannian balls on the boundary have small mass.

\color{red}\begin{coro} \label{coro - petite boule petite masse} Let $(\nu_x)_{x\in C(\Lambda)}$ be a conformal density. For all $x\in C(\Lambda)$ and $\e>0$, there is $r>0$ such that $\nu_x(B_x(\xi,r))<\e$ for all $\xi\in\Lambda$.
\end{coro}
\begin{proof} Let $x\in C(\Lambda)$ and $\e>0$. Since $\nu_x$ has no atoms, for all $\xi\in\Lambda$ there is $r_\xi>0$ such that $\nu_x(B_x(\xi,r_\xi))<\e$. Consider $\xi_1,\dots,\xi_k$ such that $\Lambda=\bigcup_{i=1}^kB_x(\xi_i,\frac{r_{\xi_i}}{2\lambda_\G})$, and set $r=\min \frac{r_{\xi_i}}{2\lambda_\G}$. If  $\xi\in B_x(\xi_i,\frac{r_{\xi_i}}{2\lambda_\G})$, then $B_x(\xi,r)\subset B_x(\xi_i,r_{\xi_i})$, hence $\nu_x(B_x(\xi,r))<\e$.
\end{proof}

As a consequence, we find that large shadows have a large mass.
\begin{coro} \label{coro - grande ombre grande masse} Let $(\nu_x)_{x\in C(\Lambda)}$ be a conformal density. For all $x\in C(\Lambda)$ and $\e>0$, there is $R>0$ such that 
$\nu_x(\mathcal S_R(y,x) ) \geq \nu_x(\Lambda) -\epsilon$ for all $y\in C(\Lambda)$.
\end{coro}
\begin{proof} Because of Corollary \ref{coro - petite boule petite masse}, it is enough to show that for all $x\in C(\Lambda)$ and $\e>0$, there is $R>0$ such that all $\xi,\eta\in\Lambda\setminus \mathcal S_R(y,x)$ satisfy $d_x(\xi,\eta)<\e$.\\
Assume that it is not the case. Then is $\e>0$ such that, for all $R>0$, we can find $y_R\in C(\Lambda)$ and $\xi_R,\eta_R\notin \mathcal S_R(y_R,x)$ with $d_x(\xi_R,\eta_R)\geq\e$.\\
Choose a sequence $R_i\to+\infty$ such that $y_{R_i}\to y\in \overline{C(\Lambda)}$, $\xi_{R_i}\to\xi\in\Lambda$ and $\eta_{R_i}\to\eta\in\Lambda$. Note that since $\mathcal S_R(y_R,x)\neq \Lambda$ for all $R>0$, we have that $d_{\h^{p,q}}(x,y_R)\geq R$, so $y\in \Lambda$.\\
Assume that $y\neq \xi$, and let $z\in (y\xi)$. We can find a sequence $z_i\to z$ such that $z_i\in [y_{R_i}\xi_{R_i})$ for all $i\in \N$. Since $\xi_{R_i}\notin\mathcal S_{R_i}(y_{R_i},x)$, we have $d_{\h^{p,q}}(x,z_i)\geq R_i\to +\infty$, so $z\in \Lambda$, which is absurd.\\
We now have that $y=\xi$, and similarly $y=\eta$, hence $d_x(\xi,\eta)=0$ which is a contradiction with $d_x(\xi_R,\eta_R)\geq\e$ for all $R$.
\end{proof}
}

\subsubsection{Shadow Lemma}\label{sec - shadow lemma}

The Shadow Lemma is one of the most important results in Patterson-Sullivan theory (i.e. the study of conformal densities). It  gives an estimate of the measures of shadows, which we will later translate into an estimate of the measures of balls on the boundary (Theorem \ref{measure_balls}). %We will encounter  the first difficulties inherent to the Lorentzian context.

\begin{theorem}[Shadow lemma]\label{shadow_lemma}
Let $(\nu_x)_{x\in C(\Lambda)}$ be a conformal density of dimension $s$, and $x\in C(\Lambda)$. There is $r_0>0$ such that, for all $r>r_0$, there is $C(r)>0$ satisfying:
\[ \frac{1}{C(r)}e^{-sd(x,\gamma.x)} \le \nu_x(\mathcal S_r(x,\gamma.x)) \le C(r) e^{-sd(x,\gamma.x)}.\]
\end{theorem}

\begin{proof}%[Proof of Theorem \ref{shadow_lemma}] 
%{\color{black}The beginning of the proof, which gives the right hand side inequality, is standard (see the proof of Lemma 4.10 in \cite{quint2006overview}).\\}
Let $r>0$ and $\gamma\in\Gamma$.

\begin{align*} 
\nu_x(\mathcal S_r(x,\gamma.x)) &= \nu_x(\gamma\mathcal S_r(\gamma^{-1}.x,x)) \\
~~&= \nu_{\gamma^{-1}.x}(\mathcal S_r(\gamma^{-1}.x,x))\\
~~&= \int_{\mathcal S_r(\gamma^{-1}.x,x)} e^{-s \beta_\xi(\gamma^{-1}.x,x)} d\nu_x(\xi).
\end{align*}
If $s\ge 0$, Lemma \ref{shadow_1} now implies that:
\[ \nu_x(\mathcal S_r(\gamma^{-1}.x,x)) e^{-sk_\G} e^{-sd_{\h^{p,q}}(x,\gamma.x)} \le \nu_x(\mathcal S_r(x,\gamma.x)) \le \nu_x(\Lambda) e^{2s(r+k_\G)} e^{-sd_{\h^{p,q}}(x,\gamma.x)} . \]
If $s<0$, we get:
\[ \nu_x(\mathcal S_r(\gamma^{-1}.x,x)) e^{-2s(r+k_\G)} e^{-sd_{\h^{p,q}}(x,\gamma.x)} \le \nu_x(\mathcal S_r(x,\gamma.x)) \le \nu_x(\Lambda) e^{sk_\G} e^{-sd_{\h^{p,q}}(x,\gamma.x)} . \]

This gives us the right hand side inequality. In order to prove the other inequality,  we now have  to show is that there are $\e>0$ and $r_0>0$ such that $\nu_x(\mathcal S_r(\gamma^{-1}.x,x)) \ge \e$ for all $r>r_0$ and $\gamma\in \Gamma$. {\color{red} This is a consequence of Corollary \ref{coro - grande ombre grande masse}.}

\end{proof}

Using the convex cocompactness, we can generalize Theorem \ref{shadow_lemma} to find an estimate of the measure of all shadows.

\begin{theorem} \label{shadow_generalised}
Let $\nu$ be a conformal density of dimension $s$, and let $x\in C(\Lambda)$. There is $r'_0>0$ such that for all $r\geq r'_0$, there is a constant $C'(r)>0$ satisfying:
\[ \frac{1}{C'(r)} e^{-sd_{\h^{p,q}}(x,y)}\leq \nu_x(\mathcal S_r(x,y))\leq C'(r) e^{-sd_{\h^{p,q}}(x,y)}\]
for all $y\in C(\Lambda)$.
\end{theorem}

\begin{proof}
Let $r_0>0$ be given by Theorem \ref{shadow_lemma}. The action of $\Gamma$ on $C(\Lambda)$ being cocompact, we can choose $R>0$ such that the ball $B_{C(\Lambda)}(x,R)$ contains a fundamental domain for the action on $C(\Lambda)$. We will show that $r'_0=r_0+R+k_\G$ and $C'(r)=\max(C(r+R+k_\G),C(r-R-k_\G))e^{s(R+k_\G)})$ fulfill the requirements.\\
Let $y\in C(\Lambda)$ and $r>r'_0$. There is $\gamma\in \Gamma$ such that $\gamma^{-1}.y\in B_{C(\Lambda)}(x,R)$, hence $d_{\h^{p,q}}(y,\gamma.x)<R$. Let $\xi\in \mathcal S_r(x,y)$, and $z\in [x\xi)$ such that $d_{\h^{p,q}}(y,z)<r$. We find $d_{\h^{p,q}}(\gamma.x,z)\leq d_{\h^{p,q}}(\gamma.x,y)+d_{\h^{p,q}}(y,z)+k_\G<R+r+k_\G$. It follows that $\mathcal S_r(x,y)\subset \mathcal S_{r+R+k_\G}(x,\gamma.x)$.\\
A similar computation shows that $\mathcal S_r(x,y)\supset \mathcal S_{r-R-k_\G}(x,\gamma.x)$.\\
By Theorem \ref{shadow_lemma}, we get:
\begin{align*}
\nu_x(\mathcal S_{r-R-k_\G}(x,\gamma.x)) &\leq & \nu_x(\mathcal S_r(x,y)) &\leq  \nu_x(\mathcal S_{r+R+k_\G}(x,\gamma.x))\\
\frac{1}{C(r-R-k_\G)}e^{-s d_{\h^{p,q}}(x,\gamma.x)} &\leq & \nu_x(\mathcal S_r(x,y)) &\leq  C(r+R+k_\G)e^{-sd_{\h^{p,q}}(x,\gamma.x).}
\end{align*}
We also have $d_{\h^{p,q}}(x,\gamma.x)\leq d_{\h^{p,q}}(x,y)+d_{\h^{p,q}}(y,\gamma.x)+k_\G\leq d_{\h^{p,q}}(x,y)+R+k_\G$, as well as $d_{\h^{p,q}}(x,\gamma.x)\geq d_{\h^{p,q}}(x,y)-R-k_\G$. It follows that:
\[ \frac{e^{-s (R+k_\G)}}{C(r-R-k_\G)}e^{-s d_{\h^{p,q}}(x,y)} \leq  \nu_x(\mathcal S_r(x,y)) \leq  e^{s (R+k_\G)}C(r+R+k_\G)e^{-s d_{\h^{p,q}}(x,y)}. \]
\end{proof}

We follow Sullivan's work  \cite{sullivan1979density} and show that the Shadow Lemma implies the following:

\begin{theorem} \label{measure_balls}
Let $\nu$ be a conformal density of dimension $s$, and let $x\in C(\Lambda)$. There is $c>0$ such that for all $\xi\in \Lambda$, $r\in (0,1)$, we have: \[\frac{\nu_x(B_x(\xi,r))}{r^s} \in \left[\frac{1}{c},c\right]\]

\end{theorem}

\begin{proof}
Let $r'_0>0$ be given by Theorem \ref{shadow_generalised}, and $t=\max(r'_0,\ln 6)$.\\
Let $\xi\in \Lambda$ and $r\in (0,1)$. According to Corollary \ref{ball_subset_shadow}, by letting $y_1\in [x\xi)$ be the point such that $d_{\h^{p,q}}(x,y_1)=-\ln r$, we find $B_x(\xi,r)\subset \mathcal S_{\ln 6}(x,y_1)\subset \mathcal S_t(x,y_1)$. By Theorem \ref{shadow_generalised}, we get $\nu_x(B_x(\xi,r))\leq C'(t) e^{-sd_{\h^{p,q}}(x,y_1)} =C'(t) r^{s}$.\\
Now let $y_2\in [x\xi)$ be such that $d_{\h^{p,q}}(x,y_2)=t+k_\G+\frac{\ln 2}{2} -\ln\frac{r}{4}$. According to Corollary \ref{shadow_subset_ball}, we have the inclusion $\mathcal S_t(x,y_2)\subset B_x(\xi,r)$. Theorem \ref{shadow_generalised} now gives us $\nu_x(B_x(\xi,r))\geq \frac{1}{C'(t)} e^{-sd_{\h^{p,q}}(x,y_2)}\geq \frac{1}{c} r^{s}$ where $c=C'(t) e^{s(t+k_\G+\frac{5\ln 2}{2})}$.
\end{proof}

\subsection{Uniqueness of conformal densities}

The proof of the uniqueness of the conformal density breaks into two very distinct parts: first we show that any conformal density has dimension $\delta_{\h^{p,q}}(\G)$, then we show that conformal densities  are ergodic. We keep the proof of the first step in this section, since it uses the tools of Lorentzian geometry elaborated in Section \ref{sec -geometric toolbox}. The second step, ergodicity, is  a straightforward generalization of the Riemannian case  (and quite technical). Its proof is postponed  to the appendix.

%\subsubsection{Uniqueness of dimension}
We follow the notes of Quint \cite{quint2006overview} and adapt it to our setting. 
\begin{prop} \label{s_bigger_than_delta}
If there is a non trivial conformal density of dimension $s$, then $s\geq \delta_{\h^{p,q}}(\Gamma)$.
\end{prop}

\begin{proof}
Let $\nu$ be a non trivial conformal density of dimension $s$. First, note that the left hand side of the inequality in the Shadow Lemma implies that $s\ge 0$: if it were not the case, the measure $\nu_x$ would be infinite.\\
Let $C$ and $r_0$ be given by Theorem \ref{shadow_lemma}, and let $r\geq r_0$. For $i\in \N$, we set $\Gamma_i=\{\gamma\in\Gamma \vert i\le d_{\h^{p,q}}(x,\gamma.x)<i+1\}$.\\
Let $p$ be the cardinal of $Z=\{\gamma\in \Gamma \vert d_{\h^{p,q}}(x,\gamma.x)\leq 1+4(r+k_\G)\}$. We will show that given $\xi\in\Lambda$ and $i\in\N$, there are at most $p$ elements $\gamma\in\Gamma_i$ such that $\xi\in\mathcal S_r(x,\gamma.x)$.\\
Let $\gamma,\gamma'\in\Gamma_i$ be such that $\xi \in \mathcal S_r(x,\gamma.x)\cap \mathcal S_r(x,\gamma'.x)$. There is $y\in [x\xi)$ such that $d_{\h^{p,q}}(\gamma.x,y)\le r$. Let us find an estimation for $d_{\h^{p,q}}(x,y)$:
\begin{align*}
d_{\h^{p,q}}(x,y) &\leq d_{\h^{p,q}}(x,\gamma.x)+d_{\h^{p,q}}(\gamma.x,y)+k_\G \leq i+1+r+k_\G\\
d_{\h^{p,q}}(x,y)&\geq d_{\h^{p,q}}(x,\gamma.x)-d_{\h^{p,q}}(y,\gamma.x)-k_\G \geq i-r-k_\G.
\end{align*}
\[i-r-k_\G\le d_{\h^{p,q}}(x,y)\le i+1+r+k_\G\]
 Similarly there is $z\in [x\xi)$ such that $d(\gamma'.x,z)\le r$ and: \[i-r-k_\G\le d_{\h^{p,q}}(x,z)\le i+1+r+k_\G.\]
  Since $y$ and $z$ lie in the same half geodesic $[x\xi)$, we see that  \[d_{\h^{p,q}}(y,z)=\pm (d_{\h^{p,q}}(x,y)-d_{\h^{p,q}}(x,z)).\] This shows that $d_{\h^{p,q}}(y,z)\le 1+2(r+k_\G)$. Finally, we find:
\[d_{\h^{p,q}}(\gamma.x,\gamma'.x)\le d_{\h^{p,q}}(\gamma.x,y)+d_{\h^{p,q}}(y,z)+d_{\h^{p,q}}(z,\gamma'.x)+2k_\G\le 1+4(r+k_\G).\] 
This means that $\gamma^{-1}\gamma'\in Z$, which shows the desired bound on the number of such elements of $\Gamma_i$.\\
{\color{black} If $a_i$ is the number of elements of $\Gamma_i$, we find:}

\begin{align*}
\nu_x(\Lambda) &\ge \nu_x\left(\bigcup_{\gamma\in\Gamma_i}\mathcal S_r(x,\gamma.x)\right) \\
&\ge \frac{1}{p}\sum_{\gamma\in\Gamma_i} \nu_x\left(S_r(x,\gamma.x)\right)\\
&\ge \frac{1}{pC}\sum_{\gamma\in\Gamma_i} e^{-sd_{\h^{p,q}}(x,\gamma.x)}\\
&\ge \frac{1}{pC} e^{-s(i+1)}a_i.
\end{align*}

Let $D=pC\nu_x(\Lambda)$, so that we find $a_i\le D e^{s(i+1)}$ for all $i$, and :
\[ \frac{1}{i}\ln(a_0+\cdots+a_i) \le \frac{1}{i}\ln(D(i+1)e^{s(i+1)}) \xrightarrow[i\to +\infty]{} s. \]
 Since $\delta_{\h^{p,q}}(\G)=\limsup \frac{1}{i}\ln(a_0+\cdots+a_i)$, we find that $s\ge \delta_{\h^{p,q}}(\G)$.

\end{proof}

Knowing the fact that every point of $\Lambda$ is radial, we can turn the inequality of Corollary \ref{s_bigger_than_delta} into an equality.

\begin{prop} \label{divergence_at_s} If there is a conformal density of dimension $s$, then the Poincaré series diverges at $s$.
\end{prop}

\begin{proof}
We will use an enumeration $\Gamma=\{\gamma_p\vert p\in\N\}$. By Lemma \ref{lemma_description_limit_set}, we have that: 
\begin{equation} \label{description_limit_set}\Lambda= \bigcup_{r>0}\bigcap_{N\in \N}\bigcup_{p\geq N} \mathcal S_r(x,\gamma_p.x)
\end{equation}
Assume that $\sum_{p=0}^{+\infty} e^{-sd_{\h^{p,q}}(x,\gamma_p.x)}<+\infty$, and let $\nu$ be a conformal density of dimension $s$.\\
Let $r_0$ be given by Theorem \ref{shadow_lemma}, and let $r\geq r_0$ and $C$ the associated constant from the same theorem. Given $\e>0$, we can find $N\in\N$ such that $\sum_{p=N}^{+\infty} e^{-sd_{\h^{p,q}}(x,\gamma_p.x)} \leq \e$. By Theorem \ref{shadow_lemma}, we have that $\nu_x(\mathcal S_r(x,\gamma_p.x))\leq Ce^{-sd_{\h^{p,q}}(x,\gamma_p.x)}$ for all $p\in\N$, hence $\nu_x(\bigcup_{p\geq N} \mathcal S_r(x,\gamma_p.x))\leq C\e$. This implies that $\nu_x(\bigcap_{N\in \N}\bigcup_{p\geq N} \mathcal S_r(x,\gamma_p.x))=0$.\\
If $r\leq r_0$, then $\mathcal S_r(x,\gamma_p.x)\subset \mathcal S_{r_0}(x,\gamma_p.x)$ for all $p\in \N$, so we also find $\nu_x(\bigcap_{N\in \N}\bigcup_{p\geq N} \mathcal S_r(x,\gamma_p.x))=0$.\\
Since the union over all $r>0$ in (\ref{description_limit_set}) is increasing, it can be written as a countable union, and we find that $\nu_x(\Lambda)=0$, which is a contradiction. Therefore $\sum_{p=0}^{+\infty} e^{-sd_{\h^{p,q}}(x,\gamma_p.x)}=+\infty$.
\end{proof}

\begin{coro} \label{s_equals_delta}
If there is a conformal density of dimension $s$, then $s=\delta_{\h^{p,q}}(\G)$.
\end{coro}
\begin{proof}
Corollary \ref{s_bigger_than_delta} gives us $s\geq \delta_{\h^{p,q}}(\G)$, and Proposition \ref{divergence_at_s} implies that $s\leq \delta_{\h^{p,q}}(\G)$.
\end{proof}

\begin{coro}
\[\sum_{\gamma\in\Gamma} e^{-\delta_{\h^{p,q}}(\G) d_{\h^{p,q}}(x,\gamma.x)}=+\infty\]
\end{coro}
\begin{proof}
This is a straightforward consequence of Proposition \ref{divergence_at_s} and the existence of the Patterson-Sullivan density which is of dimension $\delta_{\h^{p,q}}(\G)$.
\end{proof}

We conclude this section by stating the uniqueness theorem of conformal density. 

A conformal density $(\nu_x)_{x\in C(\Lambda)}$ is said to be ergodic if {\color{black} any $\G$-invariant subset $A\subset \Lambda$ satisfies $\nu_x(A)=0$ or $\nu_x(\Lambda\setminus A)=0$ for any $x\in C(\Lambda)$.} We have: 
\begin{theorem}
Any  anti-de Sitter conformal density $(\nu_x)_{x\in C(\Lambda)}$ is ergodic. Therefore, the Patterson-Sullivan density is the only conformal density up to a multiplicative constant.
\end{theorem}
Compare with \cite[Corollary 5.2.4]{nicholls1989ergodic}.

\section{Pseudo-Riemannian Hausdorff dimension and measure}\label{sec - Lorentzian hausdorff dim}
We still assume that a $\h^{p,q}$-convex cocompact group $\G\subset \PO(p,q+1)$ is fixed and denote by $\Lambda$ its limit set. We introduce the concept of pseudo-Riemannian Hausdorff dimension, {\color{black}generalizing} the usual definition in a metric space to our case. This gives an invariant that we show to be equal to the critical exponent $\delta_{\h^{p,q}}(\G)$, Theorem \ref{th - l'exposant critique = dim de hausdorff}. Moreover using a comparison with a Riemannian metric we show an inequality in every dimension : $\delta_{\h^{p,q}}(\G)\leq \Hdim(\Lambda)\leq p-1$.

\subsection{Definitions}
The Hausdorff dimension is usually defined for a  metric space. Here, we will use the Gromov distance $d_x$ instead of a metric.\\
\indent Although we could define a notion of pseudo-Riemannian Hausdorff dimension and measures in any pseudo-Riemannian manifold, dealing with this general setting would be the source of many technical difficulties. %Since the limit set of a quasi-Fuchsian group of {\color{black}$\PO(p,q+1)$ is included some de Sitter domain (see Definition \ref{definition - domaine de Sitter})}, it is enough for our purpose to deal with subsets of $\partial U(x)$ where $x\in\ADS^{n+1}$ is given.\\

%First, recall the hyperboloid model of de Sitter space: \begin{eqnarray*} \dS^n & = & \{(\xi_0,\dots,\xi_n)\in \R^{n+1}\,\vert\, -\xi_0^2+x_1^2+\cdots+\xi_n^2=1\}\\
%& = & \{\xi\in \R^{n+1}\,\vert\, \langle \xi\vert \xi\rangle_{n,1}=1\}
%\end{eqnarray*}
%
%Two points $\xi,\eta\in\dS^n$ are joined by a unique geodesic. We define $\theta(\xi,\eta)$ as the length of this geodesic if it is spacelike, and set $\theta(\xi,\eta)=0$ in other cases.\\
% Finally, given $\xi\in \dS^n$ and $r>0$, we set $B_\dS(\xi,r)=\{\eta\in\dS^n \vert \theta(\xi,\eta)<r\}$.\\

\indent If $E\subset \Lambda$, for all $s>0$ and $\e>0$ we set: \[\mathrm H_{d_x}^{s,\e}(E) = \inf \left\{ \sum r_i^s \vert E\subset \bigcup B_x(\xi_i,r_i), \xi_i\in E, r_i\le \e\right\} .\]
 Since $\mathrm H_{d_x}^{s,\e}(E)$ increases as $\e$ decreases, we can consider: \[\mathrm H_{d_x}^s(E)=\lim_{\e\to 0} \mathrm H_{d_x}^{s,\e}(E) \in \intff{0}{+\infty}.\] Finally, the pseudo-Riemannian Hausdorff dimension of $E$ is: \[\Hdim_{d_x}(E)=\inf\{ s>0 \vert \mathrm H_{d_x}^s(E)=0\}.\]

 {\color{black}\emph{Remark:} Although the definition of $\mathrm H_{d_x}^s$ makes sense for any subset of  $\partial U(x)$, it does not define an outer measure on $\partial U(x)$. However, it is one when restricted to the limit set $\Lambda$ of a $\h^{p,q}$-convex cocompact group $\Gamma$ of $\PO(p,q+1)$. This is true because $\Lambda$ is a \emph{quasi-metric space}, Lemma \ref{triangle_boundary}. For Hausdorff measures and dimension of quasi-metric spaces, see the second chapter of \cite{AlvaradoMitreaBook}.\\
 We also want to emphasize the fact that the centers $\xi_i$ of the balls must be taken in $E$. This is not so important for metric spaces, but crucial in our context (one can easily check that if centers outside of $\Lambda$ were allowed, the dimension would be $0$ for any  $E\subset \Lambda$). \\}
 {\color{black}
\indent \emph{Geometric meaning of $\Hdim_{d_x}$:} The Hausdorff dimension of a metric space reflects the number of balls of a certain radius that are necessary to cover the set. The pseudo-Riemannian Hausdorff dimension $\Hdim_{d_x}$ reflects the number of "balls" for $d_x$ that are required to cover a subset of $\Lambda$. However, "balls" for $d_x$ are actually the interiors of quadrics  (intersected with $\Lambda$), which are the pseudo-Riemannian analogue of balls.
}

\begin{prop} \label{prop - comparaison dimensions} If $A\subset \Lambda$, then $\Hdim_{d_x}(E)\leq \Hdim_h (E)$, where $\Hdim_h(E)$ is the Hausdorff dimension with respect to any Riemannian metric $h$ on $\partial U(x)$.  
\end{prop}

\begin{proof} First note that Riemannian metrics on a manifold always give the same Hausdorff dimension to compact subsets because they are locally bi-Lipschitz with respect to each other.\\
It will be sufficient for our purpose to find a Riemannian metric $h$ on $\partial \ADS^{n+1}$ such that $d_x(\xi,\eta)\leq d_h(\xi,\eta)$ for all $\xi,\eta\in \partial U(x)$, where $d_h$ is the distance associated to $h$.\\
%To make the computation easier, assume that $x=[1:0:\cdots:0]$.
%First, let us find a Riemannian metric $h$ on $\dS^n$ such that $d_h(\xi,\eta)\geq \theta(\xi,\eta)$ for all $\xi,\eta\in \dS^n$.\\ 
Let us fix a lift $\tx\in\R^{n+2}$ of $x$ satisfying $\pscalb{\tx}{\tx}=-1$. There is a natural identification between $\partial U(x)$ and the following submanifold of $\R^{n+2}$:
\[ V_{\tx}=\{u\in \R^{n+2} \vert \pscalb{u}{u}=0=1- \pscalb{\tx}{u} \}.\] Any  $\xi\in \partial U(x)$ has a unique lift in $V_{\tx}$, which we will denote $\p(\xi)$.
\begin{align*}
d_x(\xi,\eta) &= \sqrt{\frac{-\pscalb{\p(\xi)}{\p(\eta)}}{2\pscalb{\p(\xi)}{\tx}\pscalb{\tx}{\p(\eta)}}} \\
&= \sqrt{\frac{-\pscalb{\p(\xi)}{\p(\eta)}}{2}} \\
&= \frac{1}{2}\sqrt{\pscalb{\p(\xi)-\p(\eta)}{\p(\xi)-\p(\eta)}}.
\end{align*}

 %If $\xi,\eta\in \dS^n$, then $\theta(\xi,\eta)>0$ is equivalent to $\langle \xi\vert \eta\rangle_{n,1}<1$ (in the hyperboloid model mentioned above), and in this case we have $\theta(\xi,\eta)=\arccos(\langle \xi\vert \eta\rangle_{n,1})=\arccos(1-\frac{\langle \xi-\eta\vert \xi-\eta\rangle_{n,1}}{2})\leq \sqrt{\vert\langle \xi-\eta\vert \xi-\eta\rangle_{n,1}\vert} $.\\
If we denote by $\langle\cdot\vert\cdot\rangle_{e}$ the Euclidean inner product on $\R^{n+2}$, we find that $\sqrt{\pscalb{\p(\xi)-\p(\eta)}{\p(\xi)-\p(\eta)}}\leq \sqrt{\pscal{\p(\xi)-\p(\eta)}{\p(\xi)-\p(\eta)}_{e}}$ for all $\xi,\eta\in \Lambda$ (because $\pscalb{\p(\xi)-\p(\eta)}{\p(\xi)-\p(\eta)}\geq 0$), hence \[d_x(\xi,\eta)\leq\frac{1}{2} \sqrt{\pscal{\p(\xi)-\p(\eta)}{\p(\xi)-\p(\eta)}_{e}}.\]

 If $h$ is the pull-back by $\p$ of the Riemannian metric on $V_{\tx}$ induced by the Euclidean metric of $\R^{n+2}$, we have \[\sqrt{\pscal{\p(\xi)-\p(\eta)}{\p(\xi)-\p(\eta)}_{e}}\leq d_h(\xi,\eta).\]
 
 Finally, we get $d_x(\xi,\eta)\leq \frac{1}{2}d_h(\xi,\eta)\leq d_h(\xi,\eta)$.
From this, we deduce that $B_h(\xi,r)\subset B_x(\xi,r)$ for all $\xi\in \Lambda,r\ge 0$, where $B_h$ is the ball for the Riemannian metric $h$. This implies that $\Hdim_{d_x}(E)$ is smaller than the Hausdorff dimension for $h$.
\end{proof}

We also find a universal upper bound on the (Riemannian) Hausdorff dimension of $\Lambda$, which by Proposition \ref{prop - comparaison dimensions} also provides an upper bound for the pseudo-Riemannian Hausdorff dimension.
\begin{prop} \label{prop - dim riemannienne bornee} $\Hdim(\Lambda)\leq p-1$
\end{prop}
\begin{proof} Consider the linear model $\mathcal H^{p,q}$. Its boundary  $\partial\mathcal H^{p,q}$ is the quotient of the isotropic cone of $\pscalb{\cdot}{\cdot}$ by positive homotheties. The pre-image of $\Lambda$ in $\partial\mathcal H^{p,q}$ is the disjoint union of two sets $\Lambda_\pm$.\\
The map $\p:\Ss^{p-1}\times \Ss^q \to \partial \mathcal H^{p,q}$ defined by $\p(x,y)=(x,y)$ is diffeomorphism.  The fact that $\Lambda$ is negative implies that $\p^{-1}(\Lambda_+)$ is the graph of a $1$-Lipschitz function $f:\mathcal L\to \Ss^q$ for some closed subset $\mathcal L\subset \Ss^{p-1}$. It follows that $\Hdim(\Lambda_+)=\Hdim(\mathcal L)\leq p-1$. The result follows from the fact that $\Hdim(\Lambda_+)=\Hdim(\Lambda)$.
\end{proof}

\subsection{Pseudo-Riemannian Hausdorff measure and the Patterson-Sullivan density} 

The  Vitali covering lemma is a very useful tool for computing Hausdorff dimensions. Since $d_x$ is not a distance, we will need an appropriate version of this classic result.\\
Recall that by Lemma \ref{triangle_boundary}, there is $\lambda_\G\geq 1$ such that $d_x(\xi,\eta)\le \lambda_\G (d_x(\xi,\tau) + d_x(\tau, \eta))$ for all $x\in C(\Lambda), \xi,\eta,\tau\in\Lambda$.\\
{\color{black}Note that the following lemma is also a consequence of Lemma 2.7 in \cite{AlvaradoMitreaBook}, however we include a proof in order to stay as self-contained as possible.}

\begin{lemma}[Vitali for $d_x$] \label{Vitali_distance}
Given a subset $J\subset \Lambda$ and a bounded function $r:J\to (0,+\infty)$, there is a subset $I\subset J$ such that:
\begin{itemize} \item The balls $B_x(\xi,r(\xi))$ are disjoint for distinct points $\xi\in I$.
\item $\bigcup_{\xi\in J} B_x(\xi,r(\xi)) \subset \bigcup_{\eta\in I}B_x(\eta,5\lambda_\G^2r(\eta))$.
\end{itemize}
\end{lemma}

\begin{proof} 
Let $R=\sup_J r$, and consider $J_n=\{\xi\in J : 2^{-n-1}R<r(\xi)\leq 2^{-n}R\}$ for any $n\geq 0$, so that $J$ is the disjoint unions of these subsets. Define inductively subsets $I_n,H_n$ of $J_n$ by letting $H_0=J_0$, and $I_0\subset H_0$ be maximal amongst the subsets $A\subset H_0$ such that the balls $B_x(\xi,r(\xi))$ are disjoint for distinct $\xi\in A$ (such a subset exists by Zorn's Lemma). Given $I_0,\dots, I_n$, we let $H_{n+1}=\{\xi\in J_n : \forall \eta\in I_0\cup\cdots\cup I_n~ B_x(\xi,r(\xi))\cap B_x(\eta,r(\eta))=\emptyset\}$, and choose $I_{n+1}$  maximal amongst the subsets  $A\subset H_{n+1}$ such that the balls $B_x(\xi,r(\xi))$ are disjoint for distinct $\xi\in A$ . Finally, let $I=\bigcup_{n\in\N}I_n$.\\
It follows from the construction of $I$ that the considered balls are disjoint. For the second point, let $\xi\in J$, and consider $n\in \N$ such that $\xi\in J_n$. There are two cases: either $\xi\notin H_n$, in which case there is $\eta\in I_0\cup\cdots\cup I_n\subset I$ satisfying $B_x(\xi,r(\xi))\cap B_x(\eta,r(\eta))\neq \emptyset$, or $\xi\in H_n$, in which case there is $\eta\in I_n\subset I$ satisfying $B_x(\xi,r(\xi))\cap B_x(\eta,r(\eta))\neq \emptyset$ (because of the maximality of $I_n$).\\
In both cases, we find $\eta\in I_0\cup\cdots\cup I_n$ such that $B_x(\xi,r(\xi))\cap B_x(\eta,r(\eta))\neq \emptyset$. Since $r(\eta)>2^{-n-1}R$ and $r(\xi)\leq 2^{-n}R$, we have $r(\xi)\leq 2r(\eta)$, which implies $B_x(\xi,r(\xi))\subset B_x(\eta,\lambda_\G(2+3\lambda_\G)r(\eta))\subset B_x(\eta,5\lambda_\G^2r(\eta))$.\\
\end{proof}

We can now compare the pseudo-Riemannian Hausdorff measures of dimension $\delta_{\h^{p,q}}(\G)$ and the Patterson-Sullivan measures.

\begin{theorem} \label{th - comparaison Hausdorff PS}
Let $(\mu_x)_{x\in C(\Lambda)}$ denote the Patterson-Sullivan density. For all $x\in C(\Lambda)$, there is $\alpha>0$ such that $\frac{1}{\alpha}\mathrm H_{d_x}^{\delta_{\h^{p,q}}(\G)}(E)\leq \mu_x(E)\leq \alpha \mathrm H_{d_x}^{\delta_{\h^{p,q}}(\G)}(E)$ for all measurable subset $E\subset \Lambda$.
\end{theorem}

\begin{proof} By Theorem \ref{measure_balls}, we can fix a constant $c>0$ such that \[\frac{1}{c}   r^{\delta_{\h^{p,q}}(\G)} \leq \mu_x(B_x(\xi,r)) \leq  c r^{\delta_{\h^{p,q}}(\G)}\] for all $\xi\in \Lambda$, $r\in (0,1)$.\\
Let $E\subset \Lambda$ be a measurable set.We start with the left hand side inequality. Let $\e>0$. Consider the open cover $E\subset \bigcup_{\xi\in E} B_x(\xi,\frac{\e}{5\lambda_\G^2})$. By Lemma \ref{Vitali_distance}, we can find a (necessarily countable) subset $J\subset E$ such that $E\subset \bigcup_{\xi\in J} B_x(\xi,\e)$ and the balls $B_x(\xi,\frac{\e}{5\lambda_\G^2})$ for $\xi\in J$ are pairwise disjoint.\\
%By Theorem \ref{measure_balls}, there is a constant $c>0$ such that $c\nu_x(B_x(\xi,r)) \geq  r^{\delta_{\ADS}(\G)}$ for all $\xi\in \Lambda$, $r\in (0,1)$.
 Since  $\mathrm H^{\delta_{\h^{p,q}}(\G),\e}_{d_x}(E)\leq \sum_{\xi\in J} \e^{\delta_{\h^{p,q}}(\G)}$, we find:
\begin{align*}\mathrm H^{\delta_{\h^{p,q}}(\G),\e}_{d_x}(E)&\leq (5\lambda_\G^2)^{\delta_{\h^{p,q}}(\G)} c \sum_{\xi\in J}\mu_x(B_x(\xi,\frac{\e}{5\lambda_\G^2}))\\
 &\leq (5\lambda_\G^2)^{\delta_{\h^{p,q}}(\G)} c \mu_x(E).
\end{align*} 

Let us now deal with the right hand side inequality. Let $(\xi_i,r_i)$ be a countable family of points of $E$ and radii such that $E\subset \bigcup B_x(\xi_i,r_i)$. %By Theorem \ref{measure_balls}, there is a constant $c>0$ such that $\nu_x(B_x(\xi,r)) \leq c r^{\delta_{\ADS}(\G)}$ for all $\xi\in \Lambda$, $r\in (0,1)$. 
We have $\nu_x(E)\leq \sum \nu_x(B_x(\xi_i,r_i))\leq c \sum r_i^{\delta_{\h^{p,q}}(\G)}$, hence $\nu_x(E) \leq c \mathrm H^{\delta_{\h^{p,q}}(\G),\e}_{d_x}(E)$ for all $\e>0$.\\
Combining these two inequalities and letting $\e\to0$, we get:
\[ \frac{1}{c(5\lambda_\G^2)^{\delta_{\h^{p,q}}(\G)}} \mathrm H^{\delta_{\h^{p,q}}(\G)}_{d_x}(E)\leq \mu_x(E) \leq c \mathrm H^{\delta_{\h^{p,q}}(\G)}_{d_x}(E).\]

\end{proof}

\subsection{The pseudo-Riemannian Hausdorff dimension of the limit set}

{\color{black}We can now establish the equality between the critical exponent and the pseudo-Riemannian Hausdorff dimension of the limit set.}

\begin{theorem}\label{th - l'exposant critique = dim de hausdorff}
 For any $x\in C(\Lambda)$, we have $\delta_{\h^{p,q}}(\G) =\Hdim_{d_x}(\Lambda).$
\end{theorem}

{\color{black}\begin{proof}
By Theorem \ref{th - comparaison Hausdorff PS} applied to $E=\Lambda$, we find that $0<\mathrm H^{\delta_{\h^{p,q}}(\G)}_{d_x}(\Lambda)<+\infty$.\\
The positivity $0<\mathrm H^{\delta_{\h^{p,q}}(\G)}_{d_x}(\Lambda)$ implies $\delta_{\h^{p,q}}(\G)\leq \Hdim_{d_x}(\Lambda)$.\\
The finiteness $\mathrm H^{\delta_{\h^{p,q}}(\G)}_{d_x}(\Lambda)<+\infty$ implies $\delta_{\h^{p,q}}(\G)\geq \Hdim_{d_x}(\Lambda)$.
\end{proof}
}

{\color{black}The fact that $\Hdim_{d_x}(\Lambda)$ does not depend on the choice of $x\in C(\Lambda)$ is straightforward: if $x,y\in C(\Lambda)$, then the quasi metrics $d_x,d_y$ are bi-Lipschitz on $\Lambda$. However, this equality is also a consequence of our proof, as we work with a fixed $x\in C(\Lambda)$ and the critical exponent $\delta_{\h^{p,q}}(\G)$ does not depend on this point.\\}

Combining Theorem \ref{th - comparaison Hausdorff PS}, Proposition \ref{prop - comparaison dimensions} and Proposition \ref{prop - dim riemannienne bornee}, we  obtain the announced inequality.
\begin{coro} $\delta_{\h^{p,q}}(\G) \leq \Hdim(\Lambda)\leq p-1$.
\end{coro}

\section{Rigidity theorem in dimension 3}\label{sec - rigidity}

This section is devoted to Theorem \ref{th - Intro rigidity }, which is a rigidity theorem for the critical exponent in dimension $3$. Most of the arguments in this section are valid in Lorentzian signature in any dimension, i.e. for $\ADS^{n+1}=\h^{n,1}$, as long as we only look at $\ADS^{n+1}$-quasi-Fuchsian groups. In fact dimension $3$  is  only used for two reasons, which only appear at the very end of the proof. It is first used in the existence of a Cauchy surface of entropy $1$ (namely the boundary of the convex core). The existence in higher dimension of a Cauchy surface of entropy $n-1$ is not known. The second time we use the dimension is when we compute and compare  the length spectra of $C(\Lambda)/\G$ and a hypersurface $\St$.\\
We now consider that a torsion free $\ADS^{n+1}$-quasi-Fuchsian group $\G\subset \PO(n,2)$ is fixed (we will only restrict ourselves to the case $n=2$ when needed).\\
We will use the notations $\delta_\AdS(\G)$ (resp. $d_\AdS$) instead of $\delta_{\h^{n,1}}(\G)$ (resp. $d_{\h^{n,1}}$) . %This argument can in fact be bypassed. Therefore, the rigidity in higher dimension is equivalent to the existence of a Cauchy surface of entropy $n-1$. 

The proof  is based on a comparison between $\delta_\AdS(\G)$ and   the critical exponent of the  $\G$-action   on the boundary of the convex core with its induced Riemannian metric. {\color{black}{The proof is based on the work in \cite{glorieux2015entropy}, however since the inequality in the anti-de Sitter case is  opposite as the one for hyperbolic quasi-Fuchsian, we need to change many arguments which makes the proof somehow more involved. }}% We will need the equivalent of the Bowen-Margulis measure on the non-wandering set for the geodesic flow on $C(\Lambda)$. We explain the construction and prove the ergodicity of this measure in the first part, then comes the genuine proof of rigidity. 

As  mentioned in the introduction, one can use the Mess parametrization of $\AdS$ quasi-Fuchsian groups in $\PO(2,2)$ and the right notion of entropy to show that $\delta_\AdS(\G)$ is equal to the critical exponent of an action of $\G$ on $\Hyp^2\times \Hyp^2$, making  Theorem \ref{th - Intro rigidity }  equivalent to a result of C. Bishop and T. Steger \cite{bishop1991three}.\footnote{This is not a straightforward consequence of their work. Indeed, one has to verify that the $\AdS$ critical exponent is equal to the $\AdS$ entropy, ie. the exponential growth rate of the number of closed geodesics.} The proof proposed here is totally independent and has a clear Lorentzian geometric interpretation.

Moreover thanks to examples proposed in \cite{glorieux2015asymptotic}, we know the asymptotic behavior of critical exponent when the two representations in Mess parametrizations range over the product of Teichmüller spaces. \\

 {\color{black}{
This section is organized as follows. After making some comments on geodesic flow, we first introduce the Bowen-Margulis measure and show it is ergodic. The proof works as in the Riemannian context, since it uses Birkhoff's ergodic theorem which is a purely dynamical  result.

 Then we introduce a cocycle that measure the distortion of distances between the convex core and a fixed hypersurface. Using  Kingman's subadditive ergodic theorem, which is a well known result in ergodic theory, we state the main result of this section which is Theorem \ref{th - inequality h<I delta}. This cocycle is defined through a retraction map from the convex core to a Cauchy hypersurface, it uses the $\AdS$ geometry and therefore  is not classical. We study the properties of this map in Subsection \ref{subsection - inequality}.  Proposition \ref{pr - f is 1 lip} is the main ingredient in most of the  $\AdS$-geometric arguments. 
  
In the last subsection, we prove the rigidity theorem. Once again the first part uses a purely dynamical result that appears in \cite{knieper1995volume} and the second part uses $\AdS$ geometry.}}

{\color{black}{
\subsection{Anti-de Sitter geodesic flow}
In this section, we are going to study the dynamic of the geodesic flow associated to $\G$. Since we are working on Lorentzian manifolds, the geodesic flow is  only defined on a subset of the tangent bundle, the set of  spacelike tangent vectors.
The unitary tangent bundle, is the set
$$T^1\ADS^{n+1} := \{(x,v)\in T\ADS^{n+1}\, |\, \pscalb{\tilde v}{\tilde v}=+1\}.$$
Here $\tilde v\in x^\perp$ is given by an identification between $T_x\ADS^{n+1}$ and 
$x^\perp$.  %A tangent vector in $v\in T\ADS$ is given by the equivalence class $v=\{(\tx,u),(-\tx,-u)\}$, where $\tx$ is  a point in $\AdS$ and $u\in x^\perp \subset \R^{p,q+1}.$ Writing $(x,v)\in T\ADS$ is therefore  a little bit confusing, however it is much more convenient than writing $v=\{(\tx,u),(-\tx,-u)\}$ 
%Recall that a space like vector $v$ defines a space like geodesic, which has itself two endpoints on the boundary of $\ADS$: $v^\pm$. 

Every $v\in T^1\ADS^{n+1}$ is spacelike, so the geodesic of $\ADS^{n+1}$ tangent to $v$ is spacelike, and it has two endpoints on the boundary of $\ADS^{n+1}$. We denote them by $v^\pm\in \partial \ADS^{n+1}$. There is no ambiguity on these two points. Indeed, $(x,v)$ lifts to $\{(\tx,v),(-\tx,-v)\}$ and defines two geodesics on the linear model $\AdS^{n+1}$: 
$$\tx(t)= e^{-t} \tx -\frac{\sinh(t)}{\langle \tx,v+\tx\rangle_{p,q+1}}(v+\tx).$$
$$\ty(t)= e^{-t}(- \tx) -\frac{\sinh(t)}{\langle -\tx,-v-\tx\rangle_{p,q+1}}(-v-\tx).$$
And $v^+=\lim_{t\to+\infty}[\tilde x(t)]=\lim_{t\to+\infty}[\tilde y(t)]=[\tx+v]$, and $v^- = [-\tx+v]$. 

Remark that the $\G$-action is properly discontinuous only on the unitary tangent bundle over $E(\Lambda)$. This action of $\G$ is not cocompact on the unitary tangent bundle of $C(\Lambda)$ because the fibers are not compact. However, the subset of recurrent vectors defined by: \[R^1C(\Lambda)=\{v\in T^1C(\Lambda) \, |\, v^+ \in \Lambda\}\] is $\G$-invariant and cocompact. Indeed let $(x_i,v_i)$ be a sequence, of such vectors. Up to the action of $\G$, we can assume that $x_i$ converges to $x_\infty\in K$ a compact fundamental domain for the action of $\G$ on $C(\Lambda)$. Now we can look at recurrent vectors whose base point $x_\infty$. This set is homeomorphic to $\Lambda$ by the map $v\mapsto v^+$. 

Note that $E(\Lambda)/\G$ is not complete and the recurrent vectors is the set for which the geodesic flow $\phi^\AdS_t$ is well defined for all time $t>0$ : for $(x,v)\in R^1C(\Lambda)$, we denote by $\phi_t^\AdS (v) $ the tangent vector on the geodesic defined by $v$ at distance $t$ of $x$. If we take the same notation as before, we have 
$$\phi_t^\AdS (v) = [(\tx(t), \tx'(t))].$$

We will consider a smaller subset, the set of non-wandering vectors, for which the geodesic flow is  well defined for all $t\in \R$. 

}}

%We denote by $v^+$ respectively $v^-$ the limit of the geodesic ray generated by $v$ resp. $-v$

\begin{defi}
The non-wandering set, denoted by $N^1C(\Lambda)$, is the subset of $T^1C(\Lambda)$ defined by
$$N^1C(\Lambda):=\{v\in T^1C(\Lambda) \, |\, v^\pm \in \Lambda\}.$$
\end{defi}

The non-wandering set  is homeomorphic to $\Lambda^{(2)} \times \R$ thanks to the so-called Hopf parametrization. Fix a point $o \in C(\Lambda)$, the map 
$$v\in N^1C(\Lambda) \tv (v^-,v^+, \beta_{v^-} (o, \pi(v)),$$ 
is  a homeomorphism, such that the action of the geodesic flow $\phi_t^\AdS$ is given by 
$$\phi_t^{\AdS} (v) = \phi_t(v^-,v^+,s) = (v^-,v^+,t+s).$$

\subsection{Bowen - Margulis measure}
This measure was first introduced by G. Margulis in \cite{margulis1969applications} and R. Bowen in \cite{bowen1972equidistribution}. A good introduction for the hyperbolic case can be found in Chapter 8 of Nicholls' book \cite{nicholls1989ergodic} or in the book of T. Roblin \cite{roblin2003ergodicite}.

The measure $d\mu(\xi,\eta):=\frac{d\nu_x (\xi) d\nu_x(\eta)}{d_x(\xi, \eta)^{-2\delta_\AdS(\G)}}$ does not depend on the point $x$. Indeed, 
$$\frac{d\nu_x (\xi) d\nu_x(\eta)}{d\nu_y (\xi) d\nu_y(\eta)} = e^{-\delta_\AdS(\G)(\beta_\xi(x,y) +\beta_\eta (x,y))}$$
And 
 \[ \frac{d_x(\xi,\eta)^2}{d_y(\xi,\eta)^2}=\left|\frac{\langle \txi\vert\teta\rangle_{n,2}}{\langle\txi\vert \tx\rangle_{n,2}\langle \tx\vert\teta\rangle_{n,2}}        \frac{\langle\txi\vert \ty\rangle_{n,2}\langle \ty\vert\teta\rangle_{n,2}}{\langle \txi\vert\teta\rangle_{n,2}} \right|  = \left| \frac{\langle\txi\vert \ty\rangle_{n,2}\langle \ty\vert\teta\rangle_{n,2}}{\langle\txi\vert \tx\rangle\langle \tx\vert\teta\rangle_{n,2}} \right|      .  \]
Finally from the definition of $\beta$ we have :
$$\frac{\pscal{\txi}{\ty}_{n,2}}{\pscal{\txi}{\tx}_{n,2}  }=e^{-\beta_{\xi} (x,y)} \quad \text{and}\quad \frac{\pscal{\teta}{\ty}_{n,2}}{\pscal{\teta}{\tx}_{n,2}  }=e^{-\beta_{\eta} (x,y)}.$$
Hence
$$\frac{d_x(\xi,\eta)^{2\delta}}{d_y(\xi,\eta)^{2\delta}} = e^{-\delta (\beta_{\xi} (x,y) +\beta_{\eta} (x,y))}.$$
The same kind of computations shows that $\mu$ is also $\G$-invariant. 

Thanks to previous computations, we see that $N^1C(\Lambda)$ carries a measure, invariant by $\G$ as well as by $\phi_t^{\AdS}$. 
\begin{defi}
The following measure on $N^1C(\Lambda)\simeq\Lambda^{(2)}\times \R$ is called the \emph{Bowen-Margulis} measure :
$$dm(v) :=\frac{d\nu_x(v^-) d\nu_x(v^+) dt}{d_x(v^-, v^+)^{-2\delta_\AdS(\G)}}.$$
It is invariant by $\G$ and $\phi_t^{\AdS}$. 
\end{defi}
In other words, let $f:N^1(C(\Lambda))\tv \R$ be a continuous function with compact support, and let $t\tv c_{(\xi \eta)}(t)$ be a parametrization of the $\AdS$ geodesic $(\xi \eta)$. Then $$\int_{N^1C(\Lambda)} f(v) dm(v) : =\int_{\Lambda^{(2)} } \int_\R f(c_{(\xi\eta)}(t) ) dt d\mu(\xi,\eta).$$
We easily see in this formulation that it is invariant by the geodesic flow.

Remark that $N^1C(\Lambda)$ is invariant by $\G$, and as a closed subset in $T^1C(\Lambda)$, $N^1C(\Lambda)/\G$ is compact. The Bowen-Margulis measure descends to a $\phi_t^{\AdS}$-invariant finite measure whose support is $N^1C(\Lambda)/\G$, still denoted by $m$. % In dynamical system vocabulary, we can say that $(N^1C(\Lambda),m, \phi_t^{\AdS})$ is  a measure preserving system.   

%The proofs of the dynamical results are very similar to the 
{\color{black}{
Recall  Birkhoff's ergodic theorem : 
\begin{theorem}
Let $(X, \mathcal{B}, \mu)$ be  a probability space and let $\phi_t \, :\,  X\tv X$ be measure preserving flow. Let $\mathcal{I} $ denote the $\sigma$-algebra of $\phi_t$-invariant sets. Then for every $f\in L^1(X,\mathcal{B}, \mu)$ we have 
$$ \lim_{T\tv+\infty}  \frac{1}{T} \int_0^T f(\phi_t x) dt \tv E(f, \mathcal{I}),$$
for $\mu$-a.e. $x\in X$. Here $ E(f, \mathcal{I})$ denotes the conditional expectation of $f$ with respect to the $\sigma$-algebra $\mathcal{I}$. 
\end{theorem}	}}

\begin{theorem}
The geodesic flow is ergodic with respect to the Bowen-Margulis measure. 
\end{theorem}

\begin{proof}
{\color{black}{
Recall the following characterization of ergodicity: a dynamical system is ergodic if and only if any $L^1$ invariant function is constant almost everywhere. 

The proof of ergodicity is based on the the so-called Hopf's argument. Let $f$ be a measurable function on $N^1C(\Lambda)/\Gamma$.  We consider the Birkhoff means : 
$$\Phi_f^s(v) := \lim_{T\tv+\infty}  \frac{1}{T} \int_0^T f(\phi_t^{\AdS} v) dt,$$
$$\Phi_f^u(v) := \lim_{T\tv+\infty}  \frac{1}{T} \int_0^T f(\phi_{-t}^{\AdS} v) dt.$$
Thanks to Birkhoff's ergodic theorem the functions $\Phi_f^u$ and $\Phi_f^s$ exist $m$-a.e. and are invariant by the geodesic flow. Remark that since $\phi_t^\AdS$ is invertible, the $\phi_t^\AdS$-invariant $\sigma$-algebra is equal to the $\phi_{-t}^\AdS$-invariant $\sigma$-algebra, and therefore, $\Phi_f^u$ and $\Phi_f^s$  are equal $m$-a.e.

We need to prove that those two functions are constant almost everywhere.  Since continuous functions are dense in the set of $L^1$ functions, we can assume that $f$ is continuous. 

We now remark that $\Phi_f^{u}$ is constant along unstable leaves of $N^1C(\Lambda)$, that is $W^{u}(v) := \{w=(v^-,x,t) \in N^1 C(\Lambda) \, | \, x\in \Lambda\setminus \{v^+\} \}$ and $\Phi_f^{s} $ is constant along the stable leaves of $N^1 C(\Lambda)$, that is $W^{s}(v) := \{w=(x,v^+,t) \in N^1 C(\Lambda) \, | \, x\in \Lambda\setminus \{v^-\} \}$.
}}
Since the measure is a product measure,  it is a consequence of Fubini's Theorem that $\Phi_f^u$ and $\Phi_f^s$ are constant $m$-almost everywhere. This concludes the proof. 
\end{proof}

\paragraph{Remark concerning notations} 

{\color{black}{
In the rest of this section we will have to make a clear distinction between the different metrics.
The intrinsic distance on a complete $\G$-invariant Riemannian hypersurface $\St$ will be denoted by $d_{\St}$ and the corresponding balls by $B_\St(x,R)\subset \St$.}}

\subsection{Definition of a cocycle}
The proof in the anti-de Sitter setting goes in the opposite way as in  the hyperbolic case, in particular in the $\AdS$ setting, we need to consider a subaditive  cocycle living on $N^1C(\Lambda)$ that compares the distance with a given Cauchy hypersurface. For this we need to define a retraction map from $C(\Lambda) $ to $\St$.

Let $\St$ be a negatively curved Riemannian submanifold of $C(\Lambda)$ which is a Cauchy hypersurface, i.e.  a topological hypersurface in $C(\Lambda)$ which is intersected by every inextensible causal curve exactly once (causal curves are smooth curves admitting lifts $c(t)$ to $\R^{n,2}$ satisfying $\pscalc{\dot c(t)}{\dot c(t)}\leq 0$ and $\dot c(t)\neq 0$ for all $t$). A smooth Cauchy hypersurface is in particular a Riemannian manifold. The negative curvature assumption on $\St$ will allow us  to consider Patterson-Sullivan measures on $\St$ for the induced Riemannian metric. \\

Remark: In dimension $3$, we allow $\St$ to be bent along geodesic lamination. It includes in this case two examples: the boundaries of the convex core which are isometric to $\Hyp^2$, bent along laminations and the unique maximal surface. \\

%In this section we will compare the distance, shadows and Patterson-Sullivan measures for the distance on $\St$ and on $C(\Lambda)$. In order to avoid confusion we will add subscripts $\AdS$ and $\St$ on the different invariants. For example, the distance $d$ defined in  \ref{def - definition de la distance dasn C(Lambda)}, will be denoted by $d_{\h^{p,q}}$.

We now define the map  $f$ from $C(\Lambda)$ to $\St$:
\begin{defi}\label{def - definition of  the retraction f}
Let $V$ be a time-like vector field in $E(\Lambda)/\G$, we still call $V$ its lift to $E(\Lambda)$.
For any $x$ in $C(\Lambda)$ we call $f(x)$ the intersection of the integral curve of $V$ starting at $x$ and $\St$. This is well defined since $\St$ is a Cauchy hypersurface. 
\end{defi}

Let $(\phi_t)$ be the geodesic flow on $T^1C(\Lambda)$. We may write $\phi_t^\AdS$ if we want to stress that it is the geodesic flow on $\AdS$ and similarly $\phi_t^{\St}$ for the geodesic flow on $\St$.  Let $\pi : T^1 C(\Lambda) \tv C(\Lambda)$ denote the projection. For any $v\in T^1 C(\Lambda)$ we define the following cocycle: $$a(v,t)  = d_{\St} (f(\pi \phi^\AdS_t (v)) , f(\pi v ) ).$$ It is subadditive: 
\begin{eqnarray*}
a(v,t_1+t_2)& =& d_{\St} (f(\pi \phi^\AdS_{t_1+t_2} (v)) , f(\pi v ) )\\
					&\leq &  d_{\St} (f(\pi \phi^\AdS_{t_1+t_2} (v)) , f(\pi \phi^\AdS_{t_1} v ) )+d_{\St} (f(\pi \phi^\AdS_{t_1} (v)) , f(\pi v ) ) \\ 
					& \leq &  a(\phi^\AdS_{t_1} v, t_2)+a(v,t_1).
\end{eqnarray*}
Since $a$ is $\G$-invariant  it defines a subadditive cocycle on $N^1 C(\Lambda)/\G$, still denoted by $a$.\\

{\color{black}{
Recall  Kingman's subadditive ergodic theorem:
\begin{theorem}\cite{kingman1968ergodic}
Let $(\Omega,\mu, T_t)$ be a measure preserving system. Let $a_t$ be a  subbaditive family of $L^1$ cocycles ie. $a_{t+s} \leq a_t + a_t\circ T_s$, then for almost every $\omega \in \Omega$ the limit 
$\lim_{t\tv \infty} \frac{1}{t}a_t(\omega)=g(\omega) $ exists, $g$ is $T_t$ invariant and  $L^1(\Omega)$.

 Moreover if $T_t$ is ergodic with respect to $\mu$ then $g$ is constant almost everywhere.
 % and 
%$$\lim_{t\tv \infty}  \int_\Omega \frac{1}{t}a_t(\omega) d\mu = \int_\Omega g d\mu$$ 

\end{theorem}

We apply this dynamical theorem to our context and gives:}}
\begin{theorem}\label{th - ergodicity implies a(v,t)/t has almost everywhere a limit}
Les $\mu$ be a $\phi^\AdS_t$ invariant probability measure on $N^1 C(\Lambda)/\G $. Then 
$$I_\mu(\Sigma, v) := \lim_{t\tv \infty} \frac{a(v,t)}{t}$$
exists for $\mu$ almost $v\in N^1 C(\Lambda)/\G $ and defines a $\mu$-integrable function  on $N^1 C(\Lambda)/\G$, invariant under the geodesic flow. %and  we have : 

%$$\int_{N^1 C(\Lambda)/\G} I_\mu (\Sigma,v)d\mu = \lim_{t\tv \infty } \int_{N^1 C(\Lambda)/\G}\frac{a(v,t)}{t} d\mu.$$
Moreover if $\mu$ is ergodic $I_\mu(\Sigma,v)$ is constant $\mu$-almost everywhere. In this case, we write $I_\mu(\Sigma)$
\end{theorem}

Since we proved that the Bowen-Margulis measure is ergodic, we have that $I_m(\Sigma,v)$ is constant $m$-almost everywhere and we write $I_m(\Sigma)$. We can state the main theorem of this section: 
\begin{theorem}\label{th - inequality h<I delta}
Let $\St$ be a Cauchy hypersurface whose induced metric has negative curvature. 
Let $h(\Sigma)$ be the volume entropy of $\St$.  Let $m$ be the Bowen-Margulis measure on $N^1C(\Lambda)/\G$. Then 
\begin{equation}\label{eq - th inequality h <idelta}
\delta_\AdS(\G) \leq  I_{m} (\Sigma) h(\Sigma).
\end{equation}
Moreover, the equality holds  if and only if the marked length spectra of $\Sigma=\St/\G$ and $M=E(\Lambda)/\G$ are proportional and the proportionality constant is given by  $I_{m}(\Sigma)$.
\end{theorem}

We will prove that $I_m(\Sigma)\leq 1$ and therefore we have $\delta_\AdS(\G)\leq h(\Sigma).$ Moreover, in dimension $3$, applying this theorem to the boundary of the convex core, and using Theorem \ref{th - l'exposant critique = dim de hausdorff}, we will obtain the rigidity result announced: 
\begin{coro}\label{cor - rigidity in Fuchsian case}
Let $\G$ be a quasi-Fuchsian group in $PO(2,2)$ then 
$\Hdim_{d_x}(\Lambda)\leq 1$ with equality if and only if $\G$ is Fuchsian.
\end{coro}

%\indent First we define a map fromthe 
The following sections will be dedicated to the proof of Theorem \ref{th - inequality h<I delta}.

\subsection{Proof of the inequality in Theorem \ref{th - inequality h<I delta}}\label{subsection - inequality}
{\color{black}{In this section we prove the inequality of Theorem \ref{th - inequality h<I delta}, the proof works in any dimension.  Proposition \ref{pr - f is 1 lip} is the key for all the subsequent geometric arguments. As such, it is the main difference with the hyperbolic setting and it gives us a better understanding of the link between the ambient Lorentzian geometry  and the intrinsic geometry of Cauchy hypersurfaces.}}

In order to compare distances on the hypersurface $\St$ and  in $\AdS$ we will need the following proposition. 
\begin{prop}\label{prop- les distances ads et hilbert sont quasi-isometriques} Let $\Omega$ be a $\Gamma$-invariant open bounded convex set that contains $C(\Lambda)$, and denote by $d_H$ the Hilbert distance of $\Omega$. There is a constant $L>0$ such that $\frac{1}{L}d_H(x,y)\leq d_{\AdS}(x,y)\leq Ld_H(x,y)$ for all $x,y\in \partial_+C(\Lambda)$, where $\partial_+C(\Lambda)$ is the future boundary of the convex core.
\end{prop}

\begin{proof} Recall that a Hilbert metric is a Finsler metric, where affines lines are geodesics \cite{crampon2011dynamics}.  Denote by $N$ the Finsler norm on $T\Omega$ associated to the Hilbert metric.\\
Given $x\in C(\Lambda)$ and $v\in T_x\AdS^{n+1}$, we denote by $ v_\pm\in\partial C(\Lambda)$ \footnote{In general, $v_\pm\neq v^\pm$. There is equality if and only if $v_\pm\in \Lambda$.} the intersections of $\partial C(\Lambda)\subset \overline{\AdS^{n+1}}$ with the geodesic generated by $v$. Remark that, from the definition, $v_\pm $ does not necessary belong to the boundary of $\AdS$. \\
 Let $V=\{ v\in T\AdS^{n+1} : N(v)=1, v_\pm \in\partial_+C(\Lambda)\cup\Lambda\}$.\\
 First, notice that $V$ is $\Gamma$-invariant. Let us show that the action of $\Gamma$ on $V$ is cocompact. Let $K\subset C(\Lambda)$ be a compact set such that $\Gamma.K=C(\Lambda)$, and let $V_K=\{v\in V : v\in T_x\AdS^{n+1}, x\in K\}$. Since the maps $ v_\pm$ are continuous, it follows that $V_K$ is a closed subset of the unit tangent bundle over $K$, therefore is compact. Since $\Gamma.V_K=V$, it follows that the action of $\Gamma$ on $V$ is cocompact.  \\
As a consequence, the $\Gamma$-invariant function $ \langle v\vert v\rangle_{n,2}$ is bounded, and bounded away from $0$ (any vector in $V$ is spacelike because $\partial_+C(\Lambda)$ is a Cauchy hypersurface): let $L>0$ be such that $\frac{1}{L^2}\leq \langle v\vert v\rangle_{p,q+1}\leq L^2$ for all $v\in V$.\\
Let $x,y\in \tilde\Sigma$. Denote by $x(t)$ the anti-de Sitter geodesic going from $x$ to $y$, so that $d_{\AdS}(x,y)=\int_0^1 \sqrt{\langle \dot x(t)\vert \dot x(t)\rangle_{n,2}} dt$ and $d_H(x,y)=\int_0^1 N(\dot x(t))dt$. The tangent vectors $\dot x(t)$  are multiples of vectors of $V$, hence $\frac{N(\dot x(t))}{L} \leq \sqrt{\langle \dot x(t)\vert \dot x(t)\rangle_{n,2}} \leq L N(\dot x(t))$. Integrating this inequality yields the proposition.

\end{proof} 
%Examples of such sets $\Omega$ are $\e$-neighborhoods of $C(\Lambda)$ (given $\e>0$, it is the set of points that can be joined by a timelike geodesic of length less than $\e$ to a point of $C(\Lambda)$), or the whole black domain (except in the Fuchsian case).\\

An example of such an open convex set $\Omega$ is the whole black domain, except in the Fuchsian case. But the rigidity theorem is trivial in this case. It is shown in \cite{DGK} that it is always possible to choose $\Omega$ strictly convex with $\mathcal C^1$ boundary.\\

Recall that we defined in Definition \ref{def - definition of  the retraction f} a retraction $ f\, :\, C(\Lambda)\to \St$, which is obtained by following the integral curves of a fixed $\Gamma$-invariant  timelike vector field $V$. We can now control the distortion of the distance by the map $f$:

\begin{prop}\label{pr - f is 1 lip}
The function $f : C(\Lambda)\tv \St$ is a quasi-isometry, quasi 1-Lipschitz ie: $\exists K_1>1\, ,\,  K_2>0, \, \forall x,y \in C(\Lambda) :$
$$\frac{1}{K_1}d_{\AdS}(x,y)-K_1 \leq d_{\St} (f(x),f(y)) \leq d_{\AdS}(x,y) + K_2.$$
\end{prop}

\begin{proof}

First remark that  we can suppose that $x,y\in \St$. Indeed 
if $x,y\in C(\Lambda)$, since $d_{\AdS}(x,f(x)) =d_{\AdS}(y,f(y)) =0$ then $d_{\AdS}(x,y)-2k_\G\leq d_\AdS (f(x),f(y)) \leq d_{\AdS}(x,y) +2k_\G$ where $k_\G>0$ is given by the triangle inequality, Theorem \ref{th - triangle inequality}.

Let us prove the left  inequality. The group $\G$ acts properly discontinuously and cocompactly  on $\St$ as well as any open convex subset $\Omega$ containing $C(\Lambda)$ as in Proposition \ref{prop- les distances ads et hilbert sont quasi-isometriques}. Hence by Svarc-Milnor Lemma, \cite[Part I, Prop 8.19]{BridsonHaefliger} (that we can adapt in a straightforward way to $d_{\h^{p,q}}$) $\St$ is quasi-isometric to this neighborhood of $C(\Lambda)$, which is by Proposition \ref{prop- les distances ads et hilbert sont quasi-isometriques} Lipschitz equivalent to the $\AdS$ distance. 

We now prove the right inequality. If $x,y\in \St$, take any 3-space $E$ of signature $(1,2)$ containing $(x,y)$, its intersection with $\ADS^{n+1}$ is a copy of $\ADS^2$. In an affine chart,  it is isometric to $\left(\R\times (-\pi/2 ,\pi/2) , dt^2-\cosh^2(t) d\theta\right), $ where  $dt^2$ is the Lorentzian distance on the space like geodesic $(x,y)$. Now since $\St$ is a Cauchy hypersurface, the intersection of $E$ with $\St$ is a graph over $\R$, in particular the length $\ell =\int_0^{d_{\AdS}(x,y)} \sqrt{1-\cosh^2(t) \theta'(t)}dt\leq d_{\AdS}(x,y)$ of the curve from $x$ to $y$ in $E\cap \St$ is smaller than $d_{\AdS}(x,y)$. The  $\St$-distance between $x$ and $y$ is smaller than $\ell$, indeed since $\St$ is Riemannian the distance between two points its the minimum of the length of all curves joining them,  hence it is smaller than $d_{\AdS}(x,y)$.

\end{proof}

We give a series of simple corollaries that we will use during the proof of Theorem \ref{th - inequality h<I delta}. 
Corollaries \ref{cor - comparaison des boules}, \ref{cor les ombres sont comparables} and \ref{cor covering}  are the equivalents in our settings of well known comparison results between shadows and balls. They are good examples of what kind  of behavior we have to control using $\AdS$ geometry in this Lorentzian setting. 

\begin{coro}\label{pr - I geq  1}
For every ergodic $\mu$, 
$$I_\mu(\Sigma)\leq 1.$$ 
\end{coro}

\begin{proof}
Let $v$ be a typical vector for $\mu$  we have 
\begin{eqnarray*}
\frac{a(v,t)}{t} &=& \frac{d_{\St} (f(\pi \phi_t (v)) , f(\pi v ) )}{d_{\AdS} (\pi \phi_t (v) , \pi v  )}\\
						&\leq & 1+\frac{K}{t},
\end{eqnarray*}
taking the limit concludes the proof. 
\end{proof}

Since $\St$ is supposed to be negatively curved, we can apply the Morse Lemma and obtain the following result: 
\begin{coro}\label{cor  - f geodesic are quasi-geodesic}
There exists $K>0$ such that, for all $p\in C(\Lambda)$ and all $ \xi\in \Lambda$, the image by $f$ of the geodesic ray $[p,\xi)$,  $f([p,\xi)) $, is at distance ($d_{\St}$)  at most $K$ of the unique geodesic on $\St$ from $p$ to $\xi.$
\end{coro}
\begin{proof}
Let $\xi\in \Lambda$ and  $v_p(\xi)$ be the  vector in $T_p^1 C(\Lambda) $ such that $\lim_{t\tv \infty} \phi^{\AdS}_t (v_p(\xi)) =\xi$. Consider the curve $c : \R^+ \tv \St$ defined by $c(t):= f(\pi\phi^{\AdS}_t (v_p(\xi)))$. Then from Proposition \ref{pr - f is 1 lip}, for all $t,s\in \R^+$ we have 
$$ \frac{|t-s|}{K_1} -K_1 \leq d_{\St} (c(t),c(s) ) \leq |t-s|+K_2.$$
Hence $c$ is a quasi-geodesic on $\St$. Thanks to the Morse Lemma, the quasi-geodesic $c$  is at distance ($d_{\St}$) at most $K$ of a unique geodesic \cite{ballmann1985manifolds}. This constant $K$ depends only on $\St$ and $K_1,K_2$, in particular it does not depends on $p$ and $\xi$. 
\end{proof}

For the next two corollaries we can give explicit constants : $K_1,K_2$ are given by Proposition \ref{pr - f is 1 lip} and $K$ by the previous Lemma.

\begin{coro}\label{cor - comparaison des boules}
For all $R$ and all $x\in C(\Lambda)$:
$$B_{\AdS} (x, R ) \cap \St \subset B_{\St} (f(x),R+K_2).$$

For all $R$ and all $x\in \St$:
$$B_{\St} (x, R ) \subset B_{\AdS} (x,K_1 R +K_1^2.).$$

\end{coro}

\begin{proof}
We prove the first inclusion. Let $y \in B_{\AdS} (x, R ) \cap \St$, since $y\in \St$, $f(y)=y$ and we have:
\begin{eqnarray*}
d_{\St} (f(x),y) &=&d_{\St} (f(x),f(y))\\
						&\leq & d_{\AdS} (x,y) +K_2=R+K_2,
\end{eqnarray*}
thanks to Proposition \ref{pr - f is 1 lip}. 

\bigskip

We prove the second inclusion. Let $y \in B_{\St} (x,R) $ then 
\begin{eqnarray*}
d_{\AdS}(x,y) &\leq & K_1 d_{\St} (f(x),f(y))+K_1^2 \\
						&\leq & K_1 R +K_1^2.
\end{eqnarray*}

\end{proof}

We fix once for all a point $p\in C(\Lambda)$. The Lorentzian shadows  for the metric $d_{\h^{p,q}}$ centered at $x$  and  seen from $p$ will be denoted by  $\mathcal{S}_{\AdS} (x,R)$.We have  $\xi \in\mathcal S_{\AdS} (x,R)$, if and only if there exists the intersection between the geodesic ray $[p \xi)$ and the $\AdS$ balls $B_\AdS(x,R)$ is non-empty. In a similar way, the shadows in $\St$ are supposed to be centered in $f(p)$ and will be denote by  $\mathcal S_{\St}(y,R)$.

\begin{coro}\label{cor les ombres sont comparables}
For all $R>0$ and all $x\in C(\Lambda)$ 
$$\mathcal S_{\AdS} (x,R)\subset \mathcal S_{\St}(f(x),R+K+K_2).$$

For all $R>0$ and  all $x\in \St$ 
$$\mathcal S_{\St} (x,R) \subset \mathcal S_{\AdS}(x,K_1 (R+K) + K_1^2),$$
where $K_1,K_2$ are given by Proposition \ref{pr - f is 1 lip} and $K$ by Corollary \ref{cor  - f geodesic are quasi-geodesic}.
\end{coro}

\begin{proof}
Let $\xi$ be in $\mathcal S_{\AdS} (x,R)$, then by definition the $\AdS$ geodesic 
$[p\xi)_{\AdS}$ intersects $B_{\AdS} (x,R)$. By  Corollary \ref{cor - comparaison des boules}, this implies that 
$$f([p\xi)_{\AdS}) \cap B_{\St} (f(x), R+K_2)\neq \emptyset.$$
This implies that the unique geodesic from $f(p)$ to $\xi $ intersects $B_{\St} (f(x), R + K_2+K)$, where $K$ is the constant in Morse Lemma,  Corollary \ref{cor  - f geodesic are quasi-geodesic}.  In other words, $\xi$ belongs to $ \mathcal S_{\St} (f(x),R + K_2+K)$.

\bigskip

Similarly, let $\xi$ be in $\mathcal S_{\St} (x,R)$, then by definition the geodesic $[p\xi)_{\St}$ of $\St$ for the induced metric  intersects $B_{\St}(x,R)$. Let $[p\xi)_{\AdS}$ be the $\AdS$ geodesic, then by Corollary \ref{cor  - f geodesic are quasi-geodesic} $f([p\xi)_{\AdS}) $ is at distance at most $K$ of $[p\xi)_{\St}$, therefore there exists  $z\in [p\xi)_{\AdS}$ such that $d_{\St}(f(z), x) \leq R+K$. We then have 
\begin{eqnarray*}
d_{\AdS}(z,x) &\leq & K_1 d_{\St} (f(z), f(x)) +K_1^2\\
						&\leq & K_1 (R+K) + K_1^2
\end{eqnarray*}
\end{proof}

The last corollary concerns covering of subsets in $\AdS$, this is clearly a problem that appears since we do not work in a Riemannian setting: 
\begin{coro}\label{cor covering}
For all $R>4K_1^2$ there exists $\epsilon>0$ such that for all subset $A\subset C(\Lambda)$, there exists a covering of $A$ by $AdS$ balls $B_{\AdS}(x_i,R)$, $x_i$ in $A$, and such that $B_{\St}(f(x_i),\epsilon)\cap B_{\St}(f(x_j),\epsilon)=\emptyset$ for all $i\neq j$. 
\end{coro}
\begin{proof}
Let $A\subset B_{\AdS}(x_i,R)$ be a covering such that $d_{\AdS} (x_i,x_j) >R$ for all $i\neq j$ (to produce such a covering take by induction  $x_{n+1}\in A\setminus \cup_{i=1}^n  B_{\AdS}(x_i,R)$.) Let $\epsilon \in (0,\frac{R}{4K_1}-K_1)$. 

Pick $z \in B_{\St}(f(x_i),\epsilon) \cap B_{\St}(f(x_j),\epsilon)$.  Then by Proposition \ref{pr - f is 1 lip} we have
\begin{eqnarray*}
d_{\h^{p,q}}(x_i,x_j) &\leq & K_1 d_{\St} (f(x_i),f(x_j)) +K_1^2\\
							&\leq & 2\epsilon K_1 +K_1^2\\
							&\leq & R/2.
\end{eqnarray*}
Hence $x_i=x_j$. 
\end{proof}

We will prove the inequality of Theorem \ref{th - inequality h<I delta}, using comparaison of volume. The volume for a pseudo-Riemannian manifold is defined by the canonical volume form. In local coordinates it is defined by 
$\omega = \sqrt{|\det(g)|} dx^1\wedge...\wedge dx^n$.

The following lemma is  classical in Riemannian setting, the proof is similar for $\AdS$: 
\begin{lemme}\label{lem - entropy egal volume des boules}
$$\delta_\AdS( \G)  = \lim_{R\tv \infty} \frac{1}{R}\log \vol(B_{\AdS} (o,R) \cap C(\Lambda)).$$
 
\end{lemme}
\begin{proof}
It is a consequence of the cocompactness of the action on $C(\Lambda)$. It is sufficient to cover $B_{\AdS} (o,R) \cap C(\Lambda)$ with translates of a compact fundamental domain and then taking the limit. 
\end{proof}

%For $p\in C(\Lambda)$ and $\xi \in \Lambda$ we note $v_p(\xi)$ the unit (space like) vector at $p$ such that $\phi_{+\infty} (v_p(\xi) ) =\xi$. 
Since the proof consists in comparing the Patterson-Sullivan measures on $\Lambda$ associated to $\AdS$ and to $\St$ we will name these two measures $\nu^{\AdS}$ and $\nu^{\St}$ respectively.  Similarly, the two Gromov functions defined in Subsection \ref{subsec - gromov distance} on $\Lambda$ associated to $\AdS$ and $\St$ distances  will be denoted by $d_p^\AdS(\cdot, \cdot) $ and $d_p^\St(\cdot, \cdot) $ for $p\in C(\Lambda)$ or $p\in \St$ depending on the context. \\

The next lemma allows to consider dynamical behavior at a fixed point:
\begin{lemme}\label{lem - la structure produit de Bowen margulis permet de se restreindre a PS}
For all $p\in C(\Lambda)$ and for  $\nu^{\AdS}_p$-a.e. $\xi\in \Lambda$, 
$$\lim_{t\tv \infty} \frac{a(v_p(\xi),t)}{t}=I_{m}(\St).$$
\end{lemme}
%\begin{proof}
%This is a consequence of the product structure of $m$. We identify $N^1 C(\Lambda) $ with $\Lambda^{(2)} \times \R$. Since $a(v,t)$ is $\G$ invariant, we see that for $\frac{d\nu_p^{\AdS}(\xi) d\nu_p^{\AdS}(\eta) dt}{d_p^{\AdS}(\xi,\eta)}$ a.e. $(\xi,\eta,t)$ the limit $\lim \frac{a(v_p(\xi),t)}{t}$ exists and is  equal to $I_{m}(\St).$ Here $v_p(\xi)$ is the vector directing $(\eta,\xi),$ based on $p$. Clearly, if $\frac{a(v_p(\xi),t)}{t}$ admits a limit then $\frac{a(\phi_{t_0}^{\AdS} v_p(\xi),t)}{t}$ admits the same limit. Hence for $\frac{d\mu_p^{\AdS}(\xi) d\mu_p^{\AdS}(\eta)}{d_p^{\AdS}(\xi,\eta)}$ a.e. $(\xi,\eta)$  the limit $\lim \frac{a(v(\xi),t)}{t}$ exists and is  equal to $I_{m}(\St)$. Here $v(\xi)$ is any vector directing $(\eta,\xi).$ 
%
%By Lemma \ref{lem - Two space like rays with the same endpoint are at bounded distance.} if $v(\xi)\in (\eta,\xi) $ satisfies $\lim \frac{a(v(\xi),t)}{t}= I_{m}(\St)$   then $v'(\xi)\in (\eta',\xi) $ satisfies $\lim \frac{a(v'(\xi),t)}{t}= I_{m}(\St)$ for \emph{every} $\eta'\in \Lambda\setminus \xi$
%
%Hence for $\nu^{\AdS}_p$-a.e. $\xi\in \Lambda$, 
%$$\lim_{t\tv \infty} \frac{a(v_p(\xi),t)}{t}=I_{m}(\St).$$
%\end{proof}

\begin{proof}
Recall $R^{1}C(\Lambda)$ is the set of unit recurrent vectors, ie. space-like vectors in $TC(\Lambda)$ of norm $1$, such that $v^+\in \Lambda$.  In particular we have  $R^{1}C(\Lambda)\supset N^1C(\Lambda)$. We then define 
$$P=\left\{v\in R^1 C(\Lambda)\, |\, \lim_{t\tv \infty} \frac{a(v,t)}{t}=I_m(\Sigma)\right\}.$$ 
Let $p\in C(\Lambda)$ we define $A :=\{ \xi \in \Lambda\, |\, v_p(\xi) \in P\}$ and we want to show that $A$ is of full $\nu^\AdS_p$ measure. Remark that $A$ is  a $\G$-invariant subset of $\Lambda$, therefore by the ergodicity of $\nu^{\AdS}_p$ it  is of full or zero $\nu^{\AdS}_p$ measure.

 %By ergodicity of $m$ we know that $P\cap N^1C(\Lambda)$ is of full $m$ measure. 

Let $v\in P$, that we write in Hopf coordinates: $v=(v^-,v^+, t_0)$, in particular here, $v^-$ is not necessary in $\Lambda$. Let   $w$ in the same "stable leaf" as $v$, that is $w=(\eta, v^+, s)$, with $\eta \in \partial U(p)$ and $s\in \R$.  We have:
$$a(v,t) \leq d_{\St} (f(\pi \phi_t (v)) ,f(\pi \phi_t (w)) )+ a(w,t)+d_{\St} ( f(\pi w ) , f(\pi v ).$$
By Lemma \ref{lem - Two space like rays with the same endpoint are at bounded distance.}  and Proposition \ref{pr - f is 1 lip}, $d_{\St} (f(\pi \phi_t (v)) ,f(\pi \phi_t (w)) )$ is uniformly bounded. This implies that $\lim_{t\tv \infty} \frac{a(v,t)}{t} \leq \lim_{t\tv \infty} \frac{a(w,t)}{t}$. Interchanging the role of $v$ and $w$ we see that $w\in P.$ 
Hence 
$$N^1C(\Lambda)\cap P = \{ (\xi, \eta, t) \, |\, \eta\in A, \xi \in \Lambda\setminus\{\eta\}, t\in \R\}. $$
By ergodicity of $m$ we know that $P\cap N^1C(\Lambda)$ is of full $m$ measure (cf. Theorem \ref{th - ergodicity implies a(v,t)/t has almost everywhere a limit}), and therefore by the product structure of $m$, $A$ is of full $\nu^{\AdS}_p$-measure.

\end{proof}

We are finally ready to prove the inequality in Theorem \ref{th - inequality h<I delta}.
\begin{proof}[Proof of the inequality in Theorem \ref{th - inequality h<I delta}]
Let  $p\in C(\Lambda)$. By the previous lemma, for all $\kappa>0$ and $T>0$ we define the set 
$$A_p^{T,\kappa} = \left\{ \xi \in \Lambda| \left|  \frac{a(v_p(\xi),t)}{t}-I_{m}(\Sigma) \right| \leq \kappa, \quad t\geq T\right\}.$$
For all $d\in ]0,1[$ and all $\kappa>0$, there exists $T>0$ such that $\nu_p^{{\AdS}}(A_p^{T,\kappa}) \geq d$. Let $c_{p,\xi} (t) =  \pi(\phi_t^\AdS (v_p (\xi)) )$ be the geodesic of $\AdS$. For $t>T$ consider the subset $\{c_{p,\xi}  (t) | \xi \in A_p^{T,\kappa}\} \subset S_{\AdS}(p,t)$ of the Lorentzian sphere of radius $t$ and center $p$ in $\AdS$. 

Choose a covering of this subset by balls $B_{\AdS}(x_i,R)$ with $R$ sufficiently large such that both Corollary \ref{cor covering}  and Shadow Lemma for $\nu^{\AdS}_p$ apply. 
Then, by the local behavior of $\nu^{\AdS}_p$, there exists a constant $c>1$ independent of $t $ such that 
$$\frac{1}{c} e^{-\delta t}  \leq \nu_p(\mathcal S_{\AdS}(x_i,R)))\leq c e^{-\delta t}.$$
It is clear by definition that $A_p^{T,\epsilon}  \subset \cup_{i\in I} \mathcal S_{{\AdS}}(x_i,R))$ and therefore, 
\begin{eqnarray}\label{eq -1 preuve de l'inégalité}
d \leq  \nu_p^{{\AdS}} \left(  \cup_{i\in I} \mathcal S_{{\AdS}}(x_i,R))  \right)\leq \sum_{i\in I} \nu_p^{{\AdS}} (\mathcal S_{\AdS}(x_i,R))) \leq c \Card(I) e^{-\delta t}.
\end{eqnarray}

By the property of the covering in Corollary \ref{cor covering} the balls $\left\{ B_{\St} (f(x_i), \epsilon)\right\}_{i\in I}$ are disjoint. Moreover $d_{\St}(f(p),f(x_i) ) \leq t(I_{m}(\St) +\kappa))$ by definition of $A_p^{T,\kappa} $. Hence the balls $B_{\St} (f(p),  t(I_{m}(\St) +\kappa) +\epsilon)$ contains the disjoint union $\sqcup_{i\in I} B_{\St} (f(x_i), \epsilon)$.
Let $v:=\min_{i\in I} \Vol B_{\St} (f(x_i), \epsilon).$ Then 
\begin{eqnarray}\label{eq -2 preuve de l'inégalité}
\Vol B_{\St} (f(p),  t(I_{m}(\St) +\kappa) +\epsilon) \geq v \Card (I).
\end{eqnarray}

By Equations (\ref{eq -1 preuve de l'inégalité}) and (\ref{eq -2 preuve de l'inégalité}), we have 
\begin{eqnarray*}
e^{\delta t} &\leq &\frac{c}{d} \Card(I)\\
					&\leq &\frac{c}{v d} \Vol B_{\St} (f(p),  t(I_{m}(\St) +\kappa) +\epsilon).
\end{eqnarray*}
We conclude using Lemma \ref{lem - entropy egal volume des boules} since $\kappa$ is arbitrary. 
\end{proof}

\subsection{Proof of rigidity}
{\color{black}{
We will  use the following ergodic theoretic result of G. Knieper: 
\begin{theorem}\cite{knieper1995volume}
Let $(X,\mu, T)$ be an ergodic measure preserving system. Let $a : X\times \N\tv \R$ a measurable subadditive cocycle. Assume that there exists a constant $L>0$ such that for all $n,l\in /N:$
$$a(x,n)+ a(T^n x, l) \leq a(x,n+l) +L.$$
Then there is a constant $\alpha\in \R$ such that for $\mu$ almost all $x\in X$ we can find a sequence $n_j\tv \infty$ such that 
$$| a(x,n_j) -n_j\alpha | \leq 2L.$$
In this case, we have necessarily  $\lim_{n\tv \infty}\frac{a(x,n)}{n}= \alpha $ for almost every $x$. 
\end{theorem}}}

We prove that the cocycle $a$ previously defined satisfies the hypothesis of this theorem: 
\begin{prop}\label{cor - a is sous additif}
There exists $C>0$ such that for every $v\in N^1C(\Lambda)$, $t_1,t_2>0$ we have 
$$a(v,t_1) +a(\phi_{t_1}^{\AdS}(v),t_2) \leq a(v,t_1+t_2) +C$$
\end{prop}
\begin{proof}
This is a consequence of Corollary \ref{cor  - f geodesic are quasi-geodesic}. We fix $v\in N^1C(\Lambda)$ and show the inequality with a constant that does not depend on $v$. We denote by $\xi$ the endpoint of the geodesic ray tangent to $v$, that is $\xi :=v^+$.  Let $z_1$ be a point on the geodesic $[f(p)\xi)_{\St}$ such that $d_{\St} (z_1, f(\pi \phi_{t_1}^{\AdS} (v)) \leq K$. Let $z_2$ be a point on the geodesic $[f(p)\xi)_{\St}$ such that $d_{\St} (z_2, f(\pi \phi_{t_1+t_2}^{\AdS} (v))) \leq K$. Then 
$$a(v,t_1)= d_{\St} (f(p) , f( \pi \phi_{t_1}^{\AdS} (v)))\leq d_{\St} (f(p),z_1) +K  $$
and 
$$a(\phi_{t_1}^{\AdS}(v),t_2)=d_{\St} ( f( \pi \phi_{t_1}^{\AdS} (v)),f(\pi \phi_{t_1+t_2}^{\AdS} (v))) \leq d_{\St}(z_1,z_2) +2K.$$
Hence 
\begin{eqnarray*}
a(v,t_1) +a(\phi_{t_1}^{\AdS}(v),t_2) %&%=& d_{\St} (f(p) , f( \pi \phi_{t_1}^{\AdS} (v))) +d_{\St} ( f( \pi \phi_{t_1}^{\AdS} (v)),f(\pi \phi_{t_1+t_2}^{\AdS} (v)))\\
&\leq & d_{\St} (f(p),z_1) +K +d_{\St}(z_1,z_2) +2K\\
&\leq & d_{\St} (f(p),z_2) +3K \\
&=& a(v,t_1+t_2) +4K.
\end{eqnarray*}
The proposition follows with $C=4K$. 
\end{proof}

%As it is explained in \cite{knieper1995volume} it is a consequence of Corollary \ref{cor - a is sous additif} that for $m$ a.e. $v\in N^1 C(\Lambda)/\G$ there is a subsequence $t_n$ such that  $|a(v,t_n)-  I_{m} (\St) t_n | \leq L.$ As in Lemma \ref{lem - la structure produit de Bowen margulis permet de se restreindre a PS} we pass from $m$ a.e. $v\in N^1 C(\Lambda)/\G$  to $\nu_p^{\AdS}$ a.e. $\xi\in \Lambda$

Therefore Knieper's Theorem gives: 
\begin{theorem}
There exists a constant $L$ such that for $m$ almost every $v\in N^1C(\Lambda)/\Gamma$, there is  a sequence $t_n\tv \infty $ such that :
$$|d_{\St}( f(p), f (\pi  \phi_{t_n} ^{\AdS}v) ) - I_{m} (\St) t_n | \leq L.$$
\end{theorem}
As in Lemma \ref{lem - la structure produit de Bowen margulis permet de se restreindre a PS}, using the product structure of $m$, we can pass from $m$ a.e. $v\in N^1 C(\Lambda)/\G$  to $\nu_p^{\AdS}$ a.e. $\xi\in \Lambda$ and we get: 
\begin{lemme}
There exists a constant $L$ such that for $\nu_p^{\AdS}$ almost all $\xi \in \Lambda$ there is a sequence $t_n\tv \infty $ such that 
$$|d_{\St}( f(p), f (\pi  \phi_{t_n} ^{\AdS}v_p(\xi)) ) - I_{m} (\St) t_n | \leq L.$$
\end{lemme}

\begin{prop}
If there is equality in Eq.(\ref{eq - th inequality h <idelta}), then $\nu^\St$ is absolutely continuous  with respect to $\nu^\AdS$.
\end{prop}

\begin{proof}
Let $\xi\in \Lambda $ be a generic point for $\nu_p^{\AdS}$ and set $y_n :=  \pi \phi_{t_n} ^{\AdS}v_p(\xi)$. From the previous lemma we have 
\begin{equation}\label{eq 1- prop égalité}
|d_{\St}( f(p), f (y_n)) - I_{m} (\St) t_n | \leq L.
\end{equation} 
Let $R$ be large enough for the shadow lemma to hold in both $\St$ and $\AdS$.  According to \ref{cor - comparaison des boules} there exists  $R',R''>R$  such that for all $x\in \St$ we have :
$$\mathcal S_{\AdS}(x,R)  \subset \mathcal{S}_{\St} (x,R') \subset \mathcal S_{\AdS} (x,R'').$$
Applying $\nu^{\AdS}$ and letting $x=f(y_n)$ we get : 
$$\nu^{\AdS} (\mathcal S_{\AdS}(f(y_n),R) ) \leq \nu^{\AdS} (\mathcal S_{\St}(f(y_n),R') )\leq \nu^{\AdS} (\mathcal S_{\AdS}(f(y_n),R'') ).$$
By the Shadow Lemma for $\nu^{\AdS}$ there exists $c>1$ such that 
$$\frac{1}{c}  e^{-\delta  d_{\AdS}(f(p),f(y_n))}  \leq \nu^{\AdS} (\mathcal S_{\St}(f(y_n),R') )\leq c e^{-\delta  d_{\AdS}(f(p),f(y_n))}   .$$
By the  Shadow Lemma for  $\nu^{\St}$ there exists $c_2>1$ such that 
$$\frac{1}{c_2} e^{-h(\St) d_{\St}(f(p)f(,y_n))} \leq \nu^{\St} ( \mathcal S_{\St}(f(y_n),R') ) \leq c_2  e^{-h(\St) d_{\St}(f(p),f(y_n))}.$$
From equation \ref{eq 1- prop égalité}, there exist $c_3>1$ such that 
$$\frac{1}{c_3}\leq \frac{e^{-h(\St) d_{\St}(f(p)f(,y_n))} }{ e^{-\delta  d_{\AdS}(f(p),f(y_n))} }\leq c_3$$
Hence there exists $c_4>1$ such that 
$$\frac{1}{c_4}\leq  \frac{\nu^{\AdS} (\mathcal S_{\St}(f(y_n),R') )}{\nu^{\St} ( \mathcal S_{\St}(f(y_n),R') ) } \leq c_4.$$
Since $\mathcal S_{\St}(f(y_n),R') )\tv_{n\tv \infty} \xi$, this concludes the proposition. 
\end{proof}

\begin{prop}
If the Patterson-Sullivan measures $\nu^{\Sigma}$ are absolutely continuous with respect to $\nu^{AdS}$ then the Gromov functions  on $\Lambda$ seen as $\partial C(\Lambda) \cap \partial \ADS^{n+1} $ or $\partial \St$ are Hölder equivalent. 
\end{prop}

\begin{proof}
Consider on $\Lambda^{(2)}$ the Bowen-Margulis currents defined by 
$$\mu_\St (\xi,\eta) = \frac{d\nu_p^\St (\xi) d\nu_p^\St(\eta)}{d^{\St}_p(\xi,\eta) ^{2h}}$$
$$\mu_{\AdS} (\xi,\eta) = \frac{d\nu_p^{\AdS} (\xi) d\nu_p^{\AdS}(\eta)}{d_p^{\AdS}(\xi,\eta) ^{2\delta}}.$$

By assumption $\nu_p^\St$ is absolutely continuous with respect to $\nu_p^{\AdS}$ , therefore  $\mu_\St$ is absolutely continuous with respect to  $\mu_{\AdS}$,  there exist a measurable function $f:\Lambda^{(2)} \tv \R$ such that  $\mu_\St=f(\xi,\eta) \mu_\AdS$. The ergodicity and the  $\G$-invariance implies that $f$ is constant almost everywhere and therefore there exist $c>0$ such that
$$\mu_\St=c\mu_{\AdS}.$$
%Pour le voir on peut considérer les ensembles $A(c)=\{ \xi \, |\, \frac{d\mu_\St (\xi) }{d\mu_{\AdS}(\xi)} =c\}$, qui sont $\G$ invariant, donc de mesure nulle ou pleine. 

Since $\nu_p^{{\St}}$ is absolutely continuous with respect to  $\nu_p^{\AdS} $ there exists an \textit{a priori} measurable function  $u :\Lambda \tv \R^+$ such that $\nu_p^{{\St}} (\xi) =u(\xi) \nu_p^{\AdS} (\xi)$. 
However, by definition of the measures $\nu_p^{{\St}}$ and $\nu_p^{\AdS} $, we can give an explicit formula for $u$, we have for all $\xi,\eta \in \Lambda^{(2)}$:
$$u(\xi)u(\eta) d_p^{\AdS}(\xi,\eta)^{\delta} = c d_p^\St (\xi,\eta)^{h}.$$
Let $\eta, \eta'\in \Lambda$, with $\eta \neq \eta'$. We see that $u$  is a continuous function on $\Lambda\setminus \eta$ : 
$u(\xi) := \frac{c d_p^\St (\xi,\eta)^{h}}{ u(\eta) d_p^{\AdS}(\xi,\eta)^{\delta} } $. Similarly it is continuous on let $\Lambda\setminus\eta'$. Therefore, $u$ is continuous on $\Lambda$.  By compactness, there exists $C>1$ such that $\frac{1}{C}\leq u(\xi) \leq C$. 
Finally we get what we stated 
$$\frac{c}{C^2} d_p^\St (\xi,\eta)^{h} \leq d_p^{\AdS}(\xi,\eta)^{\delta}  \leq C^2 c d_p^\St (\xi,\eta)^h.$$
\end{proof}

\begin{prop}\label{pr - gromov equivalent--> length spectra proportional}
If the two Gromov functions coming from the distance on $\St$ and on $C(\Lambda)$ are Hölder equivalent then the marked length spectra of $M=E(\Lambda)/\G$ and $\St/\G$ are proportional. 
\end{prop}

\begin{proof}
Let $g\in \G$ and $\xi \in \Lambda\setminus\{g^{\pm}\}$.Then 
$$[g^-, g^+,g(\xi),\xi]_{\AdS} = e^{\ell_{\AdS}(g)}$$
and
$$[g^-, g^+,g(\xi),\xi]_{\Sigma} = e^{\ell_{\Sigma}(g)}$$
By assumption on the Gromov  functions $d_\St,d_{\h^{p,q}}$ and from the definition of the cross-ratio, Eq. (\ref{eq - cross ratio}),  there exists $C>1$ such that for all $g\in \G$ we have
$$\frac{1}{C} e^{r \ell_{\AdS}(g) }\leq e^{\ell_\Sigma(g)} \leq C e^{r \ell_{\AdS}(g)}.$$
In particular when we look at the power $g^n$ of $g$, taking the $\log$, we get for all $n>0$ and all $g\in \G$ :
$$-\log(C) +rn\ell_{\AdS}(g) \leq n\ell_\Sigma (g) \leq \log(C)+rn\ell_{\AdS}(g).$$
We finish the proof by dividing by  $n$ and taking the limit : 
$$\ell_\Sigma (g) =r\ell_{\AdS}(g).$$
\end{proof}

Putting everything together we obtain the equivalent of Bowen's Theorem for $\AdS$ quasi-Fuchsian manifolds
\begin{coro}\label{cor - Bowen GHMC}If $\G\subset \PO(2,2)$ is quasi-Fuchsian, then
$\Hdim_{d_x}(\Lambda)\leq 1$ with equality if and only if $\G$ is Fuchsian. 
\end{coro}

\begin{proof}
By Theorem \ref{th - l'exposant critique = dim de hausdorff}, we know that $\Hdim_{d_x}(\Lambda) = \delta_\AdS(\G)$, so it is sufficient to prove the rigidity for $\delta_\AdS(\G)$. 

The connected components of the boundary of the convex core are isometric to $\Hyp^2$, therefore their volume entropy is $1$. 
Applying Theorem \ref{th - inequality h<I delta} and Proposition \ref{pr - I geq  1} to these surfaces we have $\delta\leq 1$ with equality if and only if the boundary of the convex core has the same length spectrum as $M=E(\Lambda)/\G$. Let $M$ be parametrized in the Mess' parametrization by $S_1$ and $S_2$. Recall that for a homotopy class of closed curved $c$, the geodesic length of $c$ in M is $\ell_{\AdS}(c) =\frac{\ell_1(c)+\ell_2(c)}{2}$ see \cite{glorieux2015asymptotic}.  Then, corollary \ref{cor - Bowen GHMC} is a consequence of the following hyperbolic geometry result showing that $S_1=S_2$. 
\end{proof}

\begin{lemme}
Let $S_1,S_2,S_3$ be three hyperbolic surfaces. Let $\ell_j$ be their corresponding length functions. If for all closed homotopy classes of closed curves $c$ we have $\ell_1(c) +\ell_2(c) =2 \ell_3(c)$ then $ S_1=S_2$ ($=S_3$).
\end{lemme}

\begin{proof}
Recall from Bonahon, that the function $\ell_j$ extends continuously on the set of geodesic currents. They are the restriction of the intersection function with the Liouville current $L_j$ : for all $c\in \mathcal{C}$ we have $\ell_j(c) = i(c,L_j)$, \cite[Proposition 14]{BonahonGeometryofteichmuller}. Moreover, the set of closed geodesics is dense in the set of currents \cite[Proposition 2]{BonahonGeometryofteichmuller}. Hence by applying to $L_3$ the condition on the length functions we get
$$i(L_1,L_3)+ i(L_2,L_3) =2 i(L_3,L_3).$$
Here again, we use Bonahon result saying that for any two Liouville currents $L,L'$ we have $i(L,L') \geq i(L,L)$ with equality iff $L=L'$, \cite[Theorem 19]{BonahonGeometryofteichmuller}. It implies that 
$L_1=L_3$ and $L_2=L_3$. By Bonahon's  \cite[Lemma 9]{BonahonGeometryofteichmuller} we have 
$$S_1=S_2=S_3.$$
\end{proof}

\section{Appendix}
In this appendix, we  fix a $\h^{p,q}$-convex cocompact  group $\G\subset \PO(p,q+1)$. 

\subsection{Existence of conformal densities}

We fix a base point $o\in C(\Lambda)$. Let $P$ be the  Poincaré series associated to $\G$ :
$$P(s)=\sum_{\g\in \G}   e^{-s d_{\h^{p,q}}(\g o, o )}.$$

A simple computation using the definition of the critical exponent $\delta_{\h^{p,q}}(\Gamma)$ shows that $P(s)<+\infty$ when $s>\delta_{\h^{p,q}}(\Gamma)$, and $P(s)=+\infty $ if $s<\delta_{\h^{p,q}}(\Gamma)$.\\
However, as in the hyperbolic case, we don't know in advance that $P$ diverges at $\delta_{\h^{p,q}}(\G)$. In order to solve this problem, S. J.Patterson  proposed a modification of this Poincaré series using an auxiliary function $h$ {\color{black} which applies to any Dirichlet series (i.e. series of the form $\sum_{n\in \N}a_n^{-s}$).}

\begin{lemme}[Lemma 3.1 in \cite{patterson1976limit}]\label{lem - patterson lemma}
There exists an increasing function $h\, :\, \R^+\tv\R^+ $ such that  
\begin{itemize}
\item $\sum_{\g\in \G} h(d_{\h^{p,q}}(\g o, o))  e^{-s d_{\h^{p,q}}(\g o, o )}$ has the same radius of convergence as $P(s)$ and diverges at $s=\delta_{\h^{p,q}}(\G)$.
\item For all $\varepsilon >0$  there exists $y_0>0$ such that for all $y>y_0$, $x>1$ we have : 
\begin{eqnarray*}
h(y+x) < e^{\varepsilon x} h(y).
\end{eqnarray*}
\end{itemize}

\end{lemme}
%
%
%\begin{lemme}
%Let $\sum_{n\in \N} a_n^{-s}$ be a Dirichlet series, of convergence radius equal to $\delta$. There exists an increasing function 
%$h : [0,\infty [ \tv  [0,\infty[$ such that :
%\begin{itemize}
%\item  The radius of convergence of the series $\sum h(a_n) a_n^{-s} $   is $\delta$. The series diverges at $\delta$.
%\item For all $\varepsilon >0$  there exists $y_0>0$ such that for all $y>y_0$, $x>1$ we have : 
%\begin{eqnarray*}
%h(y+x) < x^{\varepsilon} h(y).
%\end{eqnarray*}
%\end{itemize}
%\end{lemme}

We call $Q$ the modified Poincaré series : 
$$Q(s)=\sum_{\g\in \G}  h(d_{\h^{p,q}}(\g o, o )) e^{-s d_{\h^{p,q}}(\g o, o )}.$$

{\color{black} From now on, we consider measures on the metrizable compact set $\overline{\h^{p,q}}=\h^{p,q}\cup\partial\h^{p,q}\subset \R\proj^{p+q}$.\\}
We denote by $\Delta_x$ the Dirac mass at $x$.  We consider the following family of measures for every $x\in C(\Lambda),\, s>\delta_{\h^{p,q}}(\G)$:
$$\mu_x^s = \frac{\sum_{\g\in \G}    h(d_{\h^{p,q}}(\g o, x ))e^{-s d_{\h^{p,q}}(\g o,x)} \Delta_{\g o}}{Q(s)}.$$
{\color{black} We now wish to consider a converging subsequence (for the weak topology) of these measures as $s\to\delta_{\h^{p,q}}(\G)$. In order to prove the existence of such limits, we must show that the total mass is bounded. For these limits to be non trivial, we also need to find a lower bound for the total mass.

\begin{lemme} \label{lemme - masse totale bornee}
For any $x\in C(\Lambda)$, there is $M(x)>0$ such that   \[\frac{1}{M(x)}\leq \mu_x^s(\overline{\h^{p,q}}) \leq M(x)\] for all $s>\delta_\AdS(\G)$.
\end{lemme}
\begin{proof} By Lemma \ref{lem - patterson lemma}, we can consider $y_0>0$ such that:
\[ \forall y>y_0~\forall x>1~~h(y+x)<e^xh(y).\]
Since $h$ is increasing, we also have:
\[ \forall y>y_0~\forall x\geq 0~~h(y+x)<e^{x+1}h(y).\]
Let $\G_+=\{\g\in\G \vert d_{\h^{p,q}}(\gamma.o,o)>y_0\}$. Note that $\G\setminus\G_+$ is finite. Remembering that $h$ is increasing, we find:
\begin{eqnarray*} \mu_x^s(\overline{\h^{p,q}})&=&\frac{\sum_{\g\in \G}    h(d_{\h^{p,q}}(\g o, x ))e^{-s d_{\h^{p,q}}(\g o,x)}}{Q(s)}\\
&\leq & \frac{\sum_{\g\in \G_+}    h(d_{\h^{p,q}}(\g o, o )+d_{\h^{p,q}}(x,o)+k_\G)e^{-s d_{\h^{p,q}}(\g o,x)}}{Q(s)} + \frac{\sum_{\g\in\G\setminus\G_+}h(y_0)}{Q(s)}\\
&\leq & \frac{\sum_{\g\in \G_+} e^{ d_{\h^{p,q}}(x,o)+k_\G+1}   h(d_{\h^{p,q}}(\g o, o ))e^{-s d_{\h^{p,q}}(\g o,x)}}{Q(s)}+\frac{h(y_0)\sharp\G\setminus \G_+}{Q(s)}\\
&\leq &e^{ d_{\h^{p,q}}(x,o)+k_\G+1} \frac{Q(s)}{Q(s)}+\frac{h(y_0)\sharp\G\setminus \G_+}{Q(s)}\\
&\leq &e^{ d_{\h^{p,q}}(x,o)+k_\G+1}+\frac{h(y_0)\sharp\G\setminus \G_+}{Q(s)}.
\end{eqnarray*}
Since $Q$ diverges at $\delta_{\h^{p,q}}(\G)$, we find an upper bound for $\mu_x^s(\overline{\h^{p,q}})$.

For the other inequality, first notice that we have:
\[\forall x\geq 0 ~\forall y>y_0+x~~ h(y-x)\geq e^{-x-1}h(y).\]
Now consider $\G_0=\{\g\in\G \vert d_{\h^{p,q}}(\g.o,o)>d_{\h^{p,q}}(o,x)+k_\G+y_0\}$. 
\begin{eqnarray*} \mu_x^s(\overline{\AdS^{n+1}})&=&\frac{\sum_{\g\in \G}    h(d_{\h^{p,q}}(\g o, x ))e^{-s d_{\h^{p,q}}(\g o,x)}}{Q(s)}\\
&\geq & \frac{\sum_{\g\in \G_0}    h(d_{\h^{p,q}}(\g o, o )-d_{\h^{p,q}}(x,o)-k_\G)e^{-s d_{\h^{p,q}}(\g o,x)}}{Q(s)}\\
&\geq & e^{-d_{\h^{p,q}}(x,o)-k_\G-1} \frac{\sum_{\g\in \G_0}    h(d_{\h^{p,q}}(\g o, o ))e^{-s d_{\h^{p,q}}(\g o,x)}}{Q(s)}\\
&\geq & e^{-d_{\h^{p,q}}(x,o)-k_\G-1} \left( 1- \frac{\sum_{\g\in \G\setminus\G_0}    h(d_{\h^{p,q}}(\g o, o ))e^{-s d_{\h^{p,q}}(\g o,x)}}{Q(s)} \right).
\end{eqnarray*}
Since $\G\setminus \G_0$ is finite and $Q(s)$ diverges at $\delta_{\h^{p,q}}(\gamma)$, we can find a positive lower bound independent on $s$.
\end{proof}

We can now consider a weak limit $\mu_x$ of $\mu_x^s$ as   $s\tv \delta_{\h^{p,q}}(\Gamma)$. }
The usual remark also stands: there could a priori exist different weak limits but we will see that it is in fact unique and will be called the Patterson-Sullivan density.

\begin{prop}\label{pr - patterson est une conformal density}
$(\mu_x)_{x\in C(\Lambda)}$ is a conformal density of dimension $\delta_{\h^{p,q}}(\Gamma)$. 
\end{prop}

\begin{proof}
Let us show the invariance. Let $g\in\G$. 
\begin{eqnarray*}
g^*\mu^s_x (E) &= &\mu^s_x(g^{-1} E)  \\
						&=& \frac{\sum_{\g\in \G}    h(d_{\h^{p,q}}(\g o,x ))e^{-s d_{\h^{p,q}}(\g o,x)} \Delta_{\g o}(g^{-1} E)}{Q(s)}\\
						&=& \frac{\sum_{\g\in \G}    h(d_{\h^{p,q}}(\g o, x ))e^{-s d_{\h^{p,q}}(\g o,x)} \Delta_{g\g o}( E)}{Q(s)}\\%\gamma'=g\gamma  
						&=& \frac{\sum_{\g'\in \G}    h(d_{\h^{p,q}}( \g' o, g x ))e^{-s d_{\h^{p,q}}(\g' o,g x)} \Delta_{\g' o}( E)}{Q(s)}\\
						&=& \mu^s_{g x} (E).
\end{eqnarray*}
The limit as   $s\tv \delta_{\h^{p,q}}(\Gamma)$ gives the invariance of $\mu_x$ by $\G$. 

Let $\epsilon>0$. 
Let $N(\xi)\subset \overline{\h^{p,q}}$ be a neighborhood of $\xi \in \Lambda$, such that $| \beta_\xi(x,y) - (d_{\h^{p,q}}(z,x)-d_{\h^{p,q}}(z,y) )| \leq \epsilon$ for $z \in N(\xi)$.
We have 
\begin{eqnarray*}
\mu^s_x(N_\xi) &=& \frac{1}{Q(s)} \sum_{\g o \in N(\xi)} h(d_{\h^{p,q}}(\g o , x) ) e^{-sd_{\h^{p,q}}(\g o , x) } \\
   						&\leq & \frac{1}{Q(s)} e^{s\epsilon} e^{-s\beta_\xi(x,y)} \sum_{\g o \in N(\xi)} h(d_{\h^{p,q}}(\g o , x) ) e^{-sd_{\h^{p,q}}(\g o , y)  } 
\end{eqnarray*}
The Patterson function $h$ is increasing and $d_{\h^{p,q}}(\g o , x) \leq d_{\h^{p,q}}(\g o , y)  +k_\G$, hence
%$$h(d(\g o , x) ) \leq h(d(\g o , y)  +K).$$

\begin{eqnarray*}
\mu^s_x(N_\xi) &\leq & \frac{1}{Q(s)} e^{s\epsilon} e^{-s\beta_\xi(x,y)} \sum_{\g o \in N(\xi)} h(d_{\h^{p,q}}(\g o , y)+k_\G ) e^{-sd_{\h^{p,q}}(\g o , y)  } .
\end{eqnarray*}
By the second property of the Patterson function and the fact that $d_{\h^{p,q}}(\g o, x) \tv \infty$ as $\g o \tv \xi$ we have : 
\begin{eqnarray*}
\mu^s_x(N_\xi) &\leq & \frac{1}{Q(s)} e^{s\epsilon} e^{-s\beta_\xi(x,y)}e^{\epsilon k_\G} \sum_{\g o \in N(\xi)} h(d_{\h^{p,q}}(\g o , y) ) e^{-sd_{\h^{p,q}}(\g o , y)  } \\
&\leq & e^{s\epsilon} e^{\epsilon k_\G} e^{-s\beta_\xi(x,y)} \mu^s_y(N_\xi).
\end{eqnarray*}
By letting $\epsilon \tv 0$ and    $s\tv \delta_{\h^{p,q}}(\Gamma)$ we get {\color{black} that $\mu_x$ is absolutely continuous with respect to $\mu_y$, and the following bound on the density:}
$$\frac{d \mu_x}{d\mu_{y}} (\xi) \leq e^{-\delta_{\h^{p,q}}(\G)\beta_\xi (x,y)}. $$
The same computation, switching the role of $x$ and $y$ gives the other inequality. Hence we obtain the quasi-conformal relation 
\begin{equation}\label{eq - quasi -confor}
\frac{d \mu_x}{d\mu_{y}} (\xi) =e^{-\delta_{\h^{p,q}}(\G)\beta_\xi (x,y)}.
\end{equation}

Since $Q(s)\tv \infty $ as $s\tv \delta_{\h^{p,q}}(\G)$, the support of $\mu_x$ is included in $\Lambda$. Moreover from the quasi-conformal relation $\mu_x$ and $\mu_{y}$ have the same support. Recall that the action of $\G$ on $\Lambda$ is minimal, therefore thanks to the invariance by $\G$, this implies that the support is {\color{black} either empty or equal to $\Lambda$. By Lemma \ref{lemme - masse totale bornee}, it cannot be empty.
}

\end{proof}

\subsection{Ergodicity}\label{sec - ergodicity }

{\color{black}{We will now prove the ergodicity of conformal densities. 

The proof uses the existence of Lebesgue density points,  which is a particular case of Lebesgue differentiation theorem. The latter relies on the Hardy-Littlewood inequality and some abstract measure theoretic results. 
Therefore, in order to show the existence of density points in the pseudo-Riemannian context, it is sufficient to obtain a Hardy-Littlewood type inequality  for pseudo-Riemannian balls in the boundary.

Let $\phi\in L^1(\nu_x)$ and define the maximal function associated to $\phi$ by 
$$\phi^* (\xi) =\lim_{\epsilon\tv 0 } \sup \frac{1}{\nu_x(B_x(\xi,\epsilon) )} \int_{ B_x(\xi,\epsilon)} \phi d \nu_x.$$
\begin{lemme}[Hardy-Littlewood maximal inequality]
There exist $C>0$ such that for all $\phi$ and all $\alpha>0$
$$\nu_x(\{\phi^* >\alpha\} ) \leq \frac{C}{\alpha}\|\phi\|_{\nu_0}.$$
\end{lemme}
In Federer \cite{Federer}, it is shown that Hardy-Littlewood inequality is true for separable metric space with doubling measure.  Theorem \ref{measure_balls},  shows that the measure $\nu_x$ is doubling. Using Vitali's covering lemma \ref{Vitali_distance} for balls in the boundary,  we can generalize the usual metric proof to our context. As we said this implies, using only measure theoretic arguments: 
\begin{lemme}[Lebesgue differentiation theorem]
For all $f\in L^1(\nu_x)$ for $\nu_x$ almost all $\xi$
$$\lim_{n\tv \infty} \frac{1}{\nu_x( \mathcal S_R(x,g_n x) )} \int_{\mathcal S_R(x,g_n x)} fd\nu_x=f(\xi),$$
as $d(x,g_n x)\tv \infty$ and $\xi \in \mathcal S_R(x,g_n x).$
\end{lemme}
}}

%We will also need the following lemma, showing that big shadows are uniformly of almost total mass: 
%\begin{lemme}\label{lem - les grosses ombres prennent toutes la masse}
%$\forall x\in C(\Lambda), \, \forall \epsilon>0, \, \exists R>0, \,  \forall y\in C(\Lambda),$
%$$\nu_x(\mathcal S_R(y,x) ) \geq \nu_x(\Lambda) -\epsilon.$$ 
%\end{lemme}
%\begin{proof}
%Let $x \in C(\Lambda)$ and $\epsilon>0$. 
%From Lemma \ref{shadow_2}, there exists $R>0$ sufficiently large such that for all $\xi,\eta \in \mathcal S_R(y,x)$ we have $d_x(\xi,\eta) \leq \epsilon$. Hence $\Lambda\setminus \mathcal S_R(y,x) \subset B_x(\xi, \epsilon)$. 
%This implies by Lemma 1.6 that 
%\begin{eqnarray*}
%\nu_x (\Lambda\setminus \mathcal S_R(y,x)  ) &\leq & \nu_x(B_x(\xi, \epsilon)) \\
%\nu_x(\Lambda) - \nu_x( \mathcal S_R(y,x) )	&\leq &\epsilon\\  
%\nu_x(\mathcal S_R(y,x) ) &\geq &\nu_x(\Lambda) - \epsilon. 
%\end{eqnarray*}
%\end{proof}

\begin{theorem}
A conformal density $(\nu_x)_{x\in C(\Lambda)}$ is ergodic.
\end{theorem}

\begin{proof}
Let $(\nu_x)_{x\in C(\Lambda)}$ be a conformal density, and fix some $x\in C(\Lambda)$. \\
Let $A$  be a $\G$-invariant subset of $\Lambda$. Suppose that $\nu_x(A) >0$ and let $\xi\in A$ be a density point and $\g_n x$ be a radial sequence converging to $\xi$, such that
$$\lim_{n\tv \infty} \frac{\nu_x ( \mathcal S_R(x,\g_n x) \cap A)}{\nu_x( \mathcal S_R(x,\g_n x) )} =1.$$
Remark that for any Borelian set $E\subset \Lambda$ and element $\g\in \G$ we have 
\begin{eqnarray*}
\nu_x(\g^{-1} E)  
							&=& \nu_{\g x} (E)\\
							&=& \int_{E} e^{-\beta_\xi( \g x , x ) } d\nu_x(\xi).
\end{eqnarray*}
Applying this to $\mathcal S_R(\g_n^{-1} x, x) \cap A = \g_n^{-1} \Big(\mathcal S_R( x, \g_n x) \cap A\Big)$ we have 

\begin{eqnarray*}
\nu_x( \mathcal S_R(\g_n^{-1} x, x) \cap A) &=& \int_{\mathcal S_R( x ,\g_n  x) \cap A } e^{-\beta_\xi (\g_n x, x) }d\nu_x(\xi). 
\end{eqnarray*}

It follows from Lemma \ref{shadow_1} that there exists $C>0$ such that the following inequalities hold for all $\xi \in \mathcal S_R( x ,\g_n  x) $:
$$ d_{\h^{p,q}}(x,\g_n x)-C\leq -\beta_\xi (\g_n x, x)  = \beta_\xi (x,\g_n x) \leq  d_{\h^{p,q}}(x,\g_n x)+C .$$
Hence 
\begin{eqnarray*}
\frac{\nu_x (\mathcal S_R(  \g_n^{-1} x, x) \cap A ) }{\nu_x (\mathcal S_R(  \g_n^{-1} x, x) )}  &=& 1 - \frac{\int_{\mathcal S_R(x, \g_n x)\cap A^c } e^{-\beta_\xi (\g_n x,x) }d\nu_x(\xi) }{ \int_{\mathcal S_R(x,  \g_n  x) } e^{-\beta_\xi (\g_n x,x) }d\nu_x(\xi)} \\
																															&\geq & 1- \frac{e^C e^{d_{\h^{p,q}} (\g_n x,x) }}{e^{-C} e^{d_{\h^{p,q}}(\g_n x,x)}} \frac{ \int_{\mathcal S_R( x ,\g_n  x) \cap A^c } d\nu_x(\xi) }{ \int_{\mathcal S_R( x ,\g_n  x) } d\nu_x(\xi) } \\
																															&\geq & 1- K  \frac{ \int_{\mathcal S_R( x ,\g_n  x)\cap A^c } d\nu_x(\xi) }{ \int_{\mathcal S_R( x ,\g_n  x) } d\nu_x(\xi) } \\
																																&\geq & 1- K  \frac{\nu_x(\mathcal S_R( x ,\g_n  x) \cap A^c) }{\nu_x(\mathcal S_R( x ,\g_n  x) ) }.
\end{eqnarray*}

Since $\xi $ is a density point for all $\varepsilon>0$, there exists $n$ sufficiently large such that 
\begin{eqnarray*}
\frac{\nu_x(\mathcal S_R( x ,\g_n  x) \cap A^c) }{\nu_x(\mathcal S_R( x ,\g_n  x) )} \leq \varepsilon.
\end{eqnarray*}
It follows that 
\begin{eqnarray}\label{eq - 1.35}
\frac{\nu_x (\mathcal S_R(  \g_n^{-1} x, x)\cap A ) }{\nu_x (\mathcal S_R(  \g_n^{-1} x, x) )}  \geq 1-K\varepsilon.
\end{eqnarray}
Finally, from Corollary \ref{coro - grande ombre grande masse} there exists $R$ sufficiently large such that for all $\g_n$ 
$$\nu_x (\mathcal S_R(  \g_n^{-1} x, x) )\geq \nu_x (\Lambda) -\varepsilon.$$ 
Putting everything together we have 
\begin{eqnarray*}
\nu_x(A) &\geq & \nu_x (\mathcal S_R(  \g_n^{-1} x, x)\cap A) \\
				&\geq & (1-K\varepsilon) \nu_x( \mathcal S_R(  \g_n^{-1} x, x)) \quad \text{From } (\text{Eq}.\ref{eq - 1.35})\\
				&\geq &  (1-K\varepsilon) (\nu_x ( \Lambda) -\varepsilon). 
\end{eqnarray*}
Since $\varepsilon$ is arbitrary, we have $\nu_x(A) = \nu_x(\Lambda)$, this shows the ergodicity of $\nu$. 
\end{proof}

\begin{prop} \label{uniqueness_conformal_density}
The Patterson-Sullivan density is the only conformal density up to a multiplicative constant.
\end{prop}
\begin{proof}
Let $\mu$ be the Patterson-Sullivan density, and let $\nu$ be another conformal density. We know from Corollary \ref{s_equals_delta}, we know that the dimension of $\nu$ is $\delta_{\h^{p,q}}(\G)$. It follows that $\mu+\nu$ is also a conformal density of dimension $\delta_{\h^{p,q}}(\G)$. The measure $\mu_x$ is absolutely continuous with respect to $\mu_x+\nu_x$, so they differ by a density function. Since these measures are both conformal measures, this function is invariant under $\Gamma$. It follows from the ergodicity that this function is constant $\mu_x+\nu_x$-almost everywhere (the constant being independent of $x$), so $\mu$ and $\mu+\nu$ are proportional, hence the proportionality between $\mu$ and $\nu$.
\end{proof}

~\\
\footnotesize \textsc{}\\
 \emph{E-mail address:}  \verb|olivier.glrx@gmail.com|
~\\
\footnotesize \textsc{Université du Luxembourg, Campus Kirchberg, 6, rue Richard Coudenhove-Kalergi, L-1359 Luxembourg}\\
 \emph{E-mail address:}  \verb|daniel.monclair@math.u-psud.fr|
 ~\\
 \footnotesize \textsc{Institut de Mathématique d'Orsay, Bâtiment 307, Université Paris-Sud, F-91405 Orsay Cedex, France}\\

\end{document}